\documentclass[11pt]{article}

\usepackage[ngerman, english,  USenglish]{babel}

\usepackage{amsmath, amsthm, pifont,amsfonts,amssymb,latexsym}
\usepackage{dsfont}
\usepackage{enumerate}
\theoremstyle{plain}
\newtheorem{theorem}{Theorem}[section]

\newtheorem{lemma}[theorem]{Lemma}
\newtheorem{proposition}[theorem]{Proposition}
\newtheorem{corollary}[theorem]{Corollary}

\theoremstyle{definition}
\newtheorem{definition}[theorem]{Definition}

\newtheorem*{property}{Property}

\theoremstyle{remark}
\newtheorem{remark}[theorem]{Remark}

\newcommand{\IP}{\mathbb P}

\newcommand{\IE}{\mathbb E}
\newcommand{\RR}{\mathbb R}

\newcommand{\ZZ}{\mathbb Z}
\newcommand{\IN}{\mathbb N}
\newcommand{\NN}{\mathbb N}
\newcommand{\1}{\mathds{1}}
\newcommand{\la}{\langle}
\newcommand{\ra}{\rangle}

\newcommand{\supp}{\textup{supp}\,}
\newcommand{\eps}{\varepsilon}

\newcommand{\Ctem}{C_{\textup{tem}}}

\newcommand{\Ito}{It\^o}

\newcommand{\dimension}{q}
\newcommand{\Rdim}{\RR^{\dimension}}
\newcommand{\intRd}{\int}

\newcommand{\be}{\begin{equation}} %
\newcommand{\eq}{\end{equation}}

\newcommand{\nn}{\nonumber}

\begin{document}
\date{\today}
\title{New results on pathwise uniqueness for the heat equation with colored noise}

\author{Thomas Rippl and Anja Sturm\\
Institute for Mathematical Stochastics\\
Georg-August-Universit\"at G\"ottingen \\
Goldschmidtstr. 7\\
37077 G\"ottingen, Germany}

\maketitle
\abstract{
We consider strong uniqueness and thus also existence of strong solutions for  
the stochastic heat equation with  a multiplicative colored noise term. Here, the noise 
is white in time and colored in $q$ dimensional space ($q \geq 1$) with a singular correlation kernel. 
The noise coefficient is H\"older continuous in the solution. We discuss improvements 
of the sufficient conditions obtained in Mytnik, Perkins and Sturm (2006) that relate the 
H\"older coefficient with the singularity of the correlation kernel of the noise. For this we use
new ideas of Mytnik and Perkins (2011) who treat the case of strong uniqueness for the 
stochastic heat equation with multiplicative white noise in one dimension. Our main result on 
pathwise uniqueness confirms a conjecture that was put forward in their paper.\\
}

  \vspace{\fill}


          \thispagestyle{empty}

\par
         




\newpage

\section{Introduction}
This work  is the third in a series of papers dealing with the pathwise uniqueness of the stochastic heat equation with H\"older continuous noise coefficients: For $t >0$ and  $x \in \Rdim$ we set $X(0,x) = X_0(x)$ and consider
\begin{equation} \label{eq:spde} \frac{\partial X}{\partial t} = \frac{1}{2} \Delta X + \sigma (t,x,X) \dot{W}(t,x) + b(t,x,X)  \text{  a.s.} \end{equation}
Here, $X: \RR_+ \times \Rdim \rightarrow \RR$ is random, $\Delta$ denotes the Laplacian, $\dot{W}$ a space-time noise on $\RR_+ \times \Rdim,$ and $\sigma$ and $b$ are real valued functions.

Stochastic partial differential equations (SPDE)  of  the form (\ref{eq:spde}) arise naturally in the description of the densities of measure-valued processes on $\Rdim,$ that are  obtained, for one, as diffusion limits of spatial branching particle systems. For example, in the case of super-Brownian motion 
in dimension $\dimension = 1$ the measure at any positive time $t>0$ 
has a density $X_t(x)=X(t,x)$ a.s., and this density satisfies the above equation \eqref{eq:spde} with $\sigma (t,x,X) = \sqrt{X},$ $b \equiv 0$ and $\dot{W}$ space-time white noise (\cite{ks:88}, \cite{reimers:89}). 

Here, we want to focus on equation (\ref{eq:spde}) in any dimension $\dimension \geq 1$ in the case when the
noise coefficient $\sigma$ is  not necessarily Lipschitz but merely H\"older continuous in the solution $X$ and 
 $\dot{W}$ is a noise that is white in time and colored in space.  This means that $W$ is a Gaussian martingale measure on $\RR_+ \times \RR^\dimension$  as introduced in \cite{jW} 
 with spatial correlation kernel $k:\RR^{2\dimension} \to \RR$ specified as follows. For  $\phi \in C_c (\RR^\dimension),$ the continuous compactly supported functions on $\RR^{\dimension},$ the real-valued process $(W_t(\phi))_{t\geq 0}$ is a Brownian motion with quadratic variation given by 
\be \label{eq:corrkernel} \langle W(\phi) \rangle_t:= t \int_{\RR^{\dimension}} \int_{\RR^{\dimension}} \phi(x) \phi(y) k(x,y)\, dx dy. \eq
SPDEs with colored noise of this form arise as diffusion limits of branching particle systems in a random environment, whose spatial correlation is described by the kernel $k$, in the case that $\sigma(t,x,X) = X,$ see \cite{aS03} 
and also \cite{lM96}. 
More general noise coefficients $\sigma$ should correspond to an additional dependence of the branching on the local particle density, see \cite{hZ10} for a recent general formulation in the non-spatial setting without a random environment. 

In this article we give conditions for pathwise uniqueness of solutions to equation (\ref{eq:spde}) with the correlation kernel $k$ in the following form: 
There exist constants $\alpha \in (0, 2 \wedge \dimension)$ and $c_{\ref{kernel}}>0$ such that
\be \label{kernel}
k(w,z) \leq c_{\ref{kernel}} (|w-z|^{-\alpha} + 1)  \quad \text{for all } w,z \in \RR^\dimension.
\eq
For noise correlation kernels of this form, existence and pathwise uniqueness of solutions to (\ref{eq:spde}) when $\sigma$ is H\"older continuous in the solution was previously considered in \cite{mps:06}, where
an equivalent formulations of condition (\ref{kernel}) can be found as well as further conditions that any correlation kernel as in (\ref{eq:corrkernel}) must satisfy. The techniques used in \cite{mps:06} for finding sufficient conditions on pathwise uniqueness were further refined in \cite{mp:11} albeit for (\ref{eq:spde}) in dimension $\dimension =1$ with space-time white noise. In this work, we want to utilize the ideas of \cite{mp:11} in order to improve the results of \cite{mps:06}. 

In order to rigorously describe our new results as well as the preceding results of \cite{mps:06} and \cite{mp:11} we introduce some conditions on the coefficients as well as some notation. We will impose a growth condition and  a H\"older continuity condition on $\sigma$ as well as the standard Lipschitz condition on $b.$ So assume that there exists a constant $c_{\ref{growthcond}}$ such that for all $(t,x,X) \in \RR_+ \times\RR^{\dimension+1},$
\begin{align}
\label{growthcond}
|\sigma(t,x,X)| +|b(t,x,X)|\leq c_{\ref{growthcond}} (1 + |X|).
\end{align}
Furthermore, for some $\gamma \in (0,1)$ there are $A_1,A_2 >0$ and for all $T>0$ there is an $A_0(T)$
so that for all $t \in [0,T]$ and all $(x,X,X') \in \RR^{\dimension + 2},$
\begin{align}
\label{holder}
 | \sigma(t,x,X) - \sigma(t,x,X') | \leq A_0(T) e^{A_1|x|} (1+|X| + |X'|)^{A_2} |X- X'|^{\gamma},
\end{align}
and there is a $B>0$ such that  for all $(t,x,X,X')\in \RR_+ \times\RR^{\dimension +2},$
\begin{align}\label{bLip}
 & |b(t,x,X)-b(t,x,X')|\leq B|X-X'|.
\end{align}
Also, we denote by  $C_c, C_0, C_b$ the spaces of continuous functions with compact support, vanishing at infinity or bounded, respectively. By $C(E,F)$ we denote the continuous functions 
from $E$ to $F$ for some topological spaces $E$ and $F.$ If the function is $k$-times continuously differentiable for $k \in \IN \cup \{ \infty\}$ we write a superscript $k.$ We also write $B^\dimension (x,r)$ for the ball with center $x$ and radius $r$ in $\RR^\dimension.$ Throughout the paper we will use the convention that  constants denoted by $c_{i.j},c_{i}$ refer to their appearance in Lemma $i.j$ or Equation ($i$), respectively. We will denote generic constants by $C,$ which may change their values from line to line. Further dependence on parameters is indicated in brackets.
Finally, let $p_t(x) = (2\pi t)^{-\dimension/2} \exp(-\frac{|x|^2}{2t})$ be the $\dimension$-dimensional heat-kernel.

We say that $(X,W)$ is a (stochastically weak) solution if there exists a filtered probability space $(\Omega, \mathcal{F}, (\mathcal{F}_t)_{t\geq 0}, P)$ that supports a colored noise $W$ defined as in (\ref{eq:corrkernel}) and (\ref{kernel}) such that  $X$ and $W$ are adapted and the mild formulation of \eqref{eq:spde} holds, namely
\begin{align} \label{eq:mspde}
 X(t,x) =& \intRd p_t(x-y) X(0,y) dy + \int_0^t \intRd p_{t-s}(x-y) \sigma(s,x,X(s,y)) W(ds \, dx)\\ 
\nonumber
&+\int_0^t \intRd p_{t-s}(x-y) b(s,x,X(s,y)) dx \, ds
\end{align}
almost surely for all $t \geq 0$ and $\phi \in C_c (\RR^\dimension)$, where we used the abbreviation $\int$ for $\int_{\RR^{\dimension}}.$ (In the following the integration domain will always be assumed to be $\RR^{\dimension}$ if nothing else is specified.)  For more details about these so called \emph{mild} solutions and the existence of the stochastic integral with respect to $W$ see \cite{dalang:99}, for more about the notion of \emph{weak} solutions see \cite{jacod:80} Def.~5.2(a).
Define the space of tempered functions by
\begin{align*}
 \Ctem := \{ f \in C(\RR^\dimension, \RR) : \| f\|_\lambda  < \infty \ \forall \lambda > 0\}  \text{ , where } || f \|_{\lambda} &:= \sup_{x\in \RR^\dimension} |f(x)| e^{-\lambda |x|}.
\end{align*}
For the existence of solutions we state
\begin{theorem}\label{theorem:existence}
 Let $X_0 \in \Ctem$ and let $b, \sigma$ be continuous functions satisfying \eqref{growthcond}. Assume that \eqref{kernel} holds for some $\alpha \in (0, 2 \wedge \dimension).$ Then there exists a stochastically weak solution to \eqref{eq:spde} with sample paths in $C(\RR_+, \Ctem)$. Additionally, it holds that for all $T,\lambda,p>0,$
 \begin{equation} \label{LpXbnd}
 \IE(\sup_{0\leq t\leq T}\sup_{x\in\RR^{\dimension}}|X(t,x)|^p e^{-\lambda|x|})<\infty.
\end{equation}
\end{theorem}
This theorem is essentially Theorem 1.2 and Theorem 1.8 of \cite{mps:06} combined, except that we add a drift $b$ and allow space and time dependence of $b$ and $\sigma$. The full proof addressing these straightforward generalizations can be found in Chapter 8 of \cite{rippl:12}.

We say that \emph{pathwise uniqueness} for \eqref{eq:spde} holds if for any two solutions $X^1$ and $X^2 \in \Ctem$ on the same filtered probability space $(\Omega, \mathcal{F}, (\mathcal{F}_t)_{t\geq 0}, P)$ supporting a noise $W$ and with $X^1_0=X^2_0$ almost surely we have that $X^1(t,x) = X^2(t,x)$  for all $t \geq 0,x \in \Rdim$ almost surely. We are now in the position to state our main result regarding pathwise uniqueness of solutions to (\ref{eq:spde}):
\begin{theorem}\label{theorem:unique}
 Let $X_0 \in \Ctem$ and assume that $b,\sigma:\RR_+ \times \Rdim \times \RR \to \RR$ satisfy \eqref{growthcond}, \eqref{holder} and \eqref{bLip}. Assume that \eqref{kernel} holds for some $\alpha \in (0, 2 \wedge \dimension).$ Then pathwise uniqueness for solutions of \eqref{eq:spde} holds if
 $$ \alpha < 2(2\gamma - 1) .$$
\end{theorem}

Our main result improves the sufficient conditions for pathwise uniqueness given in  \cite{mps:06} in the same setting: There, it was shown 
that pathwise uniqueness holds if $\alpha < (2\gamma -1).$ Since it was known already then from \cite{dalang:99,PZ00}
 that for Lipschitz continuous noise coefficients $\sigma$ (corresponding to $\gamma=1$) pathwise uniqueness holds if $\alpha<  2 \wedge \dimension$ there was an obvious gap for $\gamma$ close to $1$ in dimensions $\dimension \geq 2.$ We close this gap with the present work. In addition, heuristic arguments can be made -in the Lipschitz as well as in our H\"older continuous case, see Section \ref{secmainres}- that the sufficient conditions for pathwise uniqueness cannot be further improved, so that we believe that they are indeed necessary  and the result of Theorem \ref{theorem:unique} sharp. 

We would like to point out that the statement of Theorem \ref{theorem:unique} was already conjectured in  \cite{mp:11}. In that article, pathwise uniqueness to \eqref{eq:spde} with white noise (formally $k=\delta,$ the delta measure) is considered in dimension $\dimension =1$ for $b,\sigma:\RR_+ \times \RR^2 \to \RR$ that satisfy \eqref{growthcond}, \eqref{holder} and \eqref{bLip}. By using and significantly improving the techniques of \cite{mps:06} it is shown in this setting that pathwise uniqueness holds for $  \gamma>\frac{3}{4}.$ Recently, it has been proven in \cite{mmp:12} that this result is sharp at least when solutions can be positive and negative, implying in particular that the white noise equation with $\gamma =\frac{1}{2}$ is not pathwise unique. 

The latter question had sparked a lot of interest over the last several decades since the corresponding equation -albeit with nonnegative solutions- describes the density of super-Brownian motion on one hand. On the other hand, it is well known that the corresponding non-spatial ordinary stochastic  differential equation with respect to Brownian motion is pathwise unique if and only if $\gamma \geq \frac{1}{2}.$ Finally, we note that  it has recently been shown in \cite{Xiong12} that a certain SPDE related to super-Brownian motion (different  from \eqref{eq:spde} as it regards a distribution function valued process)  is also pathwise unique.

In this paper, we use the refined techniques put forward in  \cite{mp:11} in order to arrive at our main result, Theorem \ref{theorem:unique}. A heuristic and proof outline for the rather technical and lengthy parts of the arguments will be given in Section \ref{secmainres}. Since in the following sections many arguments are analogous to those provided in  \cite{mp:11} we do not present those parts in complete detail but refer the interested reader to \cite{rippl:12}, where all calculations are carried out explicitly.

Here, we would like to emphasize that  the main differences and additional difficulties to \cite{mp:11} lie in the fact that we are considering a multi-dimensional setting and that we need to take care of correlations stemming from the kernel $k.$ Thus, numerous adjustments and some refinements to the results in \cite{mp:11} are necessary (see for example Lemma 6.8 and the accompanying remark).

At the end of this section, we want to stress the significance of pathwise uniqueness by pointing out that existence of weak solutions combined with pathwise uniqueness generally implies the existence of strong solutions. This is a classic result for ordinary stochastic differential equations (see Proposition 1 and Corollary 1 of \cite{YW:71}). For the more general setting of stochastic partial differential equations used here we appeal to recent results of \cite{kurtz:07} in order to obtain:

\begin{theorem}\label{strong}
 Assume that the assumptions of Theorem \ref{theorem:unique} and therefore also of Theorem \ref{theorem:existence} hold. Let $(\Omega, \mathcal{F}, (\mathcal{F}_t)_{t\geq 0}, P)$ be a filtered probability space with adapted colored noise $W$ and let  $X_0 \in \Ctem$ be $\mathcal{F}_0$-measurable. Then there exists a strong adapted solution $X$ to \eqref{eq:spde} with respect to the prescribed $X_0$ and $W.$
\end{theorem}

\begin{proof}
 We want to use the terminology of \cite{kurtz:07}. In order to apply the results we need to specify the space on which the noise $W$ can be realized. One can show, see Lemma 3.3.14 of \cite{rippl:12} for details, that the Sobolev space $H^{-q-1}(\Rdim)$ is an appropriate space, which is Polish. Now set $S_1 = C(\RR_+,H^{-q-1}(\Rdim))$ and $S_2 = C(\RR_+, \Ctem),$ which is the sample path space of the solutions, and formulate the SPDE \eqref{eq:spde} as in Example 3.9 of \cite{kurtz:07}. By Theorem \ref{theorem:existence} we know that there exist compatible solutions (see Lemma 3.2 of \cite{kurtz:07} for a compatibility criterion which is applicable for weak solutions) and by Theorem \ref{theorem:unique} we have pointwise uniqueness for compatible solutions. So we can apply Theorem 3.14 a) $\Rightarrow$ b) of \cite{kurtz:07}. More details can be found in the proof of Lemma 5.1.1 of \cite{rippl:12}.
\end{proof}

We want to conclude this section with a number of remarks regarding the H\"older continuity condition on $\sigma$ stated in (\ref{holder}): 
\begin{enumerate}
\item When \eqref{growthcond} holds, it suffices to assume \eqref{holder} for $|X-X'|\leq 1$.  Indeed, \eqref{holder} (with any $\gamma >0$) is immediate from \eqref{growthcond} for $|X-X'|\geq 1$ with $A_0(T)= 2 c_4$, $A_1=0$, $A_2=1$.
\item Condition \eqref{holder} implies the following local H\"older condition: For all $K>1$ there is an $L_K$ so that for all $t\in [0,K]$
 and  $x\in B^\dimension (0,K),X,X'\in[-K,K],$
\begin{equation}
\label{locholder}
 |\sigma(t,x,X)-\sigma(t,x,X')|\leq L_K|X-X'|^\gamma.
\end{equation}
\end{enumerate}

\section{Proof of Theorem \ref{theorem:unique}} \label{secmainres}
 The proof of Theorem \ref{theorem:unique} is inspired by the idea of Yamada and Watanabe \cite{YW:71} that was already used in \cite{mps:06} and \cite{mp:11}. We closely follow Section 2 in \cite{mp:11} as most of the ideas can be transferred from white to colored noise and also to the multi-dimensional setting.

Now consider Theorem \ref{theorem:unique} and assume its hypotheses throughout.
Let $X^{1}$ and $X^{2}$ be two solutions of \eqref{eq:spde} on $(\Omega, \mathcal{F}, (\mathcal{F}_t)_{t\geq 0}, P)$ with sample paths in $C(\RR_+,\Ctem)$ a.s., with the same initial condition, $X^{1}_0=X^{2}_0=X_0\in \Ctem$, and of course the same noise $W.$  We start by observing that $X^{i}$ for 
 $i=1,2$ satisfy the distributional form of~\eqref{eq:spde}: For $\Phi \in C_{c}^{\infty}(\RR)$ we have that
\begin{align}
\label{eq:wspde}
\nn \intRd X^i(t,x) \Phi(x)\, dx
= &\intRd X^i_{0}(x)  \Phi(x)\, dx
+\int_{0}^{t}\intRd X^i(s,x) \frac{1}{2} \Delta\Phi(x)\, dx ds \\
 &+\int_{0}^{t}\intRd \sigma(s,x,X^i(s,x))\Phi(x) W(ds\, dx)\\
\nn &+\int_{0}^{t}\intRd b(s,x,X^i(s,x))\Phi(x) \, dxds\quad\forall t\geq 0 \quad a.s.
\end{align}
In fact, for adapted processes with sample paths in $C(\RR_+,\Ctem)$, the mild formulation \eqref{eq:mspde} is equivalent to the distributional formulation (\ref{eq:wspde}) of solutions to~\eqref{eq:spde}, see page 1917 of  \cite{mps:06}. 
Let for any $K>1$
\begin{equation}\label{TKdef}
T_K=\inf\{s\geq 0:\sup_y(|X^1(s,y)| \vee |X^2(s,y)|)e^{-|y|}>K\}\wedge K
\end{equation}  
be a stopping time. Since $X^i \in C(\RR_+,\Ctem)$ we have $T_K\to \infty$ for $K \rightarrow \infty.$ Up to time $T_K$ condition \eqref{holder} implies that
\begin{align}
\label{holder'}
|\sigma(t,x,X)-\sigma(t,x,X')|\leq R_0e^{R_1|x|}|X-X'|^\gamma
\end{align}
for some $R_0, R_1>0.$ Thus, a stopping time argument allows us to prove Theorem \ref{theorem:unique} for $\sigma$ where \eqref{holder} is replaced by \eqref{holder'} (see the text after (2.30) in \cite{mp:11} for more on the sufficiency of this argument).

In order to apply an argument similar to that of Yamada and Watanabe we set for any $n\in \NN$ as in \cite{mp:11}
\[a_n=\exp\{-n(n+1)/2\},\]
fix a positive function $\psi_n \in C^\infty(\RR,\RR_+)$, such that  $\supp \psi_n \subset(a_n,a_{n-1}), \psi_n(x) \leq \frac{2}{nx}$ and
\[\int_{a_n}^{a_{n-1}} \psi_n(x) \, dx = 1.\]
As this function approximates a $\delta$-function at zero as $n \rightarrow \infty,$ we define
\begin{equation}\label{eq:phi}
 \phi_n(x) := \int_0^{|x|} dy  \int_0^y dz\,\psi_n(z), \quad x \in \RR,
\end{equation}
which then approximates the modulus. More precisely, we have
\begin{align}
 \label{phi_n-1}
& \phi_n(x) \to |x| \text{ uniformly in } x \in \RR,\\
  \label{phi_n-2}
 & |\phi_n'(x)| \leq 1 \text{ for all } x \in \RR\text{ and }\\
  \label{phi_n-3}
& |\phi_n ''(x)| \leq \frac{2}{nx} \text{ for all }x \neq 0.
\end{align}
Next we fix a point $x \in \RR^\dimension$ and $t_0 >0$ and a positive function $\Phi \in C_c^\infty(\RR^\dimension, \RR_+)$ such that  $\supp \Phi \subset B^\dimension (0,1)$ and $\int \Phi(y) dy =1$. Let $ \Phi_x^m(y) = m^\dimension \Phi(m(y-x))$ for $m>0$.

\noindent
Define the difference of the solutions
\[ u  := X^{1} - X^{2}\]
and note that we can write down an equation of the form (\ref{eq:wspde}) for $u.$
Let $\la \cdot, \cdot \ra$ denote the scalar product on $L^{2}(\RR^\dimension)$ and assume $t \in [0,t_0]$. We apply the \Ito-formula for the semimartingale $\la u_t(\cdot),\Phi_x^m(\cdot)\ra,$ which is the difference of the two semimartingales given in \eqref{eq:wspde}, with $\phi_n$ as in \eqref{eq:phi} in order to obtain
\begin{align*}
 \phi_n &( \la u_t, \Phi_x^m\ra ) \\
 & = \int_0^t \intRd \phi_n'(\la u_s, \Phi_x^m\ra) \left( \sigma(s,y,X^1(s,y)) - \sigma(s,y,X^2(s,y)) \right) \Phi_x^m(y) \, W(ds\, dy) \\
 & \quad + \int_0^t \phi_n'(\la u_s, \Phi_x^m \ra) \la u_s, \frac{1}{2} \Delta \Phi_x^m\ra \, ds \\
& \quad + \frac{1}{2} \int_0^t ds \intRd dw \intRd  dz\, \psi_n(| \la u_s, \Phi_x^m\ra | ) \Phi_x^m(w) \Phi_x^m(z) k(w,z) \\
 & \qquad \times \left( \sigma(s,w,X^1(s,w)) - \sigma(s,w,X^2(s,w)) \right) \left( \sigma(s,z,X^1(s,z)) - \sigma(s,z,X^2(s,z)) \right) \\
 & \quad + \int_0^t \intRd \phi_n'( \la u_s,\Phi_x^m \ra ) \left( b(s,y,X^1(s,y)) - b(s,y,X^2(s,y)) \right) \Phi_x^m(y) \, dy ds.
\end{align*}
We integrate this function of $x$ against another non-negative test
function $\Psi \in C^{\infty}_{c}([0,t_0]\times\RR^\dimension )$.  Choose $K_1\in\NN$ so large that for $\lambda=1$, 
\begin{equation}
\label{Gamma}
\| X_0\|_\lambda < K_1\text{ and }
\Gamma\equiv\{x: \exists s\leq t_0 \text{ with } \Psi_s(x)>0\}
\subset B^{\dimension}(0,K_1).
\end{equation}
 We then apply the classical and stochastic versions of Fubini's Theorem, see Theorem~2.6 of \cite{jW}. The expectation condition in Walsh's Theorem 2.6 may be realized by localization, using the stopping times $T_K$ for $K \rightarrow \infty.$  Arguing as
in the proof of Proposition II.5.7 of \cite{perkins:03} to handle the time
dependence in $\Psi$ we then obtain that for any $t\in [0,t_0],$
\begin{align}
 \nn \la \phi_n &( \la u_t, \Phi_\cdot^m\ra ), \Psi_t(\cdot) \ra \\
 \nn & = \int_0^t \intRd \la \phi_n'(\la u_s, \Phi_\cdot ^m\ra) \Phi_\cdot^m(y), \Psi_s \ra \left( \sigma(s,y,X^1(s,y)) - \sigma(s,y,X^2(s,y)) \right)  \, W(ds\, dy) \\
 \nn & \quad + \int_0^t \la \phi_n'(\la u_s, \Phi_\cdot^m \ra) \la u_s, \frac{1}{2} \Delta \Phi_\cdot^m\ra, \Psi_s \ra \, ds \\
 \nn & \quad + \frac{1}{2} \int_0^t ds \int_{\RR^{3\dimension}} dx   dw dz\, \Psi_s(x) \psi_n(| \la u_s, \Phi_x^m\ra | ) \Phi_x^m(w) \Phi_x^m(z) k(w,z) \\
 \nn & \qquad \times \left( \sigma(s,w,X^1(s,w)) - \sigma(s,w,X^2(s,w)) \right) \left( \sigma(s,z,X^1(s,z)) - \sigma(s,z,X^2(s,z)) \right) \\
\label{eq:Iparts} & \quad + \int_0^t \la \phi_n(\la u_s,  \Phi_\cdot^m\ra),\dot{\Psi}_s \ra \, ds \\
 \nn & \quad + \int_0^t \intRd \la \phi_n'( \la u_s,\Phi_\cdot^m\ra) \Phi_\cdot^m(y), \Psi_s \ra \left( b(s,y,X^1(s,y)) - b(s,y,X^2(s,y)) \right)  \, dy ds \\
 \nn &\equiv I_{1}^{m,n}(t) + I_{2}^{m,n}(t) + I_{3}^{m,n}(t)+I_{4}^{m,n}(t)+I_5^{m,n}(t).
\end{align}
Now set $m_n=a_{n-1}^{-1/2}=\exp\{(n-1)n/4\}$ for $n\in\NN.$ This choice of $m_n$ differs from that in \cite{mps:06} and is essential for the improvements that are made here to the results in \cite{mps:06}, in particular to their Lemma 4.3.

We quote essentially Lemma 2.2 from \cite{mps:06} (where $m_n$ is used for $m$) and add a last point treating $I_5^{m_n,n}(t)$:
\begin{lemma}\label{lem:mps:2.2} For any stopping time $T$ and constant $t\geq 0$ we
have:
\begin{enumerate}
\item \begin{equation}\label{eq:I_1}
\IE(I_1^{m_n,n}(t\wedge T))=0\text{ for all }n.
\end{equation}
\item \begin{equation}
\label{eq:limI_2}
\limsup_{n\rightarrow \infty} \IE( I_{2}^{m_n,n}(t \wedge T))
\leq  \IE \Big( \int_{0}^{t \wedge T} \int
|u(s,x)| \frac{1}{2} \Delta\Psi_s(x)\, dx ds \Big).
\end{equation}
\item
\begin{equation}
\label{eq:limI_4}
\lim_{n\rightarrow \infty} \IE(I_4^{m_n,n}(t\wedge
T))=\IE\Bigl(\int_0^{t\wedge T}|u(s,x)|\dot\Psi_s(x)\,ds\Bigr).
\end{equation}
\item
\begin{equation}
\label{eq:limI_5}
\lim_{n\rightarrow \infty} \IE(I_5^{m_n,n}(t\wedge T))\leq B \IE\Bigl(\int_0^{t\wedge T}| u(s,x)|\Psi_s(x)\,ds\Bigr) \text{ with } B \text{ as in }(\ref{bLip}).
\end{equation}
\end{enumerate}
\end{lemma}
\begin{proof}
The points (a), (b) and (c) are proven in Lemma 2.2 of \cite{mps:06}.
 We only need to show the last point (d), for which we follow (2.48) of \cite{mp:11}. Since $|\phi'_n(x)|\leq 1$ for all $x \in \RR^q$ by (\ref{phi_n-2}), \eqref{bLip} implies that for a stopping time $T$,
\begin{equation}\label{I5bound}
I_5^{m_n,n}(t\wedge T)\leq B\int^{t\wedge T}_0\int_{\RR^{2\dimension}} |u(s,y)|\Phi_x^{m_n}(y)\Psi_s(x)\, dydxds =: B\tilde I_5^n(t\wedge T).
\end{equation}
The integral over $y$ converges pointwise in $x$ and $s$ due to continuity. 
Using \eqref{LpXbnd} we can obtain an integrable bound for this integrand and Lebesgue's Dominated Convergence Theorem thus implies for $n\to \infty$, 
\begin{equation}\label{I5conv}
\tilde I_5^n(t\wedge T)\to \int _0^{t\wedge T}\int |u(s,x)|\Psi_s(x)\, dxds\text{ a.s. }
\end{equation}
and hence in $L^1$ since, again by \eqref{LpXbnd}, $(\tilde I_5^n(t))_{n\in\NN}$ is $L^2$-bounded.
\end{proof}
It will be $I_3^{m_{n+1}, n+1}$ which will mostly concern us for the rest of this work. In its integral definition we may assume $|x| \leq K_1$ by \eqref{Gamma} and so $|w| \vee |z| \leq K_1+1$. If $K \geq K_1$, $s\leq T_K$ and $|w| \leq K_1+1$ we have by \eqref{TKdef}
\[|X^i(s,w)|\leq K e^{|w|}\leq Ke^{(K_1+1)} =: K' \ \text{ for }i=1,2.\]  
Therefore \eqref{kernel}, \eqref{locholder} and the fact that $\psi_n(x) \leq \frac{2}{nx}\1\{a_{n}<x<a_{n-1}\}$ show that since $K' \geq K_1+1$ for all $t\in [0,t_0]$, 
\begin{align}
\label{I3bnd1} & I_3^{m_{n+1},n+1} (t\wedge T_{K})\\
\nn&\leq \frac{c_{\ref{kernel}}}{2} \int_0^{t\wedge T_{K}}\int_{\RR^{3\dimension}} 2(n+1)^{-1}|\la u_s,\Phi_x^{m_{n+1}}\ra|^{-1}\1\{a_{n+1}<|\la u_s,\Phi_x^{m_{n+1}}\ra|<a_n\}\\
\nn&\phantom{\leq \frac{1}{2}\int_0^{t\wedge T_{K}}} \times L_{K'}^2|u(s,w)|^\gamma |u(s,z)|^\gamma \Phi_x^{m_{n+1}}(w) \Phi_x^{m_{n+1}}(z) (|w-z|^{-\alpha} + 1) \Psi_s(x)\, dwdzdxds\\
\nn&\leq c_{\ref{kernel}} L^2_{K'}a_{n+1}^{-1}\int_0^{t\wedge T_{K}}\int_{\RR^{3\dimension}} \1\{a_{n+1}<|\la u_s,\Phi_x^{m_{n+1}}\ra|<a_n\} |u(s,w)|^\gamma |u(s,z)|^\gamma \\
\nn&\qquad \qquad  \quad \times \Phi_x^{m_{n+1}}(w) \Phi_x^{m_{n+1}}(z) (|w-z|^{-\alpha} + 1)\Psi_s(x) \, dwdzdxds.
\end{align}
We note that $a_{n+1}^{-1}=a_n^{-1-2/n}.$ Thus, as the quantity of interest we define
\begin{align}\label{Indef}
I^n(t)=a_n^{-1-2/n}\int_0^{t}\int_{\RR^{3\dimension}} & \1 \{|\la u_s,\Phi_x^{m_{n+1}}\ra|<a_n \} |u(s,w)|^\gamma |u(s,z)|^\gamma \\
\nn& \Phi_x^{m_{n+1}}(w) \Phi_x^{m_{n+1}}(z) (|w-z|^{-\alpha} + 1)\Psi_s(x) \, dwdzdxds.
\end{align}
\begin{proposition}\label{prop:2.1}
 Suppose $\{ U_{M,n,K} : M,n,K \in \NN, K \geq K_1 \}$ are $\mathcal{F}_t$-stopping times such that for each $K \in \NN^{\geq K_1}$,
 \begin{eqnarray*}
  &(H_1)  &\phantom{AAAAAAAAA}U_{M,n,K} \leq T_K \text{ for all } M,n \in \NN, \\
  & &\phantom{AAAAAAA}U_{M,n,K} \nearrow T_K \text{ as } M \to \infty \text{ for all } n \in \NN, \\
  & & \phantom{AAAAAAAAA}\lim_{M\to \infty} \sup_{n \in \NN} P(U_{M,n,K} < T_K ) = 0 , \\
  &&\text{and} \\
  &(H_2)& \phantom{AAAAAA}  \lim_{n\to \infty} \IE(I^n(t_0 \wedge U_{M,n,K} )) = 0 \text{ for all } M\in \NN,
 \end{eqnarray*}
 are satisfied.
 Then the conclusion of Theorem \ref{theorem:unique} holds.
\end{proposition}
The proof of this proposition is the same as the proof of Proposition 2.1 in \cite{mp:11}, here using Lemma \ref{lem:mps:2.2}. What one shows is that $(t,x) \mapsto \IE[u(t,x)]$  is a non-negative subsolution of the heat equation with Lipschitz drift started in $0.$ Hence, two solutions coincide pointwise and so by continuity of paths we have: $X^1=X^2.$ We omit the details and refer to the proof of Proposition 2.1 in \cite{mp:11}.

Observe that all that is left for the proof of our main result, Theorem \ref{theorem:unique}, is the construction of the stopping times $U_{M,n,K}$ and the verification of $(H_1)$ and $(H_2)$. As these steps are extremely long we want to give a heuristic explanation for the sufficiency of $\alpha < 2(2\gamma -1)$ leading to $(H_2)$ even if we will not yet discuss the construction of the stopping times, which is done in Section \ref{sec:6}.

\medskip

\noindent{\bf Notation.} For $t,t'\geq 0$ and $x,x'\in\Rdim$ let $d((t,x),(t',x'))=\sqrt{|t'-t|}+|x'-x|$, where $|\cdot|$ always denotes the Euclidean norm on the corresponding space.  

\medskip

Note that the indicator function in the definition of $I^n$ in ({\ref{Indef})  implies that there is an $\hat x_0\in B^\dimension (x,\sqrt a_n)$ such that $|u(s,\hat x_0)|\le a_n$.  If we could take $\hat x_0=w =z$ we could bound $I^n(t)$ by $C(t)a_n^{-1-2/n-\alpha/2+2\gamma}$ using that for $C=C(\dimension)$
\begin{equation}
\label{intPhi}
 \int_{\RR^{2\dimension}} dwdz\,  \Phi_x^{m_{n+1}}(w)  \Phi_x^{m_{n+1}}(z) (|w-z|^{-\alpha} + 1) \leq C m_{n+1}^{\alpha},
 \end{equation}
 see page 1929 of \cite{mps:06}. Thus, $(H_1)$ and $(H_2)$ would follow immediately with $U_{M,n,K}=T_K$.  (The criticality of $\alpha < 2(2\gamma-1)$ in this argument is deceptive as it follows from our choice of $m_n$.)
Thus, in order to satisfy the hypotheses of Proposition~\ref{prop:2.1} we now turn to obtaining good bounds on $|u(s,w)-u(s,\hat x_0)|$ with $|\hat{x}_0-w| \leq 2 \sqrt{a_n}.$ The standard $1-\alpha/2-\eps$-H\"older modulus of $u$ (see Theorem 2.1 in ~\cite{ss:02}) will not give a sufficient result. In \cite{mps:06}, provided that $\alpha < 2\gamma-1,$ the H\"older modulus {\it near points where $u$ is small} was refined to $1-\eps$ for any $\eps>0$. More precisely, let
\begin{align*} Z(N,K)(\omega)&=\{(t,x)\in [0,T_K]\times B^\dimension (0,K): \text{ there is a }(\hat t_0,\hat x_0)\in [0,T_K]\times \Rdim \text{ such that }\\
&\phantom{AAAAAAAAAAAAAAA}\ |u(\hat t_0,\hat x_0)|\leq 2^{-N} \text{ and }d((\hat t_0,\hat x_0),(t,x))\leq 2^{-N}\}.
\end{align*}
Theorem 4.1 of  \cite{mps:06} then states (see Theorem 2.2 in \cite{mp:11} for the formulation used here):
\begin{theorem}\label{thm:collipmod} For each $K\in \NN$ and $0<\xi<\frac{1-\frac{\alpha}{2}}{1-\gamma} \wedge 1$ there is an $N_0=N_0(\xi,K,\omega)\in\NN$ a.s. such that for all natural numbers $N\geq N_0$ and all 
$(t,x)\in  Z(N,K)$, 
\[ d((t',x'),(t,x))\leq 2^{-N}\text{ and }t'\leq T_K \text{ implies }|u(t',x')-u(t,x)|\leq 2^{-N\xi}.\]
\end{theorem}
Theorem 4.1 of \cite{mps:06} is stated and proved for equation \eqref{eq:spde} without a drift. For the necessary changes to include the drift we refer to Section 9.9 in \cite{rippl:12}.

We now argue how this locally improved H\"older regularity can be used.  As already mentioned after \eqref{intPhi} the choice of $m_n$ is crucial. It is related to the locally improved H\"older regularity and so for the moment set $m_n=a_{n-1}^{-\lambda_0}$ for some $\lambda_0>0.$ We will take the liberty to use the approximation $m_n \approx  a_n^{-\lambda_0}$ in the following heuristic argument. Then for $(H_2)$ it suffices to show
\begin{align}\label{Intozero}
I^n(t)&\approx a_n^{-1}\int_0^t\int_{\RR^{3\dimension}} \1\{|\langle u_s,\Phi_x^{m_{n+1}}\rangle |<a_n\}|u(s,w)|^\gamma |u(s,z)|^\gamma \\
\nonumber&\phantom{a_n^{-1}\int_0^t\int A}\Phi_x^{m_{n+1}}(w) \Phi_x^{m_{n+1}}(z) (|w-z|^{-\alpha} + 1)\Psi_s(x)\, dw dz dx ds
\to 0\text{ as n }\to \infty.
\end{align}
For $x$ fixed, the point $\hat{x}_0$ mentioned before \eqref{intPhi} will now lie in $B^\dimension (x,m_n^{-1})$ and on the other hand only those $w$ and $z$ with $|w-x|\vee |z-x| \leq m_n^{-1}$ will appear in the integral \eqref{Intozero}.  So $w,z \in B^\dimension (\hat{x}_0, 2 a_n^{\lambda_0}).$
 Theorem~\ref{thm:collipmod} implies that for $\alpha < 2\gamma -1$ 
 \begin{equation}\label{uliponZ}u(t,\cdot) \text{ is $\xi$-H\"older continuous near its zero set for $\xi<1$},
 \end{equation}
 which allows us to bound $|u(s,w)-u(s,\hat x_0)|$ by  $(2a_n^{\lambda_0})^{\xi }$, and therefore $|u(s,w)|$ by $a_n + 2a_n^{\lambda_0 \xi}$ which in turn is bounded by $ 3a_n^{\lambda_0 \xi}$ if $\lambda_0\le 1.$ We can use this and \eqref{intPhi} in \eqref{Intozero}  to bound $I^n(t)$ for $0<\lambda_0\leq 1$ by a constant times the following
 \begin{align*}
 a_n^{-1-\alpha \lambda_0}\int_0^t \int a_n^{\lambda_0\xi 2\gamma} \Psi_s(x)\, dxds
 \le t a_n^{-1+\lambda_0(\xi2\gamma-\alpha)}
 \to 0\text{ as }n\to\infty,
 \end{align*}
if $2\gamma-1> \alpha$ and we choose $\lambda_0,\xi$ close to one. This was just the result in \cite{mps:06}. However, in Theorem \ref{thm:collipmod} the restriction by $1$ in the condition $\xi < \frac{1-\frac{\alpha}{2}}{1-\gamma} \wedge 1$ seems unnatural and not optimal.
 
To obtain an improved result we need to extend the range of $\xi$ beyond 1. We will obtain a statement close to the following one:
\begin{equation}\label{eq:uC2-} \nabla u(s,\cdot)\text{ is }\xi\text{-H\"older on }\{x: u(s,x)\approx \nabla (s,x)\approx 0\} \text{ for }\xi <1,
\end{equation}
where $\nabla u$ denotes the spatial derivative (in a loose sense as $u$ is not differentiable).  Actually, we cannot really write down \eqref{eq:uC2-} formally, but some statements come close to it, e.g.~Corollary~\ref{cor:5.9} for $m=\bar m+1$.

At this point we would like to note that a similar argument as in \cite{mp:11} shows that, using the techniques for $\alpha > 2(2\gamma -1)$, we will not be able to improve \eqref{eq:uC2-} to
\begin{equation}\label{eq:uC2}
u(s,\cdot)\text{ is 
 $C^2$ on }\{x:u(s,x)\approx \nabla u(s,x)\approx 0\}.
\end{equation}
So we can extend the range of $\xi$ up to $2-\eps$, but not beyond with this technique.

Assuming $\alpha < 2 (2\gamma-1)$ and \eqref{eq:uC2-}, we outline the idea of how we will be able to derive \eqref{Intozero}.
We  choose $0=\beta_0 < \beta_1 < \dots < \beta_L = \bar{\beta} < \infty$, a finite grid, and define
\begin{align*}
\nn \hat{I}_{n,i} (t)&:= a_n^{-1-\frac{2}{n}} \int_0^t\int_{\RR^{3\dimension}} \1_{\hat{J}_{n,i}(s)}(x) |u(s,w)|^\gamma |u(s,z)|^\gamma \\
 \nn
 &\phantom{AAAAAAA}\Phi_x^{m_{n+1}}(w) \Phi_x^{m_{n+1}}(z)(|w-z|^{-\alpha} + 1) \, dw dz\, \Psi_s(x)\, dx ds, \\
 \intertext{for all $i= 0, \dots, L$, where}
 & \nn \hat{J}_{n,i}(s) = \{x \in \Rdim: |\la u_s, \Phi_x^{m_{n+1}} \ra | < a_n, |\nabla u(s,x)| \in (a_n^{\beta_{i+1}}, a_n^{\beta_i} ] \}
 \intertext{for $i <L$ and for $i=0$,}
 & \nn \hat{J}_{n,0}(s) = \{x \in \Rdim: |\la u_s, \Phi_x^{m_{n+1}} \ra | < a_n, |\nabla u(s,x)| > a_n^{\beta_1} \}
 \intertext{and for $i=L$,}
& \nn \hat{J}_{n,L}(s) = \{ x \in \Rdim: |\la u_s, \Phi_x^{m_{n+1}} \ra | < a_n, |\nabla u(s,x)| \in [0, a_n^{\beta_L}] \}.
\end{align*}
Since 
\begin{equation}
\label{eq:In_isum}
I^n(t) = \sum_{i=0}^L \hat{I}_{n,i}(t),
\end{equation}
 our goal of proving $I^n(t) \to 0$ will be attained, if we can show that
\begin{align}\label{eq:In_i_tozero}
 \hat{I}_{n,i}(t) \to 0 \qquad \text{ for all } i= 0, \dots, L.
\end{align}
For a grid of $\beta_i$ fine enough we will be able to replace the condition that the absolute value of the gradient is contained in $(a_n^{\beta_{i+1}}, a_n^{\beta_i} ]$ in the definition of $\hat{J}_{n,i}(s)$ by the condition that it is approximately equal to $a_n^{\beta_i}$ for $i=1,\dots, L.$
Note that due to the boundedness of the support of $\Phi^n_x$, for $x \in \hat{J}_{n,i}(s)$ there must be $\hat{x}_n(s) \in B^\dimension (x,a_n^{\lambda_0})$ such that ~$|u(s,\hat{x}_n(s))|< a_n$. By \eqref{eq:uC2} we have for $w \in B^\dimension (x,a_n^{\lambda_0})$ and $[\hat{x}_n(s),w]$ the Euclidean geodesic between the two points: 
\begin{align}\label{eq:Holder:estim}
 \nn |u(s,w)|  &\leq a_n + \sup_{\tilde{w} \in [\hat{x}_n(s),w]} |\nabla u(s,\tilde{w})| \cdot |\hat{x}_n(s) - w| \\
 \nn & \leq a_n +  \sup_{\tilde{w} \in [\hat{x}_n(s),w]}(|\nabla u(s,x)| + |\tilde{w}-x|^\xi) |\hat{x}_n(s) - w| \\
 \nn & \leq a_n + (a_n^{\beta_i} + 2 a_n^{\lambda_0 \xi}) a_n^{\lambda_0} \\
  & \leq 7 (a_n \vee a_n^{\beta_i+ \frac{1}{2}}),
\end{align}
if we choose $\lambda_0 = \frac{1}{2}$, which is the smallest possible value for balancing the terms. Similarly, $\beta_i \leq \frac{1}{2}$ is optimal in \eqref{eq:Holder:estim}.
If we put this estimate into \eqref{eq:In_isum}, then we can bound $\hat{I}_{n,i}(t)$ by 
\[ a_n^{-1-\frac{2}{n}} 
 (a_n^{2\gamma} \vee a_n^{2 \gamma \beta_i+ \gamma}) \int_0^t\int \1_{\hat{J}_{n,i}(s)}(x) \Phi_x^{m_{n+1}}(w) \Phi_x^{m_{n+1}}(z) (|w-z|^{-\alpha} + 1) \Psi_s(x) \, dwdzdxds \]
and  \eqref{intPhi} leads to the bound
\begin{equation}
\label{eq:I_n,ibound}
a_n^{-1-\frac{2}{n}-\frac{\alpha}{2}}  (a_n^{2\gamma} \vee a_n^{2 \gamma \beta_i+ \gamma}) \int_0^t \int_{B^\dimension (0,K_1)} \1_{\hat{J}_{n,i}(s)}(x) \, dx ds,
\end{equation}
for some $K_1>0,$ since $\Psi$ is compactly supported. If $\beta_i$ is rather small, we find ourselves in the situation that the H\"older estimate \eqref{eq:Holder:estim} is not that strong. With a choice of $\lambda_0=1$ we would have gotten back to the case $\alpha < 2\gamma -1$, since small $\beta_i$ corresponds to neglecting the estimate on derivatives. However, particularly in that case we can give a good estimate on $|\hat{J}_{n,i}(s)|,$ the $\dimension$-dimensional Lebesgue measure of $\hat{J}_{n,i}(s)$.

But, let us first consider $\beta_L=\bar{\beta}$.
Then, by the estimate in \eqref{eq:I_n,ibound} we have
\[ \hat{I}_{n,L}(t) \leq C t (a_n^{2\gamma - 1-\frac{2}{n} -\alpha/2} \vee a_n^{(2\beta_L+1)\gamma - 1-\frac{2}{n} -\alpha/2}) \to 0\]
as $n\to \infty$ as long as we require $\beta_L=\bar{\beta} \geq 1/2.$ From this and the considerations just after \eqref{eq:Holder:estim}, we know that it should suffice to choose $\bar{\beta} = 1/2$, or more precisely, choosing $\bar{\beta}$ smaller will not lead to an optimal result, whereas $\bar{\beta} >\frac{1}{2}$ will not improve the result.

We still need to check the convergence for $i= 0, \dots, L-1$ and write in order to simplify notation $\beta=\beta_i$ and $J_n = \hat{J}_{n,i}(s)$. From \eqref{eq:uC2} we see that if $x\in J_n$, then  there is a direction $\sigma_x \in S^{\dimension -1}:=\{x \in \RR^q: |x|=1\}$ with $\sigma_x \cdot \nabla u(s,y)\geq \frac{1}{2}a_n^{\beta}$ if $|y-x|\le L a_n^{\beta/\xi}$ for an appropriate constant $L$ and $(y-x) \parallel \sigma_x,$ meaning that $(y-x)$ is parallel to $\sigma_x.$ \label{df:parallel} 
Assuming for the heuristic that $u(s,x)>-a_n$ (which we only know precisely for a point $\hat{x}_n(s) \in B^\dimension (x, a_n^{1/2})$ due to $|\la u_s, \Phi_x^{a_n^{1/2}}\ra|< a_n$) we obtain because of the positive gradient for $y \in x + \RR_+ \sigma_x$ by the Fundamental Theorem of Calculus:
\[u(s,y)>a_n \text{ if }4a_n^{1-\beta}< |y-x|\leq L a_n^{\beta/\xi}.\]
Similarly, one can also show (but we will not go into details here) that, by adapting $L$ appropriately, if $x,z \in J_n$ and $|x-z| \leq  L a_n^{\beta /\xi}$, we also have for $ z' \in z + \sigma_x \cdot [4a_n^{1-\beta};L a_n^{\beta/\xi}]$ that $u(s,z')>a_n$ and thus $z' \notin J_n.$ So for $x \in J_{n},$ denoting by $\{x+\sigma_x^{ortho}\}$ the plane through $x$ orthogonal to $\sigma_x,$ we have 
\begin{align}
 \nn |B^\dimension (x,La_n^{\beta/\xi}) \cap J_n| &\leq \int_{\{x+\sigma_x^{ortho}\} \cap B^\dimension (x,L a_n^{\beta/\xi})} dz \int_{-La_n^{\beta/\xi}}^{L a_n^{\beta/\xi}} dz' \1\{z +\sigma_x z' \in J_n \} \\
 &\leq C (a_n^{\beta/\xi})^{\dimension -1} a_n^{1-\beta}. \label{eq:heu:cover}
\end{align}
Covering the box $[-K_1,K_1]^\dimension$ with finitely many balls of radius $\frac{L}{2}a_n^{\beta/\xi}$ and using \eqref{eq:heu:cover} we obtain  $|J_n|\leq C(L,K_1)a_n^{1-\beta}a_n^{-\beta/\xi}$. We can use this in  \eqref{eq:I_n,ibound} to get
\begin{align}
\hat{I}_{n,i}(t) & \leq Ct a_n^{-1-\frac{2}{n}-\frac{\alpha}{2}+2\gamma (1\wedge (\beta_i + \frac{1}{2}))+1-\beta_i -\beta_i/\xi} 
\leq  Ct a_n^{-1-\frac{2}{n}-\frac{\alpha}{2}+2\gamma (1\wedge (\beta_i + \frac{1}{2}))+1-\beta_i(1+1/\xi)}
\label{In2bnd}
\end{align}
for all $ \beta_i \leq \bar{\beta}.$ The right hand side of  (\ref{In2bnd}) tends to zero  for all $\beta _i \leq \bar{\beta}$,
if $-\frac{\alpha}{2}+2\gamma(1\wedge (\beta_i + \frac{1}{2}))-2 \beta_i>0,$ for all $\beta _i \leq \bar{\beta}$,
i.e.~
\[ \gamma>\frac{1}{2}(1\wedge (\bar\beta + \frac{1}{2}))^{-1}(\frac{\alpha}{2}+2 \bar\beta) \]
since the right hand side is increasing as a function in $\bar{\beta}.$ Therefore, it attains its minimum value on the interval $[\frac{1}{2}, \infty)$ for $\bar{\beta}=1/2$  at  $\frac{1}{2}(1+\frac{\alpha}{2}).$ Then the estimate shows that: $\hat{I}_{n,i}(t)$ tends to zero for all $0 \leq \beta_i \leq \bar{\beta},$ if
\[ \alpha < 2(2\gamma -1). \]

This is what we wanted to show and ends the heuristic outline of the proof (some more details in the case of white noise can be found in Section 2 of \cite{mp:11}).

\begin{remark}
In the previous heuristics it suffices to consider one direction of the gradient. This will be sufficient to obtain uniqueness for $\alpha < 2(2\gamma -1)$ rigorously. However, it is tempting to include further information on the gradient, e.g.~$\nabla u \approx (a_n^{\beta^1},a_n^{\beta^2},\dots).$ We believe that no further improvement can be achieved, since \eqref{eq:Holder:estim} only requires the size of the principal component of the gradient.
\end{remark}

\section{Verification of the hypotheses of Proposition \ref{prop:2.1}}\label{sec:3}

In this section we make the heuristics of the previous section rigorous in the sense that we derive hypothesis $(H_2)$. This proof relies on the definition of sets similar to the ones defined before \eqref{eq:In_isum} and on Proposition \ref{prop:3.3}, whose proof is given in Section \ref{sec:6} and contains the verification of hypothesis $(H_1).$

We follow the arguments of Section 3 in \cite{mp:11} and will also restrict our attention to the case $b\equiv 0$ for notational convenience. All of the results can be extended to non-trivial $b$ satisfying the Lipschitz condition (\ref{bLip}), for more details we refer to Section 8 of \cite{mp:11} or Section 9.10 of \cite{rippl:12}. Otherwise, we assume the setting of the beginning of Section \ref{secmainres}. That means that $X^1$ and $X^2$ are two solutions of the SPDE \eqref{eq:spde} with the same noise $W$ and $u:= X^1-X^2$ is the difference of the two, i.e.
\begin{equation}
 u(t,x) = \int_0^t \intRd p_{t-s}(y-x) D(s,y) W(ds \, dy) \text{ a.s. for all } t \geq 0, x\in \Rdim,
\end{equation}
where
$D(s,y)= \sigma(s,y,X^1(s,y))-\sigma(s,y,X^2(s,y))$ which by \eqref{holder'} obeys
\begin{equation}\label{def:D}
 |D(s,y)| \leq R_0 e^{R_1|y|} |u(s,y)|^\gamma.
\end{equation}
Let $(P_t)_{t\geq0}$ be the heat-semigroup acting on $\Ctem$.
For $\delta \geq 0$ set
\begin{equation}\label{eq:u=u1+u2}
 u_{1,\delta} (t,x) := P_\delta (u((t-\delta)^+, \cdot))(x), \quad u_{2,\delta} := u -u_{1,\delta}.
\end{equation}
With the help of the Stochastic-Fubini-Formula (Theorem 2.6 in \cite{jW}, where localization with $T_K$ and \eqref{LpXbnd} are used for the condition on the expectation) reformulate that  for $\delta \leq t$ to
\begin{align*}
 u_{1,\delta}(t,x) =  \int_0^{(t-\delta)^+} \intRd p_{t -s}(y-x) D(s,y)\, W(ds\, dy).
\end{align*}
We define the following functions 
\begin{eqnarray}
\label{Gdeltadef}
G_\delta(s,t,x) &=& P_{(t-s)^+ + \delta}(u_{(s-\delta)^+})(x),\\ 
\label{Fdelta,ldef}
F_{\delta,l} (s,t,x) &=& -\partial_{x_l} G_\delta(s,t,x),  \quad 1\leq l\leq \dimension,
\end{eqnarray}
 for which we easily obtain $u_{1,\delta}(t,x) = G_\delta(t,t,x)$. We denote by $$p_{t,l}(x) = \partial_{x_l}p_t(x) , \, 1\leq l \leq \dimension$$ the spatial derivative of the heat-kernel. Then the following result holds, which is analogous to Lemma 3.1 in \cite{mp:11} and the lines preceding it and has essentially the same proof:
\begin{lemma}\label{lem:3.1}
 The random fields $G_\delta$  and $F_{\delta,l}$ are both jointly continuous in $(s,t,x) \in \RR_+^2 \times \RR^{\dimension}$ and
 \begin{align*} G_{\delta} (s,t,x) &= \int_0^{(s-\delta)^+} \int p_{(t\vee s)-r}(y-x)D(r,y)\, W(dr\,dy), \\
 F_{\delta,l}(s,t,x) &= \int_0^{(s-\delta)^+} \int p_{(t\vee s)-r,l}(y-x)D(r,y)\, W(dr\,dy) \text{, where } 1\leq l \leq \dimension.\end{align*}
Additionally, $u_{1,\delta}$ and $u_{2,\delta}$ are both $C(\RR_+,\Ctem)$-valued.
\end{lemma}
\noindent
Note that for the special choice of $s=t$ in the previous lemma we have that 
$$ \partial_{x_l} u_{1,\delta}(t,x) = -\int_0^{(t-\delta)^+} \int p_{t-r,l}(y-x)D(r,y)\, W(dr\,dy) =- F_{\delta,l}(t,t,x).$$
For $(t,x)\in \RR_+ \times \RR^{\dimension}$ and $n \in \NN$ let
$$ B_n(t,x) := \{ y \in \RR^{\dimension}: |y-x| \leq \sqrt{a_n},\, |u(t,y)| = \inf \{|u(t,z)|:|z-x| \leq \sqrt{a_n} \} \} $$
be the set of points with the smallest $u$-values in a certain neighborhood close to $x$ and let
$$ \hat{x}_n(t,x) $$
be a measurable choice of a point in $B_n(t,x)$ (e.g. with the smallest first coordinate, if this does not suffice to uniquely select a point, take the smallest second coordinate and so on).
Let us fix two positive but very small constants $\eps_0,\eps_1$ throughout the paper
\begin{equation}\label{eq:epsconditions}
 \eps_1 \in \left( 0, \frac{1}{32}(2(2\gamma-1)-\alpha) \right), \ \eps_0 \in \left( 0, \frac{1}{4}(1-\gamma) \eps_1 \right).
\end{equation}
Let $L  = L( \eps_0, \eps_1) = \lfloor \eps_0^{-1}(1/2 - 6 \eps_1) \rfloor \in \NN$ and set for $i= 0, \dots, L$
\begin{equation}\label{eq:3.10}
 \beta_i = i \eps_0 \in [0, \frac{1}{2}- 6 \eps_1], \quad \lambda_i = 2(\beta_i + \eps_1) \in [0,1]
\end{equation}
and $\beta_{L+1} = \frac{1}{2}- \eps_1.$ So alltogether for $i=0, \dots, L+1$:
\begin{equation}\label{eq:3.12}
 \beta_i \in [0,\frac{1}{2}-\eps_1]. 
\end{equation}
We define the following subsets of $\RR^{\dimension}$:
\begin{align*}
 J_{n,0}(s) := \{x \in \RR^{\dimension}: |x| \leq K_0,|\la u_s, \Phi_x^{m_{n+1}} \ra| \leq a_n, | \nabla u_{1,a_n} (s, \hat{x}_n (s,x)) | \geq \tfrac{a_n^{\eps_0}}{4} ] \},
\end{align*}
\begin{align*}
 J_{n,L}(s) := \{x \in \RR^{\dimension}: |x| \leq K_0,|\la u_s, \Phi_x^{m_{n+1}} \ra| \leq a_n, | \nabla u_{1,a_n} (s, \hat{x}_n (s,x)) | \leq \tfrac{a_n^{\beta_L}}{4} ] \},
\end{align*}
and for $i= 1 ,\dots, L-1$:
\begin{align*}
 J_{n,i}(s) := \{x \in \RR^{\dimension}: |x| \leq K_0,|\la u_s, \Phi_x^{m_{n+1}} \ra| \leq a_n, | \nabla u_{1,a_n} (s, \hat{x}_n (s,x)) | \in [\tfrac{a_n^{\beta_{i+1}}}{4}, \tfrac{a_n^{\beta_i}}{4} ] \}.
\end{align*}
Recall \eqref{Indef} and observe that for $t\geq0$, $n\in\NN$:
\begin{align}
 I^n(t) & \leq a_n^{-1-2/n} \sum_{i=0}^{L(\eps_0,\eps_1)} \int_0^t ds \int_{\RR^{3\dimension}} dxdwdz\, \1_{J_{n,i}(s)}(x) |u(s,w)|^\gamma |u(s,z)|^\gamma \nn \\
 & \qquad \qquad \qquad \times \Phi_x^{m_{n+1}}(w)  \Phi_x^{m_{n+1}}(z) (|w-z|^{-\alpha} + 1) \Psi_s(x) \nn \\
 & =: \sum_{i=0}^{L(\eps_0,\eps_1)} I_{n,i}(t). \label{eq:I_ni}
\end{align}
To verify the hypotheses of Proposition \ref{prop:2.1}, it suffices to show the existence of stopping times $U_{M,n,K}$ satisfying $(H_1)$ as well as for $i=0,\dots,L$,
\begin{equation*}
(H_{2,i})\phantom{AAAA} \lim_{n\to\infty}\IE(I_{n,i}(t_0\wedge U_{M,n,K}))=0\quad  \text{ for all }M,K \in\NN \text{ with } K \geq K_1. \quad \phantom{AAAAAAAAAAAAAA}
\end{equation*}
We will get to the definition of these stopping times in Section \ref{sec:6}.  We now define
\begin{align*}
\sigma_x := \sigma_x (n,s) := \nabla u_{1,a_n}(s,\hat{x}_n(s,x)) ( |\nabla u_{1,a_n}(s,\hat{x}_n(s,x)) |)^{-1}
\end{align*}
as the direction of the gradient $\nabla u_{1,a_n}$ at the point $\hat{x}_n(s,x)$ close to $x.$ We also set
\[ \bar{l}_n(\beta_i) = a_n^{\beta_i + 5 \eps_1},\]
where dependence on $\beta_i$ is not written out explicitly if there are no ambiguities.

To get $(H_{2,i})$ we need to derive some properties of points in $J_{n,i}$. Therefore, set
\begin{align*}
 \tilde{J}_{n,0}(s) := \{ x & \in \RR^{\dimension}: |x| \leq K_0, |\la u_s, \Phi_x^{m_{n+1}} \ra| \leq a_n, \\
 & \sigma_x \cdot \nabla u_{1,a_n^{\eps_0}}(s,x')  \geq a_n^{\eps_0}/16 \,  \text{ for all } x' \in \RR^{\dimension} \text{ s.t.~} |x'-x| \leq 5 \bar{l}_n(\beta_0) \\
 & \text{and } |u_{2, a_n^{\lambda_0}} (s,x') - u_{2, a_n^{\lambda_0}} (s,x'') | \leq 2^{-75} a_n^{\beta_{1}} (|x'-x''| \vee a_n^{\frac{2}{\alpha} (\gamma- \beta_{1} - \eps_1)} \vee a_n ) \\
 & \qquad \text{for all } x',x'' \in \RR^\dimension \text{ s.t.~} |x' -x | \leq 4 \sqrt{a_n}, |x'' -x' | \leq \bar{l}_n(\beta_0)\\
 & \text{and } |u(s,x')| \leq 3 a_n^{(1-\eps_0)/2} \text{ for all }x' \in \RR^\dimension \text{ s.t.~}   |x'  -x| \leq \sqrt{a_n} \},
\end{align*}
\begin{align*}
 \tilde{J}_{n,L}(s) := \{ &  x  \in \RR^{\dimension}: |x| \leq K_0, |\la u_s, \Phi_x^{m_{n+1}} \ra| \leq a_n, \\
 & |\nabla u_{1,a_n^{\lambda_L}}(s,x')| \leq a_n^{\beta_L}  \,  \text{ for all } x' \in \RR^{\dimension} \text{ s.t.~} |x'-x| \leq 5 \bar{l}_n(\beta_L) \\
 & \text{and } |u_{2, a_n^{\lambda_L}} (s,x') - u_{2, a_n^{\lambda_L}} (s,x'') | \leq 2^{-75} a_n^{\beta_{L+1}} (|x'-x''| \vee a_n^{\frac{2}{\alpha} (\gamma- \beta_{L+1}  - \eps_1)} \vee a_n ) \\
 & \qquad \text{for all }  x',x'' \in \RR^\dimension \text{ s.t.~} |x' -x | \leq 4 \sqrt{a_n}, |x'' -x' | \leq \bar{l}_n(\beta_L) \}
\end{align*}
and for $i= 1 ,\dots, L-1$:
\begin{align*}
 \tilde{J}_{n,i}(s) := & \{ x  \in \RR^{\dimension}: |x| \leq K_0, |\la u_s, \Phi_x^{m_{n+1}} \ra| \leq a_n, \\
 &  |\nabla u_{1,a_n^{\lambda_i}}(s,x')| \leq a_n^{\beta_i} \text{ and }\sigma_x \cdot \nabla u_{1,a_n^{\lambda_i}}(s,x') \geq a_n^{\beta_{i+1}}/16 \, \\
 & \phantom{AAAAAA} \text{ for all } x' \in \RR^{\dimension} \text{ s.t.~} |x'-x| \leq 5 \bar{l}_n(\beta_i) \\
 & \text{and } |u_{2, a_n^{\lambda_i}} (s,x') - u_{2, a_n^{\lambda_i}} (s,x'') | \leq 2^{-75} a_n^{\beta_{i+1}} (|x'-x''| \vee a_n^{\frac{2}{\alpha} (\gamma- \beta_{i+1}  - \eps_1)}  \vee a_n) \\
 & \qquad \quad \text{for all } x',x'' \in \RR^\dimension \text{ s.t.~} |x' -x | \leq 4 \sqrt{a_n}, |x'' -x' | \leq \bar{l}_n(\beta_i) \}.
\end{align*}
We also define two deterministic constants 
\begin{eqnarray*}
n_M(\eps_1) = \inf \{ n\geq1: a_n^{\eps_1} \leq 2^{-M-8} \}, \ n_0(\eps_0,\eps_1) = \sup \{ n\in \NN:\sqrt{a_n} < 2^{-a_n^{-\eps_0 \eps_1/4}} \}
\end{eqnarray*}
 and will from now on always assume that
\begin{equation}\label{eq:3.14}
 n > n_M(\eps_1) \vee n_0(\eps_0,\eps_1).
\end{equation}

The next proposition shows that we can ultimately estimate the size of the sets $ \tilde{J}_{n,i}(s)$ instead of that of $ J_{n,i}(s):$
\begin{proposition}\label{prop:3.3}
 $\tilde{J}_{n,i}(s)$ is a compact set for all $s \geq 0$, $i \in \{0, \dots, L\}.$ There exist stopping times $U_{M,n,K}$ satisfying $(H_1)$ such that for all $n \geq n_M$, $i \in \{0, \dots, L\},$ and $s \leq U_{M,n,K},$
 $$   J_{n,i}(s)  \subset \tilde{J}_{n,i}(s). $$
\end{proposition}
\noindent
The proof of this proposition can be found in Section \ref{sec:6}.
We will use this proposition to show $(H_{2,i})$ at the end of this section. We need the following notation for $ i \in \{0,\dots, L\}$:
$$ l_n(\beta_i) := (129 a_n^{1-\beta_{i+1}}) \vee a_n^{\frac{2}{\alpha}(\gamma-\beta_{i+1} - \eps_1)} ,$$
where we omit the dependence on $\beta_i$ if there are no ambiguities and obtain: 
\begin{lemma}\label{lem:3.6}
If $i\in\{0,\dots,L\}$ and $n > n_M(\eps_1)$, then 
\[l_n(\beta_i)<\sqrt{a_n}<\frac{1}{2}\bar{l}_n(\beta_i).\]
\end{lemma}

\begin{proof}
\begin{align*}
l_n(\beta_i)a_n^{-1/2}&=(129a_n^{\frac{1}{2}-\beta_{i+1}})\vee a_n^{\frac{2}{\alpha}(\gamma-\beta_{i+1} - \eps_1) - \frac{1}{2}}\\
&\le (129a_n^{\eps_1})\vee a_n^{\frac{1}{2\alpha}(4 \gamma - 4 \beta_i - 4 \eps_0 - \eps_1 - \alpha) }\\
&<1
\end{align*}
by \eqref{eq:epsconditions}, \eqref{eq:3.12} and because $a_n^{ \eps_1}<2^{-8}$ by \eqref{eq:3.14}. This gives the first inequality. For the second one, use $\beta_i\le \frac{1}{2}-6\eps_1$  and \eqref{eq:3.14} to see that
\begin{equation*}\sqrt{a_n}\, \bar{l}_n(\beta_i)^{-1}=a_n^{\frac{1}{2}-\beta_i-5\eps_1}\le a_n^{\eps_1}<1/2.
\end{equation*}
\end{proof}

We give some elementary properties of the sets $\tilde{J}_{n,i}(s)$.
\begin{lemma}\label{lem:3.4}
 Assume $s\geq0$, $i \in \{0, \dots, L\}, x \in \tilde{J}_{n,i}(s)$, $x'\in \RR^\dimension$ and $|x'-x| \leq 4 \sqrt{a_n}$.
 \begin{enumerate}
  \item  If $i >0$, $x'' \in \RR^{\dimension}$ s.t. $|x''-x'| \leq \bar{l}_n(\beta_i)$, then 
  $$|u(s,x'')-u(s,x')| \leq 2 a_n^{\beta_i} ( | x''-x'| \vee a_n^{\frac{2}{\alpha} (\gamma- \beta_{i+1}  - \eps_1)} \vee a_n).$$
  \item If $i<L$, $x''\in \RR^\dimension$ s.t.~$(x''-x') \parallel \sigma_x$ and $a_n^{\frac{2}{\alpha} (\gamma- \beta_{i+1}  - \eps_1)} \vee a_n \leq |x'-x''| \leq \bar{l}_n(\beta_i)$, then
  $$ u(s,x'')-u(s,x') \begin{cases}
                       \geq 2^{-5} a_n^{\beta_{i+1}} |x''-x'| \text{ if } (x''-x')\cdot \sigma_x \geq 0 , \\
                       \leq -2^{-5} a_n^{\beta_{i+1}} |x''-x'| \text{ if } (x''-x')\cdot \sigma_x < 0 .
                      \end{cases} $$
  \item Let $y \in \tilde{J}_{n,i}(s), |x-y| \leq \bar{l}_n(\beta_i) $. Additionally let $y', y'' \in \RR^{\dimension}$, s.t $|y-y'| \leq \sqrt{a_n},  (y''-y') \parallel \sigma_x$ and $|y''-y'| \in ( a_n^{\frac{2}{\alpha} (\gamma- \beta_{i+1}  - \eps_1)}\vee a_n, \bar{l}_n(\beta_i) ).$ Then
  $$ u(s,y'')-u(s,y') \begin{cases}
                       \geq 2^{-5} a_n^{\beta_{i+1}} |y''-y'|  \text{ if } (y''-y')\cdot \sigma_x \geq 0 , \\
                      \leq -2^{-5} a_n^{\beta_{i+1}} |y''-y'| \text{ if } (y''-y')\cdot \sigma_x < 0 .
                      \end{cases} $$
   \item If $i >0$, then for $|w-x| < \sqrt{a_n}, $
   $$ | u(s,w) | \leq 5 a_n^{\beta_i + 1/2}.$$
 \end{enumerate}
\end{lemma}

\begin{proof}
 To prove (a) let $n,i,s,x,x',x''$ be as above. Since
 $$ |x'-x| \vee |x''-x| \leq 4 \sqrt{a_n} \vee(4 \sqrt{a_n}+ \bar{l}_n (\beta_i) ) \leq 5 \bar{l}_n(\beta_i),$$
 the distance to $x$ of any point on the line between $x'$ and $x''$ is bounded from above by $5\bar{l}_n(\beta_i)$. By the Mean Value Theorem and the definition of $\tilde{J}_{n,i}(s),$ we get
 \begin{align*}
  | u(s,x'') - u(s,x') | & \leq | u_{1,a_n^{\lambda_i}}(s,x'') - u_{1,a_n^{\lambda_i}}(s,x') | + | u_{2,a_n^{\lambda_i}}(s,x'') - u_{2,a_n^{\lambda_i}}(s,x') | \\
  & \leq a_n^{\beta_{i}} |x''-x'| + 2^{-75}a_n^{\beta_{i+1}} (|x''-x'| \vee a_n^{\frac{2}{\alpha} (\gamma- \beta_{i+1}  - \eps_1)} \vee a_n ) \\
  & \leq 2 a_n^{\beta_i} (|x''-x'| \vee  a_n^{\frac{2}{\alpha} (\gamma- \beta_{i+1}  - \eps_1)} \vee a_n).
 \end{align*}
To prove (b) w.l.o.g.~consider $(x''-x')\cdot \sigma_x \geq 0$ and so estimate analogously to (a) (remember that $[\cdot,\cdot]$ denotes the Euclidean geodesic between two points in $\Rdim$):
\begin{align*}
 u(s,x'') - u(s,x') & \geq \inf_{ y \in [x',x'']} [ \nabla u_{1,a_n^{\lambda_i}}(s,y) \cdot (x''-x') ] - | u_{2,a_n^{\lambda_i}}(s,x'') - u_{2,a_n^{\lambda_i}}(s,x') | \\
  & \geq (a_n^{\beta_{i+1}}/16 ) \sigma_x \cdot (x''-x') - 2^{-75}a_n^{\beta_{i+1}} |x''-x'| \\
  & \geq (a_n^{\beta_{i+1}}/32 ) (x''-x')\cdot \sigma_x.
\end{align*}
Next, we prove (c) using that $|y'-x| \vee |y''-x| < \sqrt{a_n} + \bar{l}_n(\beta_i) + \bar{l}_n(\beta_i) \leq 5 \bar{l}_n(\beta_i).$
\begin{align*}
 u(s,y'') - u(s,y') & \geq \inf_{ z \in [y',y'']} [ \nabla u_{1,a_n^{\lambda_i}}(s,z) \cdot (y''-y') ] - | u_{2,a_n^{\lambda_i}}(s,y'') - u_{2,a_n^{\lambda_i}}(s,y') | \\
  & \geq (a_n^{\beta_{i+1}}/16 ) (y''-y')\cdot \sigma_x - 2^{-75}a_n^{\beta_{i+1}} |y''-y'| \\
  & \geq (a_n^{\beta_{i+1}}/32 ) (y''-y')\cdot \sigma_x,
\end{align*}
where, in the next to last inequality, we used that $x \in \tilde{J}_{n,i}(s)$ for the $\nabla u_{1,a_n^{\lambda_i}}$-part and $y \in \tilde{J}_{n,i}(s)$ for the $u_{2,a_n^{\lambda_i}}$-part.

Finally, prove (d) much in the same way as the previous claims: We have $|\hat{x}_n(s,x) - w|  < |\hat{x}_n(s,x) - x| + |x-w|  \leq 2\sqrt{a_n} \leq \bar{l}_n(\beta_i)$ by Lemma \ref{lem:3.6}.  So we can apply (a) for $x'=\hat{x}_n(s,x)$ and  $x'' =w$ to obtain
\begin{align*} | u(s,w) | &\leq |u(s,\hat{x}_n(s,x)) | + |u(s,\hat{x}_n(s,x)) - u(s,w)|  \\ 
&\leq a_n + 2 a_n^{\beta_i}(|w-\hat{x}_n(s,x) | \vee a_n^{\frac{2}{\alpha} (\gamma- \beta_{i+1}  - \eps_1)} \vee a_n) \\
&\leq 5a_n^{\beta_i + 1/2}\end{align*}
since  $a_n^{\frac{2}{\alpha}(\gamma-\beta_{i+1}+\eps_1)} < \sqrt{a_n},$ again by Lemma \ref{lem:3.6}.
\end{proof}
Now, define 
\[ F_n(s,x) := \la \Phi^{m_{n+1}}_x, u_s \ra = \int_{B^\dimension (0,\sqrt{a_n})} \Phi^{m_{n+1}} (z) u(s,x+z) \, dz. \]
The next lemma provides some conclusions that can be drawn about points that lie in $\tilde{J}_{n,i}(s)$ for $i \in \{0, \dots, L-1\}$ and $s \in \RR_+.$
\begin{lemma}\label{lem:3.5}
Assume $i \in \{ 0, \dots, L-1\}$, $s \in \RR_+$. 
\begin{enumerate}
 \item If $x \in \tilde{J}_{n,i}(s)$, $\tilde x \in \RR^{\dimension}$ with $ (\tilde{x}-x)\parallel \sigma_x $ and $l_n(\beta_i) <  | (\tilde{x}-x)\cdot \sigma_x | \leq \bar{l}_n(\beta_i)$, then
 $$ F_n(s,\tilde{x}) - F_n(s,x) \begin{cases} \geq 2^{-5} a_n^{\beta_{i+1}} |\tilde{x}-x| \text{  , if } (\tilde{x}-x)\cdot \sigma_x \geq 0,\\
                          \leq - 2^{-5} a_n^{\beta_{i+1}} |\tilde{x}-x| \text{  , if } (\tilde{x}-x)\cdot \sigma_x < 0.
                         \end{cases}
 $$
  \item If $x,y \in \tilde{J}_{n,i}(s)$, $|x-y| \leq \bar{l}_n(\beta_i)$. Then for $\tilde{y} \in \RR^{\dimension}$, such that  $(y-\tilde{y}) \parallel \sigma_x$ and $l_n(\beta_i) < |y-\tilde{y}| < \bar{l}_n(\beta_i)$ it holds that 
 $$ \tilde{y} \notin \tilde{J}_{n,i}(s).$$
 \item If $x \in \tilde{J}_{n,i}(s)$, $z\in \RR^{\dimension}$ and $|x-z| \leq \bar{l}_n(\beta_i)/2$, then
 $$\int_{(-\bar{l}_n/2, \bar{l}_n/2)} db\,  \1\{z  + \sigma_x b \in \tilde{J}_{n,i}(s) \cap B^\dimension (x, \bar{l}_n/2)\, \} \leq 2 l_n(\beta_i) .$$
\end{enumerate}
\end{lemma}
\begin{proof}
 For (a) assume $(\tilde x -x)\cdot \sigma_x \in [l_n(\beta_i), \bar{l}_n(\beta_i)]$. Then
 $$ F_n(s,\tilde x) - F_n(s,x) = \int_{B^{\dimension}(0,\sqrt{a_n})} \Phi^{m_{n+1}}(z) (u(s,\tilde x +z) - u(s,x+z)) \, dz.$$
 Clearly, $|z| \leq \sqrt{a_n}$ and for $x'' = \tilde{x} + z, x' = x +z$, we have
 $$ |x'-x| \leq \sqrt{a_n},\, (x''-x') = (\tilde{x} -x) \parallel \sigma_x,\, |x'-x''| \in [l_n(\beta_i), \bar{l}_n(\beta_i)].$$
 Therefore, we can apply Lemma \ref{lem:3.4} (b) in the case $(\tilde x - x) >0$ to obtain
 \begin{align} F_n(s,\tilde{x}) - F_n(s,x) & \geq \int_{B^\dimension(0,\sqrt{a_n})} \Phi^{m_{n+1}}(z) |\tilde{x} -x| 2^{-5} a_n^{\beta_{i+1}} \, dz \nonumber \\
  & \geq 2^{-5} a_n^{\beta_{i+1}} |\tilde{x} -x|.
 \end{align}
The same can be done in the case $(\tilde x -x)\cdot \sigma_x < 0$.

To show (b) use the same ideas as before, where Lemma \ref{lem:3.4} (b) is replaced by Lemma \ref{lem:3.4} (c), in order to deduce that
\begin{align*}
 |F_n(s,\tilde{y}) | & \geq |F_n(s,\tilde{y}) - F_n(s,y)| - |F_n(s,y)| \\
 & \geq 2^{-5} a_n^{\beta_{i+1}} l_n(\beta_i) - a_n \\
 &  \geq \frac{97}{32} a_n.
\end{align*}
Hence, since $|\la u_s, \phi_{\tilde{y}}^{m_{n+1}}\ra|> a_n$ it follows that $\tilde{y} \notin \tilde{J}_{n,i}(s).$

For (c) assume that $y = z + \sigma_x b \in \tilde{J}_{n,i}(s)$ for a certain $b \in [- \bar{l}_n/2,\bar{l}_n/2]$ (otherwise the integral is $0$ anyway). Observe that
$$ |y-x| \leq |x-z| + |b| \leq \bar{l}_n.$$
So, we can apply (b) for $x,y \in \tilde{J}_{n,i}(s)$ to obtain that
\begin{align*}
  \int_{(-\bar{l}_n/2, \bar{l}_n/2)} & db\,  \1\{z  + \sigma_x b \in \tilde{J}_{n,i}(s) \cap B^\dimension (x,\bar{l}_n/2 )\} \   
   \leq \int_{(-\bar{l}_n, \bar{l}_n)}  db\,  \1 \{y  + \sigma_x b \in \tilde{J}_{n,i}(s) \, \}
   \leq 2 l_n(\beta_i).
\end{align*}
\end{proof}

Let $\Sigma_x$ be a $\dimension \times (\dimension-1)$ dimensional matrix consisting of an orthonormal basis of the orthogonal space $\sigma_x^{ortho} = \{y \in \RR^\dimension: \sigma_x \cdot y = 0 \}$ and let $|A|$ denote the Lebesgue measure of a measurable set $A \subset \RR^{\dimension}$.

\begin{lemma}\label{lem:3.7}
 For $i \in \{0, \dots, L-1\}$ and $s \geq 0$, $n \in \NN$ there is a constant $c_{\ref{lem:3.7}} = c_{\ref{lem:3.7}}(\dimension)$ such that 
 $$ |\tilde{J}_{n,i}(s)| \leq c_{\ref{lem:3.7}} K_0^{\dimension} l_n(\beta_i) \bar{l}_n(\beta_i)^{-1}.$$
\end{lemma}

\begin{proof}
 Set $B_x = B^{\dimension} (x, \bar{l}_n(\beta_i)/4)$ and cover the compact set $\tilde{J}_{n,i}(s)$ with a finite number of these balls, say $B_{x'^1}, \dots, B_{x'^{Q'}}$. If $|x'^j-x'^k| \leq \bar{l}_n(\beta_i) /4$, then $B_{x'^j} \subset B^{\dimension} (x'^k, \bar{l}_n(\beta_i)/2)$. So, if we increase the radius of the balls around $x'^1, \dots, x'^{Q'}$ to $\bar{l}_n(\beta_i)/2$, it suffices to use those balls whose centers have at least distance $\bar{l}_n(\beta_i)/4,$  which we denote by $x^1, \dots, x^{Q}.$  If we consider $B^{\dimension}(x^k, \bar{l}_n(\beta_i)/8), k =1,\dots, Q$, then all of these balls are disjoint. Thus, we have
 \begin{equation}\label{eq:lem:3.6:pr:1} Q \leq  K_0^{\dimension} (\bar{l}_n(\beta_i)/8)^{-\dimension} \end{equation}
 and also
 \begin{equation}\label{eq:lem:3.6:pr:2} \tilde{J}_{n,i}(s) \subset \bigcup_{k=1}^{Q} B^{\dimension}(x^k, \bar{l}_n(\beta_i)/2) \cap \tilde{J}_{n,i}(s).\end{equation}
 Next we want to consider the Lebesgue measure of the sets on the right-hand-side using some kind of Cavalieri decomposition and Lemma \ref{lem:3.5} (c). Fix $k \in \{1,\dots, Q\}$ and denote by $C(\dimension)$  the volume of the $\dimension$-dimensional Euclidean ball. We have
 \begin{eqnarray*}
&&  |  B^{\dimension}(x^k,  \bar{l}_n(\beta_i)/2) \cap \tilde{J}_{n,i}(s) | \\
  & =& \int_{B^\dimension(0, \bar{l}_n/2)} dz' \, \1 \{x^k + z' \in \tilde{J}_{n,i}(s) \, \} \\
  & \leq& \int_{B^{\dimension-1}(0, \bar{l}_n/2)} dz' \int_{(-\bar{l}_n/2, \bar{l}_n/2)} db \,  \1 \{x^k + \Sigma_{x^k} z' + \sigma_{x^k} b \in \tilde{J}_{n,i}(s) \, \} \\
  & \leq& \int_{B^{\dimension-1}(0, \bar{l}_n/2)} dz' \, 2 l_n(\beta_i) \\
  & =& 2 C(\dimension-1) (\bar{l}_n(\beta_i)/2)^{\dimension-1} l_n (\beta_i).
 \end{eqnarray*}
Here,  we were able to apply Lemma \ref{lem:3.5}(c) in the last inequality with $z= x^k+\Sigma_{x^k} z'$
since $|x^k + \Sigma_{x^k} z' -x^k|  = |\Sigma_x z'| = |z'| \leq \bar{l}_n/2$. And therefore, by \eqref{eq:lem:3.6:pr:1} and \eqref{eq:lem:3.6:pr:2} for $c_{\ref{lem:3.7}} =4 \cdot 4^\dimension  C(\dimension -1)$ we obtain
$$ | \tilde{J}_{n,i}(s) | \leq c_{\ref{lem:3.7}} K_0^{\dimension} l_n(\beta_i) (\bar{l}_n(\beta_i))^{-1}.$$
\end{proof}
\smallskip
\noindent
We are now in the position to complete the \\

\noindent
\textbf{Verification of the Hypothesis $(H_2)$ in Proposition \ref{prop:2.1}.} 

Let $n > n_M(\eps_1) \vee n_0(\eps_0,\eps_1)$, $t >0$ and $M \in \NN$ fixed.\\
First, consider $i=0$. For $x \in J_{n,0} (s)$ and $|y-x| \leq \sqrt{a_n}$ we have $|u(s,y)| \leq 3 a_n^{(1-\eps_0)/2}$ due to Proposition \ref{prop:3.3}.  So, we obtain in \eqref{eq:I_ni} for $n$ large enough so that $\eps_1>\frac{2}{n}:$
\begin{align}
 I_0^n( & t_0 \wedge U_{M,n,K})  \leq a_n^{-1-2/n}3^{2\gamma} a_n^{\gamma(1-\eps_0)} \int_0^{t_0 \wedge U_{M,n,K}} ds \intRd dx\, \Psi_s(x) \1_{J_{n,0}(s)}(x) \\
 & \nn \qquad \phantom{AAAAAAAAAAAAAAA}\intRd dw \intRd dz\, \Phi_x^{m_{n+1}}(w) \Phi_x^{m_{n+1}}(z) (|w-z|^{-\alpha} +1) \\
 & \nn \leq C(\|\Psi\|_\infty, \|\Phi\|_\infty) a_n^{-1-\eps_1} a_n^{\gamma(1-\eps_0)} \\
 &  \nn \qquad \qquad \phantom{AAAAAAA} a_n^{-\alpha/2} t_0 K_0^{\dimension} l_n(\beta_0) \bar{l}_n(\beta_0)^{-1} \text{ (by \eqref{eq:3.14} and Lemma \ref{lem:3.7})} \\
 & \nn \leq C t_0 K_0^{\dimension} a_n^{-1-\alpha/2+\gamma(1-\eps_0)} (a_n^{1-\eps_0-6\eps_1} \vee a_n^{\frac{2}{\alpha}(\gamma-\eps_1-\eps_0)-6\eps_1}).
\end{align}
And this expression tends to zero as $n\to \infty$ since by \eqref{eq:epsconditions}
\begin{align*}
 -1-\alpha/2 +\gamma(1-\eps_0)+1-\eps_0-6\eps_1 & \geq \gamma -\alpha/2 -8\eps_1 >8 \eps_1 > 0
 \intertext{as well as }
 -1-\alpha/2 +\gamma(1-\eps_0)+\frac{2}{\alpha}(\gamma-\eps_1-\eps_0)-6\eps_1 & \geq \gamma -1-\alpha/2 +(\gamma - \eps_1 -\eps_0) - 7 \eps_1 \\
 & \geq 2\gamma - 1 -\alpha/2 - 9 \eps_1 > 23 \eps_1>0.
\end{align*}
Next, let $i \in \{ 1, \dots, L\}$ and assume $x \in \tilde{J}_{n,i}(s), y \in \RR^{\dimension}, |y-x| \leq \sqrt{a_n}.$ So, we can use Lemma \ref{lem:3.4} (d) to get that
\begin{align}
 | u(s,y)| &  \leq 5 a_n^{\beta_i + \frac{1}{2}}.
\end{align}
Put that into \eqref{eq:I_ni} for $y=w$ and $y = z$ to obtain that
\begin{align}
  I_i^n(t_0 \wedge U_{M,n,K}) & \leq 5^{2\gamma} a_n^{-1-2/n} a_n^{2\beta_i \gamma + \gamma} \int_0^t ds \int dx\, \1_{J_{n,i}(s)}(x) \Psi_s(x) \nn \\
 & \qquad \times \int_{\RR^{2\dimension}} dwdz\, \Phi_x^{m_{n+1}}(w)  \Phi_x^{m_{n+1}}(z) (|w-z|^{-\alpha} + 1).  \label{eq:pr:H_2:3}
\end{align}
To treat the integral in $w$ and $z$, we use \eqref{intPhi} leading to
\begin{align}
\label{I_isub}
 I_i^n(t_0 \wedge U_{M,n,K}) & \leq 25 C a_n^{-1-2/n} a_n^{2\beta_i \gamma +\gamma} m_{n+1}^{\alpha} \int_0^{t_0 \wedge U_{M,n,K}} \intRd  \, \1_{J_{n,i}(s)} (x) \Psi_s (x) \, dx  ds\\
 \nonumber 
 & \leq C(\alpha, \dimension, \|\Psi\|_\infty, \|\Phi\|_\infty)  a_n^{-1-\eps_1} a_n^{2\beta_i \gamma +\gamma} a_{n}^{-\frac{\alpha}{2}} \int_0^{t_0 \wedge U_{M,n,K}} | \tilde{J}_{n,i}(s) |\, ds. \end{align}
 Next, we use Lemma \ref{lem:3.7} in the case $i\in \{1, \dots, L-1\}$ and obtain
 \begin{align*} 
 I_i^n(t_0 \wedge U_{M,n,K}) & \leq C  a_{n}^{-1-\eps_1} a_n^{2\beta_i \gamma +\gamma} a_n^{-\frac{\alpha}{2}} t_0 C(\dimension) K_0^{\dimension} l_n(\beta_i) \bar{l}_n(\beta_i)^{-1} \\
 & \leq C t_0 a_n^{-1-\frac{\alpha}{2} +2\beta_i \gamma + \gamma - \eps_1} (a_n^{1-\beta_{i+1}} \vee a_n^{\frac{2}{\alpha}(\gamma-\beta_{i+1}-\eps_1)} )a_n^{-\beta_i -5 \eps_1} \\
 & = C t_0  (a_n^{-1-\frac{\alpha}{2} +2\beta_i \gamma + \gamma + 1-\beta_{i+1} -\beta_i -6 \eps_1} \\ 
 & \qquad \vee a_n^{-1-\frac{\alpha}{2} +2\beta_i \gamma + \gamma + \frac{2}{\alpha}(\gamma-\beta_{i+1}-\eps_1) -\beta_i -6 \eps_1} ) \\
 & =: C t_0 [a_n^{\rho_{1,i}} \vee a_n^{\rho_{2,i}}].
\end{align*}
Hence, it suffices to check for positivity of $\rho_{1,i}$ and $\rho_{2,i}$ to obtain the desired result.
\begin{align*}
 \rho_{1,i} & = -1-\frac{\alpha}{2} +2\beta_i \gamma + \gamma + 1-\beta_{i+1} -\beta_i -6 \eps_1 \\
 & = -\frac{\alpha}{2} +\gamma +2 \gamma \beta_i -2 \beta_i - 6 \eps_1 - \eps_0\\
 & \geq \frac{1}{2}(2(2\gamma-1)-\alpha) + 1 -\gamma -2\beta_i(1-\gamma) - 7 \eps_1\\
 & > 8 \eps_1 + (1-\gamma)(1-2\beta_i) - 7 \eps_1 > \eps_1 >0
\end{align*}
by \eqref{eq:epsconditions}. Additionally, note that by \eqref{eq:epsconditions},
\begin{align*}
 \frac{2}{\alpha} (\gamma - \beta_{i+1}-\eps_1 ) &= \frac{2\gamma -1}{\alpha} + \frac{1-2\beta_{i+1}-2\eps_1}{\alpha} \\
 & \geq \frac{1}{2} + \frac{1}{2}(1-2\beta_i-4\eps_1) \\
 & = 1-\beta_i-2\eps_1.
\end{align*}
So we can calculate
\begin{align*}
 \rho_{2,i} & = -1-\frac{\alpha}{2}+2\beta_i \gamma + \gamma + \frac{2}{\alpha} (\gamma - \beta_{i+1}-\eps_1 ) - \beta_i-6\eps_1 \\
 & \geq -1-\frac{\alpha}{2} + 2\beta_i \gamma +\gamma +1-\beta_i-2\eps_1 -\beta_i -6 \eps_1 \\
 & \geq 2\gamma - 1 -\frac{\alpha}{2} - 8 \eps_1 > \eps_1 >0
\end{align*}
since $\beta_i \leq \frac{1}{2}$ by \eqref{eq:epsconditions}.

To finish the proof, we note that in the case $i=L$ it suffices to use a trivial bound on the integral in (\ref{I_isub}) and we obtain with 
$\beta_L \geq \frac{1}{2}-6 \eps_1-\eps_0\geq \frac{1}{2}-7 \eps_1$ from \eqref{eq:epsconditions}:
\begin{align*}
 I_L^n(t_0 \wedge U_{M,n,K}) & \leq C  a_{n}^{-1-\eps_1} a_n^{2\beta_L \gamma +\gamma} a_n^{-\frac{\alpha}{2}} t_0 K_0^{\dimension} \\
 & \leq C a_n^{\gamma - 1 -\alpha/2 + 2 (\frac{1}{2}-7 \eps_1) \gamma - \eps_1} \\
 & = C a_n^{2\gamma -1-\alpha/2 -15 \eps_1} \leq C a_n^{\eps_1}.
\end{align*}
And so, we are done with the proof of Proposition \ref{prop:2.1}.\hfill $\Box$

\section{Heat kernel estimates}\label{sec:4}
This section will be concerned with estimates for the heat kernel in $\RR^{\dimension}$ defined by
 \[ p_t(x) = (2\pi t)^{-\frac{\dimension}{2}} \exp \left( - \frac{|x|^2}{2t} \right), \]
 and its derivative in space
\[ p_{t,l}(x) := \partial_{x_l} p_t(x) = - \frac{x_l}{t} p_t (x), \, 1\leq l \leq \dimension\]
 for $x\in \RR^{\dimension},\, t>0$.  There are already a number of results in Section 5 of \cite{mps:06} regarding bounds on heat kernels, in particular when they are connected by a correlation kernel and also in Section 4 of \cite{mp:11}
 regarding the derivatives of heat kernels.
Here, we will combine the techniques used for those results in order to obtain bounds on  integrals of the derivatives $p_{t,l}$ that are connected by a correlation kernel related to colored noise. All of the proofs are put into the appendix. As necessary we will highlight the dependence of constants $C$ on various quantities.

This first simple lemma will be used frequently later on:
\begin{lemma}\label{lem:algebra}
 Let $0<r_0\leq r_1.$ Then there is a constant $C=C(r_0,r_1) >0$ such that  for all $r \in [r_0,r_1]$ and $a \geq 0, u\geq 1,$ 
  \begin{equation} 
  \label{eq:ineq:1} 
  a \leq C  u^{1/r}  \exp ( \frac{a^r}{u}) \leq C  u^{1/r_0}  \exp ( \frac{a^r}{u}) . \end{equation}
\end{lemma}
A trivial consequence is the following Lemma 4.2 in \cite{mp:11}:
\begin{lemma}\label{lem:4.2}
For the heat kernel in $\RR^{\dimension}$ there is a constant $C>0$ 
such that for $l=1,\dots,\dimension, t>0, x \in \Rdim,$  \[|p_{t,l}(x) | \leq C \frac{1}{\sqrt{t}} p_{2t}(x).\]
\end{lemma}
The next lemma is about the integral over distances of heat kernel derivatives:
\begin{lemma}\label{lem:heat-ker:diff:2}
For $\alpha \in (0,\dimension), K\geq 0$, there is a positive constant $C= C(\alpha, \dimension, K) < \infty$ such that for any $x,x' \in \RR^{\dimension}$, $0< t \leq t' \leq K,$
\begin{align*}
    \intRd \intRd & | (p_{t,l}(w-x) - p_{t',l}(w-x'))(p_{t,l}(z-x) - p_{t',l}(z-x')) | \ (|w-z|^{-\alpha} + 1) \, dw dz \\
      &\leq C t^{-1-\alpha/2} \left( 1 \wedge  \frac{|x-x'|^2 + |t-t'|}{t}   \right).
\end{align*}
\end{lemma}
A simple extension of Lemma 5.1 in \cite{mps:06} is the following lemma:
\begin{lemma}\label{lem:mps:5.1:add}
 For $0<t\leq t'$, $0 \leq r_1,r_2, r_3 \leq R$, there is a constant $C=C(R)$ such that
 \begin{align*} \intRd \intRd p_t(w) p_{t'}(z) |w|^{r_1} |z|^{r_2} e^{r_3(|w| + |z|)} (|w-z|^{-\alpha} +1 ) \, dw dz \\
  \leq C e^{2 r_3^2 t'} t^{r_1/2} t'^{r_2/2} (t^{-\alpha/2} +1 ),
 \end{align*}
and there is a constant $C= C(K,R)$ such that for $x,y\in [-K,K]^{\dimension}$:
\begin{align*} \intRd \intRd p_t(x-w) p_{t'}(y-z) |w|^{r_1} |z|^{r_2} e^{r_3(|w| + |z|)} (|w-z|^{-\alpha} +1 ) \, dw dz \\
  \leq C e^{2 r_3^2 t'} (t^{r_1/2}+1) (t'^{r_2/2}+1) (t^{-\alpha/2} +1 ).
 \end{align*}
\end{lemma}
Using the two previous lemmas we can obtain a result on integrals ``outside'' a certain area:
\begin{lemma}\label{lem:4.4}
 For all $R>2$, $K>0$, there is a constant $C=C(K,R)$ such that for all $0 \leq p,r\leq R, \eta_0, \eta_1 \in (1/R, 1/2)$, $l=1,\dots,\dimension$, $0\leq s < t\leq t' < K$ and $x,x' \in [-K,K]^{\dimension}$:
 \begin{equation}\label{eq:lem:4.4}
  \begin{split}
   \intRd \intRd & |w-x|^p |z-x|^p ( p_{t-s,l}(w-x) - p_{t'-s,l}(w-x') ) ( p_{t-s,l}(z-x) - p_{t'-s,l}(z-x') ) \\
   & \times \1 \{ |w-x| > (t'-s)^{1/2-\eta_0}\vee 2 |x-x'| \}  \, e^{r|w-x| + r |z-x|} (|w-z|^{-\alpha} + 1) \, dw dz\\
   & \leq C (t-s)^{-1-\alpha/2} \exp( - \eta_1 (t'-s)^{-2\eta_0} /256 ) \left[1\wedge \left( \frac{|x-x'|^2+ |t-t'|}{t-s} \right) \right]^{1-\eta_1/2}.
  \end{split}
 \end{equation}
\end{lemma}

\section{Local bounds on the difference of two solutions}\label{sec:localbounds}
In this section we present the extension of Theorem \ref{thm:collipmod}, i.e. the results showing (in some sense) ``H\"older-continuity of order 2''. This section is very similar in its ideas to Section 5 of \cite{mp:11}. Hence, we do not give all of the proofs but can refer the interested reader  to Section 9.4 of \cite{rippl:12} for the details.
First, let us recall that for $n \in \NN,$
\[
a_n =\exp(-n(n+1)/2)
\]
and for $(t,x), (t',x') \in \RR_+ \times \Rdim,$
\[ d((t,x),(t',x')) = \sqrt{|t-t'|} + |x-x'|. 
\]
Define for $N,K,n \in \NN$, $\beta \in [0,1/2]$ the random set
\begin{equation}\begin{split}
 Z(N,n,K,\beta)(\omega) = \{& (t,x) \in [0,T_K] \times [-K,K]^{\dimension} \subset \RR_+ \times \RR^{\dimension}: \text{ there is a } \\
 &\quad (\hat{t}_0,\hat{x}_0) \in [0,T_K]\times \RR \text{ such that } d((t,x),(\hat{t}_0,\hat{x}_0))\leq 2^{-N},\\
 &\quad 
  |u(\hat{t}_0,\hat{x}_0) | \leq a_n \wedge (\sqrt{a_n}2^{-N}), \text{ and } |\nabla u_{1,a_n}(\hat{t}_0,\hat{x}_0)| \leq a_n^\beta \}.
\end{split} \end{equation}
For $\beta=0$ define $Z(N,n,K,0) = Z(N,n,K)$ as above, but with the condition on $\nabla u_{1,a_n}$ omitted. 
Note that $(t,x) \in Z(N,n,K,\beta)$ always implies $t \leq K$.
For $ \gamma  < 1 $ define recursively $\gamma_0=1$ and
  \begin{equation}
  \gamma_{m+1} = \gamma \gamma_m + 1 - \frac{\alpha}{2}.
  \end{equation}
This gives the explicit formula
  \begin{equation}
  \gamma_m = 1 + \frac{(\gamma - \alpha/2)(1-\gamma^m)}{1-\gamma}.
  \end{equation}
Since  $\alpha< 2(2\gamma -1)$ we have that $\gamma_m$ is increasing to  $\gamma_\infty = \frac{1-\alpha/2}{1-\gamma} > 2$ for $m\rightarrow \infty.$ So there will be an $\bar{m}\in \NN$ such that  $\gamma_{\bar{m}+1} > 2 \geq \gamma_{\bar{m}}$. Set $\tilde{\gamma}_m := \gamma_m \wedge 2, 0 \leq m \leq \bar{m}+1.$

\begin{definition}
 A collection of $[0,\infty]$-valued random variables $\{N(\alpha):\alpha \in A\}$ will be called stochastically bounded uniformly in $\alpha$, iff
 $$ \lim_{M\to\infty} \sup_{\alpha \in A} \IP[N(\alpha)\geq M] =0.$$
\end{definition}
\noindent For $m \in \ZZ_+$, we will let $(P_m)$ denote the following property:
\begin{property}[$P_m$]

 \hangindent2em
  \hangafter=0
  For any $n\in \NN, \xi,\eps_0 \in (0,1),K\in \NN^{\geq K_1}$ and $ \beta \in [0,1/2]$, there is an $N_1(\omega) = N_1(m,n,\xi,\eps_0, K,\beta)$ in $\NN$ a.s. such that for all $N\geq N_1$, if $(t,x)\in Z(N,n,K,\beta)$, $t'\leq T_K$ and $d((t,x),(t',x')) \leq 2^{-N}$, then
  \begin{equation}
  \label{P_mprop}
  |u(t',x')| \leq a_n^{-\eps_0} 2^{-N\xi} [(\sqrt{a_n}\vee 2^{-N})^{\tilde{\gamma}_m -1} + a_n^{\beta} \1 \{m>0\} ].
  \end{equation}
  Moreover, $N_1$ is stochastically bounded uniformly in $(n,\beta)$. 

\end{property}

\begin{proposition}\label{prop:5.1}
  Property $(P_m)$ holds for any $m \leq \bar{m}+1$.
\end{proposition}
The induction start is proved as in Proposition 5.1 of\cite{mp:11} using Theorem \ref{thm:collipmod}, here, instead of their Lemma 2.3, so we omit it.
The induction step from $(P_m)$ to $(P_{m+1})$ is a bit more technical and needs some preparation. It will be completed at the end of this section on page \pageref{pr:prop:5.1:part2}.

To get there we first write down a lemma which tells us what we can get out of Property $(P_m)$:
\begin{lemma}\label{lem:5.2}
 Let $0 \leq m \leq \bar{m}+1$. Assume that $(P_m)$ holds. Let $ \eta, \xi, \eps_0, K, \beta$ be as in $(P_m)$. If $N \in \NN$ and $c_{\ref{lem:5.2}}(\omega)= (4a_n^{-\eps_0} + 2^{2N_1(\omega)} 2K e^K)^2$, then on the event
 \[ \{ \omega: N \geq N_1(m,n,\xi,\eps_0, K,\beta),(t,x)\in Z(N,n,K,\beta) \},\]
 we have, denoting $\bar{d}_N = 2^{-N} \vee d((s,y),(t,x)),$
  \begin{equation}
    |u(s,y)|  \leq \sqrt{c_{\ref{lem:5.2}}} e^{|y-x|}\bar{d}_N^\xi 
     [(\sqrt{a_n} \vee \bar{d}_N)^{\gamma_m -1} + \1 \{m>0\} a_n^\beta ]
  \end{equation}
for all $s < T_K$ and $y \in \RR^{\dimension}$. 
\end{lemma}
As the proof is essentially the same as Lemma 5.2 in \cite{mp:11}' we omit it. The lemma gives control on $u(s,y)$ for $y$ close to points in $Z(N,n,K,\beta)$. To do the induction step we want to use this control in the estimate $|D(r,w)|  \leq R_0 e^{R_1|w|} |u(r,w)|^\gamma$ from (\ref{def:D}), which appeared in 
\begin{equation}
\label{def:F_de,m}
F_{\delta,l} (s,t,x) = \int_0^{(s-\delta)^+} p_{t-r, l}(w-x)D(r,w) \, W(dr\, dw),
\end{equation}
$\delta>0,$ $0 \leq s\leq t$, $x \in \RR^{\dimension}$ and for $1 \leq l \leq \dimension,$ see (\ref{Fdelta,ldef}) and Lemma \ref{lem:3.1}. This is related to the derivative of $u_{1,\delta}$ as given in Lemma \ref{lem:3.1}. Using the bound from Lemma \ref{lem:5.2} will lead to an improved bound on $u_{1,\delta}$. Later, we will also give estimates for $u_{2,\delta}$ and the combination of the two bounds allows us to do the induction step at the end of the section.

To estimate $F_{\delta,l}$ we use the following decomposition for $s \leq t \leq t'$, $s'\leq t'$:
  \begin{equation}\label{eq:F_delta:decomp}
  \begin{split}
  |F_{\delta,l} (s,t,x) & -F_{\delta,l}(s',t',x') | \\
  & \leq  |F_{\delta,l} (s,t,x) -F_{\delta,l}(s,t,x') |  
  + |F_{\delta,l} (s,t,x') -F_{\delta,l}(s,t',x') | \\  &\quad 
  + |F_{\delta,l} (s,t',x') -F_{\delta,l}(s',t',x') | \\
  & = | \int_0^{(s-\delta)^+} (p_{t-r, l}(w-x') - p_{t-r, l}(w-x) ) D(r,w) \, W(dr \, dw) |\\
  &\quad + |  \int_0^{(s-\delta)^+} (p_{t-r, l}(w-x') - p_{t'-r, l}(w-x') ) D(r,w) \, W(dr\, dw) | \\
  &\quad + | \int_{(s\wedge s'-\delta)^+}^{(s\vee s'-\delta)^+} p_{t'-r, l}(w-x')  D(r,w) \, W(dr\, dw) |, \ 1\leq l \leq \dimension.
  \end{split}
  \end{equation}
All of these three expressions in the moduli are martingales in the upper integral bound, when the rest of the values $x,x',t,t', (s\wedge s' -\delta)^+$ stay fixed. We want to consider the quadratic variations of these martingales and use the Dubins-Schwarz theorem. In order to calculate the first two quadratic variations we need to introduce the following partition of $\RR^{\dimension}$ (for fixed values of $x,x', \eta_0$):
  \begin{align*}
  A_1^{\eta_0} (r,t) &= \1 \{ y \in \RR^{\dimension}: |y-x| \leq (t-r)^{1/2 - \eta_0} \vee 2 |x-x'| \}
  \\
  A_2^{\eta_0} (r,t) &= \1 \{ y \in \RR^{\dimension}: |y-x| > (t-r)^{1/2 - \eta_0} \vee 2 |x-x'| \} = A_1^{\eta_0} (r,t)^C,
  \end{align*}
whenever $0\leq r < t$. 
For estimating \eqref{eq:F_delta:decomp}, we now introduce the following square functions for $i,j \in \{1,2\}$:
  \begin{align*}
    Q_{X,\delta, \eta_0}^{i,j} (s,t,x,t,x')  &= \int_0^{(s-\delta)^+} dr \int_{A_i^{\eta_0}(r,t)} dw \int_{A_j^{\eta_0}(r,t)} dz    \\
    & \quad | (p_{t-r, 1}(w-x') - p_{t-r, 1}(w-x) ) (p_{t-r, 1}(z-x') - p_{t-r, 1}(z-x) )|\\
    & \quad R_0^{2} e^{R_1 (|w|+|z|)} |u(r,w)|^\gamma |u(r,z)|^\gamma  (|w-z|^{-\alpha} + 1),
  \end{align*}
  \begin{align*}
    Q_{T,\delta, \eta_0}^{i,j} (s,t,x',t',x')  &= \int_0^{(s-\delta)^+} dr \int_{A_i^{\eta_0}(r,t')} dw \int_{A_j^{\eta_0}(r,t')} dz  \\
    & \quad | (p_{t-r, 1}(w-x') - p_{t'-r, 1}(w-x') ) (p_{t-r, 1}(z-x') - p_{t'-r, 1}(z-x') ) |\\
    & \quad R_0^{2} e^{R_1 (|w|+|z|)} |u(r,w)|^\gamma |u(r,z)|^\gamma  (|w-z|^{-\alpha} + 1)
  \end{align*}
and
  \begin{align*}
    Q_{S,\delta} (s,s',t',x') & = \int_{(s\wedge s'-\delta)^+}^{(s\vee s'-\delta)^+} dr \intRd dw \intRd dz \, | p_{t'-r, 1}(w-x') p_{t'-r, 1}(z-x') |\\   
    & \quad  R_0^{2} e^{R_1 (|w|+|z|)} |u(r,w)|^\gamma |u(r,z)|^\gamma (|w-z|^{-\alpha} + 1).
  \end{align*}
Now we want to establish an upper bound for 
  \begin{equation}\label{eq:Q:delta:bound}
  Q_{\delta, \eta_0}^{\text{tot}} (s,t,x,s',t',x') = Q_{S,\delta} (s,s',t',x')+ \sum_{i,j=1}^2 (Q_{T,\delta, \eta_0}^{i,j} (s,t,x',t',x') + Q_{X,\delta, \eta_0}^{i,j} (s,t,x,t,x')),
  \end{equation}
when $s,t,x,s',t',x'$ are subject to some restrictions. Then \eqref{eq:Q:delta:bound} is clearly an upper bound itself for the quadratic variation of each of the three martingales in \eqref{eq:F_delta:decomp}.
\begin{remark}\label{rem:only:l=1}
Since we would execute the same calculations  for any spatial dimension $l$ we restrict ourselves now to $l=1$ for the estimates on $F_{\delta,l}$. We already omitted this dependence in the definitions leading up to (\ref{eq:Q:delta:bound}). Also, note that
dependence of constants on the universal constants $\alpha, \dimension, \gamma, R_0$ and $R_1$ will not be mentioned in the following lemmas.
\end{remark}
We combine two estimates for the cases $(i,j)= (1,2), (2,1)$ or $(2,2),$ so $i+j\geq 3:$
\begin{lemma}\label{lem:5.4a}
For all $K \in \NN^{\geq K_1}$, $R>2$ there exist $c_{\ref{lem:5.4a}}(K,R)$, $N_{\ref{lem:5.4a}}(K,\omega)$ almost surely such that $\forall \eta_0, \eta_1 \in (1/R, 1/2)$, $\delta \in (0,1]$, $\beta \in [0,1/2]$, $N, n \in \NN$, $(t,x) \in \RR_+\times \RR^{\dimension}$ the following holds for $i+j \geq 3$: For $\omega \in \{ (t,x) \in Z(N,n,K,\beta), N \geq N_{\ref{lem:5.4a}} \}$ we have
  \begin{equation}
    \begin{split}
        Q_{X,\delta, \eta_0}^{i,j} (s,t,x,t,x') & \leq %
    c_{\ref{lem:5.4a}} 2^{4N_{\ref{lem:5.4a}}} |x-x'|^{2-\eta_1}, \\
Q_{T,\delta, \eta_0}^{i,j} (s,t,x',t',x') & \leq c_{\ref{lem:5.4a}} 2^{4N_{\ref{lem:5.4a}}}
     [|t-t'|^{1-\eta_1/2} +  |t-t'|^{1-\eta_1/2}  \delta^{-1-\alpha/2}
    (|t-t'|\wedge 1)^{4\gamma} ],\\
    \end{split}
  \end{equation}
for all $0 \leq s \leq t \leq t', |x'| \leq K+1.$
\end{lemma}

\begin{proof}
We will just give the proof for $i=2$ without taking into account $j$, i.e.~the restriction on $z$. This suffices by symmetry.
Use the estimate \eqref{def:D} on $D$, take $\xi = 3/4, m=0$ and set $N_{\ref{lem:5.4a}} (K,\omega) = N_1(0,n,3/4,\eps_0,K,\beta),$ see the definition of $(P_m)$ in (\ref{P_mprop}). Also set w.l.o.g.~$\delta < s.$ Then, in Lemma \ref{lem:5.2} for the case $m=0$ we can take $\eps_0 = 0$ as well ($c_{\ref{lem:5.2}} = C(K)2^{N_1(0,3/4,K)}$) and obtain
  \begin{align*}
    Q_{X,\delta, \eta_0}^{i,j} (s,t,x,t,x') & \leq c_{\ref{lem:5.2}} \int_0^{s-\delta} dr \intRd \intRd \1 \{ |w-x| > (t-r)^{1/2 - \eta_0} \vee 2 |x-x'| \} \\
    &\qquad \quad (p_{t-r, 1}(w-x') - p_{t-r, 1}(w-x) ) (p_{t-r, 1}(z-x') - p_{t-r, 1}(z-x) )\\
    &\qquad \quad e^{R_1 (|w|+|z|)} e^{\gamma   (|w-x|+|z-x|)} R_0^{2}  (|w-z|^{-\alpha} + 1)\\
    &\qquad \quad [2^{-N} \vee d((r,w),(t,x))]^{3\gamma/4} [2^{-N} \vee d((r,z),(t,x))]^{3\gamma/4} 
    \, dw  dz. 
  \end{align*}
Using $d((r,w),(t,x))^\gamma = (\sqrt{t-r} + |w-x| )^\gamma \leq 2 ((t-r)^{\gamma/2} + |w-x|^\gamma )$ and $t,r\leq K$ we bound this by
  \begin{align*}
  & c_{\ref{lem:5.2}} R_0^2 \int_0^{(s-\delta)^+} dr \intRd \intRd \1 \{ |w-x| > (t-r)^{1/2 - \eta_0} \vee 2 |x-x'| \} \\
    &\qquad \quad (p_{t-r, 1}(w-x') - p_{t-r, 1}(w-x) ) (p_{t-r, 1}(z-x') - p_{t-r, 1}(z-x) )\\
    &\qquad \quad e^{2 R_1 |x|} e^{(\gamma + R_1)   (|w-x|+|z-x|)} \\
    &\qquad \quad 2( K^{3\gamma/8} + |w-x|^{3\gamma/4})\, 2(K^{3\gamma/8} + |z-x|^{3\gamma/4})
    (|w-z|^{-\alpha} + 1) \, dw  dz. 
  \end{align*}
With the help of Lemma \ref{lem:4.4} for $t=t'\leq K$ bound this by
  \begin{align*}
  & c_{\ref{lem:5.2}} C(K,R) \int_0^{s-\delta} dr \, (t-r)^{-1-\alpha/2} \exp (- \frac{\eta_1 (t-r)^{-2\eta_0}}{256} ) [1\wedge \frac{|x-x'|^2}{t-r} ]^{1-\eta_1/2}  \\
  & \leq c_{\ref{lem:5.2}} C(K,R) (256R)^R \int_0^{s-\delta} dr \,[1\wedge \frac{|x-x'|^2}{t-r} ]^{1-\eta_1/2} \\
  & \leq c_{\ref{lem:5.2}} C(K,R) |x-x'|^{2-\eta_1} \int_0^t dr\, (t-r)^{\eta_1/2 -1} \\
  & \leq c_{\ref{lem:5.2}} C(K,R) |x-x'|^{2-\eta_1},
  \end{align*}
  where we used Lemma \ref{lem:algebra}. The proof for the temporal estimate is similar but we omit it here.
\end{proof}
Next we need to consider the distances for the cases $i=j=1$.
\begin{lemma}\label{lem:5.5a}
  Let $0\leq m \leq \bar{m}+1$ and assume that $(P_m)$ holds. For all $K \in \NN^{\geq K_1}$, $R>2$ $n\in \NN$, $\beta \in [0,1/2]$, $\eps_0 \in (0,1)$, there exist $c_{\ref{lem:5.5a}}(K,R)$, $N_{\ref{lem:5.5a}}(m,n,R,\eps_0,K,\beta)(\omega) \in \NN$ almost surely such that 
  for all $\eta_1 \in (1/R, 1/2)$, $\eta_0 \in (0,\eta_1/32)$, $\delta \in [a_n,1]$, $N \in \NN$, $(t,x) \in \RR_+\times \RR^{\dimension}$ the following holds: For $\omega \in  \{ (t,x) \in Z(N,n,K,\beta), N \geq N_{\ref{lem:5.5a}} \}$ we have
    \begin{align}
\nn      Q_{X,\delta, \eta_0}^{1,1} (s,t,x,t,x')  & \leq c_{\ref{lem:5.5a}} [a_n^{-2\eps_0} + 2^{4N_{\ref{lem:5.5a}}}]  \left[ |x-x'|^{2-\eta_1} (\bar{\delta}_N^{(\gamma \gamma_m - 1 -\alpha/2)\wedge 0} + a_n^{2\beta \gamma} \bar{\delta}_N^{(\gamma-1-\alpha/2)\wedge0} ) \right. \\
      & \left. \quad  \quad  \quad  \quad  \quad  \quad  \quad  \quad  \quad  \quad  + (d\wedge \sqrt{\delta})^{2-\eta_1} \delta^{-1-\alpha/2} [\bar{d}_N^{2\gamma \gamma_m} + a_n^{2\beta \gamma} \bar{d}_N^{2\gamma} ] \right], \\
  \nn Q_{T,\delta, \eta_0}^{1,1} (s,t,x',t',x') & \leq c_{\ref{lem:5.5a}} [a_n^{-2\eps_0} + 2^{4N_{\ref{lem:5.5a}}}] \\
   & \quad \quad \left[|t-t'|^{1-\eta_1/2}  (\bar{\delta}_N^{(\gamma \gamma_m - 1 -\alpha/2)\wedge 0} + a_n^{2\beta \gamma} \bar{\delta}_N^{\gamma-1-\alpha/2} ) \right. \\
  \nn & \quad \quad + \left. (|t-t'| \wedge \delta)^{1-\eta_1/2}  \delta^{-1-\alpha/2}  [\bar{d}_N^{2\gamma \gamma_m} + a_n^{2\beta \gamma} \bar{d}_N^{2\gamma} ] \right], 
    \end{align}
for all $ 0 \leq s \leq t, |x'| \leq K+1.$
  Here $\bar{d}_N = |x-x'| \vee 2^{-N}$ and $\bar{\delta}_N = \delta \vee \bar{d}_N^2.$ Moreover, $N_{\ref{lem:5.5a}}$ is stochastically bounded uniformly in $(n,\beta).$
\end{lemma}

\begin{proof}
First, we estimate $Q_X$. Let $\xi = 1-(8R)^{-1} \in (15/16,1)$ and set $N_{\ref{lem:5.5a}} = N_1(m,n,\xi,\eps_0,K,\beta)$. W.l.o.g.~$s > \delta$ and therefore we always have $d((r,w),(t,x)) \wedge d((r,z),(t,x)) \geq \sqrt{a_n}$ in the integral. An application of Lemma \ref{lem:5.2} and the bound on $|w-x|, |z-x|$ respectively gives
  \begin{align*}
    Q_{X,\delta, \eta_0}^{1,1} (s,t,x,t,x') & \leq c_{\ref{lem:5.2}} \int_0^{s-\delta} dr \intRd \intRd dwdz\, e^{4 R_1 K} e^{4 \gamma   K} R_0^{2\gamma} \\
    & \quad (p_{t-r, 1}(w-x') - p_{t-r, 1}(w-x) ) (p_{t-r, 1}(z-x') - p_{t-r, 1}(z-x) )\\
    & \quad [2^{-N} \vee ((t-r)^{1/2} + (t-r)^{1/2 - \eta_0} \vee 2|x-x'| )]^{2\gamma \xi} \\
    & \quad \{ [2^{-N} \vee ((t-r)^{1/2} + (t-r)^{1/2 - \eta_0} \vee 2|x-x'| )]^{\gamma_m -1} + a_n^\beta \}^{2\gamma}\\
    & \quad (|w-z|^{-\alpha} +1 ).
  \end{align*}
Let $\gamma' = \gamma (1-2\eta_0)$ and observe the trivial inequalities 
  \begin{align}
  \sqrt{t-r} & \leq K^{\eta_0} (t-r)^{1/2-\eta_0}, \label{eq:t.r.eta}\\
  |x-x'| &\leq C(\dimension) K |x-x'|^{1-2\eta_0} . \nonumber
  \end{align}
Then, Lemma \ref{lem:heat-ker:diff:2} allows the following bound
  \begin{align*}
  Q_{X,\delta, \eta_0}^{1,1} (s,t,x,t,x') & \leq c_{\ref{lem:5.2}} C(K) \int_0^{s-\delta} dr\, (t-r)^{-1-\alpha/2} [1\wedge\frac{|x-x'|^2}{t-r}] \\
  & \quad \phantom{AAAAAAA}[2^{-2N\gamma} \vee (t-r)^{\gamma'} \vee  |x-x'|^{2\gamma'} ]^\xi \\
  & \quad \phantom{AAAAAAA} [2^{-N\gamma'(\gamma_m-1)} \vee (t-r)^{\gamma'(\gamma_m-1)}  \vee |x-x'|^{2\gamma(\gamma_m-1)}  + a_n^{2\beta \gamma}].
  \end{align*}
Using 
  \begin{align*}
  2^{-2N\gamma} \vee (t-r)^{\gamma'} \vee  |x-x'|^{2\gamma'}  & \leq 2^{-2N\gamma'}  \vee  |x-x'|^{2\gamma'} +  (t-r)^{\gamma'} \\
  & \leq 2 [\bar{d}_N^{2\gamma'} \vee (t-r)^{\gamma'}],
  \end{align*}
we can bound the above by
  \begin{eqnarray}
  &&Q_{X,\delta, \eta_0}^{1,1} (s, t,x,t,x') \\
  &\leq& c_{\ref{lem:5.2}} C(K) \int_0^{s-\delta} dr\, (t-r)^{-1-\alpha/2} [1\wedge\frac{|x-x'|^2}{t-r}] \, 2^\xi (\bar{d}_N^{2\gamma' \xi} \vee (t-r)^{\gamma' \xi} )\nonumber \\
  && \qquad \qquad \qquad \qquad  2^{\gamma_m -1} [(\bar{d}_N^{2} \vee (t-r))^{\gamma' (\gamma_m-1)} + a_n^{2\beta \gamma}] \nonumber \\
  &\leq& 4 c_{\ref{lem:5.2}} C(K)  \int_0^{s-\delta} dr\, \1 \{t-r \geq \bar{d}_N^2 \} (t-r)^{-1-\alpha/2+\gamma' \xi} [1\wedge\frac{|x-x'|^2}{t-r}]  \nonumber \\
  && \qquad \qquad \qquad \qquad [ (t-r)^{\gamma' (\gamma_m-1)} + a_n^{2\beta \gamma}] \nonumber\\
  && \quad + 4 c_{\ref{lem:5.2}} C(K) \int_0^{s-\delta} dr\, \1 \{t-r < \bar{d}_N^2 \} (t-r)^{-1-\alpha/2} [1\wedge\frac{|x-x'|^2}{t-r}]  \bar{d}_N^{2\gamma' \xi} \nonumber\\
  && \qquad \qquad \qquad \qquad [\bar{d}_N^{2\gamma' (\gamma_m-1)} + a_n^{2\beta \gamma}] \nonumber \\
  &=& c_{\ref{lem:5.2}} C(K) (I_1 + I_2). \label{eq:pr:5.5:1}
  \end{eqnarray}
We start with an estimate on $I_1$. If $r\leq s-\delta$ and $t-r \geq \bar{d}_N^2$ then
  \begin{equation}\label{eq:r:leq:s-delta}
  r \leq t- \bar{d}_N^2 \wedge s-\delta \leq t - \bar{d}_N^2 \wedge t -\delta = t- \bar{\delta}_N.
  \end{equation}
Use that to obtain 
  \begin{align*}
  I_1 &\leq \int_0^{t-\bar{\delta}_N} dr\, \left((t-r)^{-1-\alpha/2+\gamma' \xi + \gamma' (\gamma_m -1)} [1\wedge\frac{|x-x'|^2}{t-r}] \right.\\
  & \phantom{\leq \int_0^{t-\bar{\delta}_N} dr\, } \left. + (t-r)^{-1-\alpha/2+\gamma' \xi} [1\wedge\frac{|x-x'|^2}{t-r}] a_n^{2\beta \gamma}\right). 
  \end{align*}
We want to drop the minimum with $1$ to consider
$$ |x-x'|^2 \int_{\bar{\delta}_N}^t du\, \left( u^{-2-\alpha/2+\gamma' \xi + \gamma' (\gamma_m -1)}   + u^{-2-\alpha/2+\gamma' \xi}  a_n^{2\beta \gamma}  \right).$$
It holds that for  $p\in (-1,1), 0<a<b \leq K,$
  \begin{align}
  \int_a^b u^{p-1}\, du  \leq \log (b/a) (a^{p} + b^{p})\leq 2K  \log (b/a) a^{p \wedge 0}.
  \end{align}
Here, the first inequality follows since for $p \neq 0$ the left hand side equals $\frac{1}{|p|}|a^{p} - b^{p}|,$ which can be bounded using $1-x \leq - \log x, x \geq 0.$ For $p=0$ we even have equality. The second inequality follows by distinguishing cases for $p$ negative, positive, and zero, and by noting that $K \geq 1.$
Hence,  we have
  \begin{align*}
  I_1 & \leq 2 K |x-x'|^2 \log(K/\bar{\delta}_N) ( \bar{\delta}_N^{(-1-\alpha/2+\gamma' \xi + \gamma' (\gamma_m -1) )\wedge 0} +\bar{\delta}_N^{(-1-\alpha/2+\gamma' \xi)\wedge 0} a_n^{2\beta \gamma}).
  \end{align*}
The $\log$-term is bounded by $C(K,R)|x-x'|^{-\eta_1/2}$ (use Lemma \ref{lem:algebra}). Moreover, by Lemma 4.1(c) in \cite{mp:11} we bound
  \begin{align}
  I_2 & \leq  \frac{2}{2/ \alpha} (\delta \wedge |x-x'|^2) \delta^{-1-\alpha/2}\bar{d}_N^{2\gamma' \xi} [\bar{d}_N^{2\gamma' (\gamma_m-1)} + a_n^{2\beta \gamma}].
  \end{align}
Therefore,
  \begin{align*}
  Q_{X,\delta, \eta_0}^{1,1} (s,t,x,t,x')  \leq & c_{\ref{lem:5.2}} C(K,R) \\
  & \, \big[ |x-x'|^{2-\eta_1/2} ( \bar{\delta}_N^{(-1-\alpha/2  + \gamma' (\gamma_m +\xi-1) )\wedge 0} +\bar{\delta}_N^{(-1-\alpha/2+\gamma' \xi\wedge 0 )} a_n^{2\beta \gamma})\\
  &  \quad + (\delta \wedge |x-x'|^2) \delta^{-1-\alpha/2}\bar{d}_N^{2\gamma' \xi} [\bar{d}_N^{2\gamma' (\gamma_m-1)} + a_n^{2\beta \gamma} \big].
  \end{align*}
To finish the proof for $Q_X$ we replace $\xi=1-(8R)^{-1}$ by $1$ and $\gamma'=\gamma (1-2\eta_0)$ by $\gamma$ at the cost of multiplying by $d^{-\eta_1/2}\geq \bar{\delta}_N^{-\eta_1/4}$ since
$$\xi \gamma' = \gamma ( 1-2\eta_0)(1-(8R)^{-1}) \geq \gamma (1- \eta_1/4), \text{ hence } \xi \gamma'-\gamma \geq -\gamma\eta_1/4 \geq -\eta_1/4$$
and
  \begin{align*}
  \gamma' (\gamma_m + \xi -1) & = \gamma (1-2\eta_0)(\gamma_m - \frac{1}{8R}) \geq \gamma \gamma_m - \frac{\eta_1}{4}
  \end{align*}
by some algebra using $\eta_1 >32\eta_0 \vee R^{-1}.$

We will not give the proof for $Q_T$ as it is quite similar except that some exponents change slightly.
\end{proof}
Finally, there is an estimate on $Q_S$:
\begin{lemma}\label{lem:5.6}
 Let $0\leq m \leq \bar{m}+1$ and assume that $(P_m)$ holds. For all $K \in \NN^{\geq K_1}$, $R>2$ $n\in \NN$, $\beta \in [0,1/2]$, $\eps_0 \in (0,1)$, there exist $c_{\ref{lem:5.6}}(K,R,\gamma)$, $N_{\ref{lem:5.6}}(m,n,R,\eps_0,K,\beta)(\omega) \in \NN$ almost surely such that for all $ \eta_1 \in (1/R, 1/2)$,  $\delta \in [a_n,1]$, $N \in \NN$, $(t,x) \in \RR_+\times \RR^{\dimension}$ the following holds: For $  \omega \in  \{ (t,x) \in Z(N,n,K,\beta), N \geq N_{\ref{lem:5.6}} \}$ we have
  \begin{align*}
   Q_{S,\delta} (s,s',t',x') & \leq c_{\ref{lem:5.6}} [a_n^{-2\eps_0} + 2^{4N_{\ref{lem:5.6}}}] \\
   & \quad \quad \left[ |s-s'|^{1-\eta_1/2} (\bar{\delta}_N^{(\gamma \gamma_m - 1 -\alpha/2)\wedge 0} + a_n^{2\beta \gamma} \bar{\delta}_N^{(\gamma-1-\alpha/2)\wedge0} ) \right. \\
   & \qquad \left. + (|s'-s|\wedge \delta)^{1-\eta_1/2} \1 \{\delta < \bar{d}_N^2 \}  \delta^{-1-\alpha/2} [\bar{d}_N^{2\gamma \gamma_m} + a_n^{2\beta \gamma} \bar{d}_N^{2\gamma} ]\right] 
  \end{align*}
for all $ 0 \leq s \leq t, s' \leq t', |x'| \leq K+1.$ Here $\bar{d}_N = (|t-t'|^{1/2} + |x-x'|) \vee 2^{-N}$ and $\bar{\delta}_N = \delta \vee \bar{d}_N^2.$ Moreover, $N_{\ref{lem:5.6}}$ is stochastically bounded uniformly in $(n,\beta)$.
\end{lemma}
We omit the proof, since again it is similar to a proof in \cite{mp:11} (Lemma 5.6), here using $\xi = (3/2-(2\gamma)^{-1}) \wedge (1-(4\gamma R)^{-1})$ as well as Lemma \ref{lem:algebra} and Lemma \ref{lem:mps:5.1:add}.

\noindent \textbf{Notation:} For $s,t,s',t' \geq 0, x,x' \in \Rdim$ we now introduce 
\begin{equation}
\label{tilde{d}}
\tilde{d} ((s,t,x),(s',t',x')) := \sqrt{|s-s'|} + \sqrt{|t-t'|} + |x-x'|.
\end{equation}
As a corollary of all the previous calculations we get a bound on $Q_{\delta,\eta_0}^{\text{tot}}$ as defined in \eqref{eq:Q:delta:bound}.
\begin{corollary}\label{cor:local:bounds:on:F}
 Let $0\leq m \leq \bar{m}+1$ and assume that $(P_m)$ holds. For all $K \in \NN^{\geq K_1}$, $R>2,$ $n\in \NN$, $\beta \in [0,1/2]$, $\eps_0 \in (0,1)$, there exist $c_{\ref{cor:local:bounds:on:F}}(K,R)$, $N_{\ref{cor:local:bounds:on:F}}(m,n,R,\eps_0,K,\beta)(\omega) \in \NN$ almost surely such that for all $ \eta_1 \in (1/R, 1/2)$, $\eta_0 \in (1/R, \eta_1/32),$  $\delta \in [a_n,1]$, $N \in \NN$, $(t,x) \in \RR_+\times \RR^{\dimension}$, the following holds:  For  $\omega \in  \{ (t,x) \in Z(N,n,K,\beta), N \geq N_{\ref{cor:local:bounds:on:F}} \}$ we have
 \begin{equation}\begin{split}
  Q_{\delta,\eta_0}^{\text{tot}} (s,t,x,s',t',x')  \leq c_{\ref{cor:local:bounds:on:F}} (a_n^{-2\eps_0} + 2^{4N_{\ref{cor:local:bounds:on:F}}} ) \tilde{d}^{2-\eta_1} [& \delta^{-1-\alpha/2} \bar{d}_N^{\gamma \gamma_m} + \delta^{-1-\alpha/2} a_n^{2\beta \gamma} \bar{d}_N^{2 \gamma} \\
  &  + \bar{\delta}_N^{(\gamma \gamma_m -1-\alpha/2)\wedge 0} + a_n^{2\beta \gamma}\bar{\delta}_N^{\gamma -1-\alpha/2}]   \end{split}
 \end{equation}
for all $0 \leq s \leq t\leq t'\leq T_K, |x'| \leq K+1.$
Here, $\tilde{d} = \tilde{d}((s,t,x),(s',t',x'))$ from (\ref{tilde{d}}), $\bar{d}_N = d((t,x),(t',x')) \vee 2^{-N}$ and $\bar{\delta}_N = \delta \vee \bar{d}_N$.
Moreover, $N_{\ref{cor:local:bounds:on:F}}$ is stochastically bounded uniformly in $(n,\beta).$
\end{corollary}
\noindent
\textbf{Notation: } Let us now introduce
\begin{equation}\label{eq:def:Delta:u_1'}
 \begin{split}
  \bar{\Delta}_{u_1'} (m,n,\lambda,\eps_0,2^{-N}) =  a_n^{-\eps_0}[& a_n^{-\lambda/2(1+\alpha/2)} 2^{-N\gamma \gamma_m} + (a_n^{\lambda/2} \vee 2^{-N})^{(\gamma \gamma_m -1-\alpha/2)\wedge 0} \\
  & + a_n^{-\lambda/2(1+\alpha/2) + \beta \gamma} (a_n^{\lambda/2} \vee 2^{-N})^\gamma ].
 \end{split}
\end{equation}
We note that in this definition and in the following $\lambda \in [0,1]$ replaces the analogous $\alpha$ of \cite{mp:11}.

\begin{proposition}\label{prop:5.8}
 Let $0\leq m \leq \bar{m}+1$ and assume that $(P_m)$ holds. For all $n\in \NN, \eta_1\in (0,1/2], \eps_0\in(0,1), K\in \NN^{\geq K_1}, \lambda \in [0,1], \beta\in[0,1/2]$ there is an $N_{\ref{prop:5.8}}=N_{\ref{prop:5.8}}(m,n,\eta_1,\eps_0,K,\lambda,\beta)(\omega) \in \NN^{\geq 2}$ almost surely such that for all $ N \geq N_{\ref{prop:5.8}}$, $(t,x) \in Z(N,n,K,\beta)$, $s\leq t $, $s'\leq t'\leq T_K$ and $\tilde{d} =\tilde{d}((s,t,x),(s',t',x')) < 2^{-N}$ it holds that
 \begin{equation}
  |F_{a_n^\lambda,l} (s,t,x) - F_{a_n^{\lambda},l} (s',t',x')| \leq 2^{-86}\dimension^{-4}  \tilde{d}^{1-\eta_1}  \bar{\Delta}_{u_1'} (m,n,\lambda,\eps_0,2^{-N}), \quad l=1,\dots,\dimension.
 \end{equation}
Moreover, $N_{\ref{prop:5.8}}$ is stochastically bounded uniformly in $(n,\lambda,\beta)$. 
\end{proposition}
\begin{proof}We do the proof for $l=1$ only, see Remark \ref{rem:only:l=1}.
 Let $R=33\eta_1^{-1}, \eta_0 \in (R^{-1},\eta_1/32)$ and consider the case $t\leq t'$ in the beginning only. Set
 \begin{align*}
  d &= d((t,x),(t',x')), \\
  \tilde{d} & = \sqrt{|s'-s|} + d, \\
  \bar{d}_N & = d \vee 2^{-N},\\
  \bar{\delta}_{n,N} & = a_n^\lambda \vee \bar{d}_N^2.
 \end{align*}
By Corollary \ref{cor:local:bounds:on:F} for $(t,x)\in Z(N,n,K,\beta), N\geq N_{\ref{cor:local:bounds:on:F}}$ it holds that
\begin{equation}\label{eq:pr:prop:5.8:1}
 \begin{split}
   Q_{a_n^\lambda,\eta_0}^{\text{tot}}(s,t,x,s',t',x')^{1/2} & \leq c_{\ref{cor:local:bounds:on:F}}^{1/2}(K,\eta_1) (a_n^{-\eps_0} + 2^{2N_{\ref{cor:local:bounds:on:F}}})  \tilde{d}^{1-\eta_1/2} \\ 
  & \quad   [(a_n^\lambda)^{-1/2(1+\alpha/2)} [\bar{d}_N^{\gamma \gamma_m} + a_n^{\beta \gamma} \bar{d}_N^\gamma ] \\ 
  & \quad \quad + \bar{\delta}_{n,N}^{(\gamma \gamma_m -1 -\alpha/2)/2 \wedge 0} + a_n^{\beta \gamma} \bar{\delta}_{n,N}^{(\gamma - 1-\alpha/2)/2} ]  \end{split}
\end{equation}
for all $s\leq t\leq t', s'\leq t' \leq T_K, |x'| \leq K+2.$
Therefore define for $c_{\dimension}= 2 + \log (2+\dimension)$
\begin{equation*}
 \begin{split}
  \Delta(m,n,\bar{d}_N) &= 2^{-96} a_n^{-\eps_0} \{ a_n^{-\lambda/2(1+\alpha/2)} [ \bar{d}_N^{\gamma \gamma_m} + a_n^{\beta \gamma} \bar{d}_N^\gamma ] \\
  & \quad \quad + (\sqrt{\bar{\delta}_{n,N}})^{(\gamma \gamma_m -1-\alpha/2)\wedge 0} + a_n^{\beta \gamma} (\sqrt{\bar{\delta}_{n,N}})^{\gamma-1-\alpha/2} \} (\dimension^5  4^{c_{\dimension}})^{-1},
 \end{split}
\end{equation*}
which satisfies
$$ \Delta(m,n,2^{-N+1}) \leq (2^{\gamma \gamma_m}\vee 2^\gamma \vee 2^0 \vee 2^{\gamma -1 -\alpha/2} ) \Delta(m,n,2^{-N}) \leq 4 \Delta (m,n,2^{-N}).$$
Choose $N_3 = \frac{33}{\eta_1}[N_{\ref{cor:local:bounds:on:F}} + N_4(K,\eta_1)] + \frac{4}{\eta_1}(8+10\log \dimension)$, where $N_4$ is chosen in such a way that
\begin{align*}
 \dimension^5 4^{c_{\dimension}} c_{\ref{cor:local:bounds:on:F}}^{1/2} [a_n^{-\eps_0} + 2^{2N_{\ref{cor:local:bounds:on:F}}}] 2^{-\eta_1N_3 /4} & \leq c_{\ref{eq:pr:prop:5.8:1}} (K,\eta_1)[a_n^{-\eps_0} + 2^{2N_{\ref{cor:local:bounds:on:F}}}]2^{-8N_{\ref{cor:local:bounds:on:F}} -8N_4} \\
 & \leq a_n^{-\eps_0} 2^{-100},
\end{align*}
i.e. $N_4 = N_4(a_n,\eps_0,N_{\ref{cor:local:bounds:on:F}},c_{\ref{eq:pr:prop:5.8:1}})$ and hence $N_3 = N_3(n,\eps_0,N_{\ref{cor:local:bounds:on:F}},K,\eta_1)$, which is stochastically bounded uniformly in $(n,\lambda,\beta)$.

Let $N' \in \NN$ be such that $\tilde{d} \leq 2^{-N'}$, which implies $\tilde{d}^{1-\eta_1/2} \leq 2^{-N'\eta_1/4}\tilde{d}^{1-3\eta_1/4}$. Then it is true that on the event
\begin{align*}
 \{ \omega: &(t,x)\in Z(N,n,K+1,\beta), N\geq N_3,N'\geq N_3\}\\
 \intertext{we have that}
 Q_{a_n^\lambda,\eta_0}^{\text{tot}}(s,t,x,s',t',x')^{1/2} & \leq c_{\ref{cor:local:bounds:on:F}}^{1/2} (a_n^{-\eps_0} + 2^{2N_{\ref{cor:local:bounds:on:F}}}) 2^{-N'\eta_1/4}\tilde{d}^{1-3\eta_1/4} 2^{100} a_n^{\eps_0}\Delta(m,n,\bar{d}_N) (\dimension^7 4^{c_{\dimension}}) \\
 & \leq \tilde{d}^{1-3\eta_1/4} \frac{1}{16}\Delta(m,n,\bar{d}_N), 
\end{align*}
whenever $ s\leq t \leq t', s'\leq t' \leq T_k, |x'|\leq K+2.$
Recalling the decomposition of $F_{\delta,1}$ in \eqref{eq:F_delta:decomp} into the sum of three martingales and applying the Dubins-Schwarz-Theorem we can write as long as $s\leq t \leq t', s'\leq t'$ and $\tilde{d} \leq 2^{-N},$
\begin{align}
 \IP [ & |F_{a_n^\lambda,1} (s,t,x) - F_{a_n^{\lambda},1} (s',t',x')| \geq \tilde{d}((s,t,x)(s',t',x'))^{1-\eta_1} \Delta(m,n,\bar{d}_N),  \nonumber \\
 &\qquad \qquad \phantom{AAAAAAAAAAAAA} (t,x) \in Z(N,n,K+1,\beta), N' \wedge N \geq N_3, t'\leq T_K ] \nonumber \\
 & \leq 3 \IP[ \sup_{u \leq \tilde{d}^{2-3\eta_1/2}(\Delta(m,n,\bar{d}_N)/16)^2 } |B(u)| \geq \tilde{d}^{1-\eta_1} \Delta(m,n,\bar{d}_N)/3  ] \nonumber \\
 & \leq 3 \IP[\sup_{u\leq 1} |B(u)| \geq \tilde{d}^{-\eta_1/4} ] \nonumber \\
 & \leq C \int_{\tilde{d}^{-\eta_1/4}}^{\infty} \exp(-y^2/2) \, dy \nonumber \\
 & \leq C \exp(-\tilde{d}^{-\eta_1/2}/2), \label{eq:pr:5.8:1}
\end{align}
where we used the Reflection Principle in the next to last inequality.

Next apply a lemma similar to the Kolmogorov-Centsov estimate Lemma 5.7 in \cite{mp:11}, which is used in the proof of Proposition 5.8 in \cite{mp:11}. For details we refer to the proof in Section 9.4 of \cite{rippl:12}. Then, we obtain for a certain $N_{\ref{prop:5.8}}$ which is bounded uniformly in $n,\lambda, \beta$: For $N \geq N_{\ref{prop:5.8}}$ and $(t,x)\in Z(N,n,K,\beta)(\omega)$, $\tilde d= \tilde d((s,t,x),(s',t',x'))\leq 2^{-N}$ and $s\leq t \leq t', s' \leq t'\leq T_K$ we have
\begin{equation}
 |F_{a_n^\lambda,1} (s,t,x) - F_{a_n^{\lambda},1} (s',t',x')| \leq 32 (\dimension+2) 4^{c_{\dimension}+1} \Delta(m,n,2^{-N}) \tilde{d}^{1-\eta_1} .
\end{equation}
Thus, 
\begin{equation}\label{eq:pr:prop:5.8:2}
 |F_{a_n^\lambda,1} (s,t,x) - F_{a_n^{\lambda},1} (s',t',x')| \leq 2^{-88}\dimension^{-4}  \tilde{d}^{1-\eta_1}  \bar{\Delta}_{u_1'} (m,n,\lambda,\eps_0,2^{-N}) .
\end{equation}
However if $t' \leq t$, then $(t',x') \in Z(N-1,n,K+1,\beta),$ and interchanging $(s,t,x)$ with $(s',t',x')$ gives the same estimate as \eqref{eq:pr:prop:5.8:1} so that we  obtain that $ Q_{a_n^\lambda}^{tot}(s',t',x',s,t,x) $ is bounded by 4 times the right hand side of  \eqref{eq:pr:prop:5.8:1}.
Proceeding as in the case $t\leq t'$ we end up with \eqref{eq:pr:prop:5.8:2} replaced by
\begin{equation*}
 |F_{a_n^\lambda,1} (s',t',x') - F_{a_n^{\lambda},1} (s,t,x)| \leq 2^{-86}\dimension^{-4}  \tilde{d}^{1-\eta_1}  \bar{\Delta}_{u_1'} (m,n,\lambda,\eps_0,2^{-N}) .
\end{equation*}
This completes the proof for the first coordinate. Clearly, the constants $c_{\ref{prop:5.8}}$ and $N_{\ref{prop:5.8}}$ can be chosen such that the result holds uniformly for all dimensions $1\leq l \leq \dimension$.\\
\end{proof}

So putting things together we get for $\nabla u_{1,\delta} (t,x) = (F_{\delta,l}(t,t,x))_{1\leq l \leq \dimension}$:
\begin{corollary}\label{cor:5.9}
 Let $0\leq m \leq \bar{m}+1$ and assume that $(P_m)$ holds. Let $n,\eta_1,\eps_0,K,\lambda$ and $\beta$ be as in Proposition \ref{prop:5.8}. For all $N\geq N_{\ref{prop:5.8}}, (t,x) \in Z(N,n,K,\beta)$, $x'\in \Rdim$ and $t' \leq T_K$:
 \begin{align*}
  & d((t,x),(t',x')) \leq 2^{-N} \text{  implies that}\\
  & |\nabla u_{1,a_n^{\lambda}}(t,x) - \nabla u_{1,a_n^{\lambda}}(t',x')| \leq 2^{-85} \dimension^{-2} \, d((t,x),(t',x'))^{1-\eta_1} \bar{\Delta}_{u_1'}(m,n, \lambda,\eps_0,2^{-N}).
 \end{align*}
\end{corollary}
This result gives us something like a H\"older regularity of the gradient $\nabla u_{1,\delta}$ with $\delta = a_n^\lambda$. This will be helpful later.

Recalling the definition of $J_{n,i}$, however, we just ``know'' the range of the gradients of $u_{1,\delta}$ for $\delta = a_n$. But it will be helpful to find a result relating this range to the  gradients of $u_{1,\delta}$ for $\delta = a_n^{\lambda}$. The definition of $F_{\delta,l}$ allows us to relate these two gradients, since for $\delta \geq a_n$ and $s=t-\delta + a_n,$
\begin{align}
 \partial_{x_l} u_{1,\delta} (t,x) = \partial_{x_l} P_\delta (u_{(t-\delta)^+}) (x) & = \partial_{x_l} P_{t-s+a_n}(u_{(s-a_n)^+})(x) \nonumber \\
 & = - F_{a_n,l}(s,t,x) \label{eq:u_1:F} \\
 & = - F_{a_n,l}(t-\delta+a_n,t,x). \nonumber 
\end{align}
Note the last equality holds for any $t,\delta, a_n \geq 0$, where they are trivial if $t-\delta \leq 0$. So we need to relate $F_{a_n,l}(t-a_n^\lambda+a_n,t,x)$ and $F_{a_n,l}(t,t,x)$. We can show a lemma on the square function $Q_{T,a_n}(s,t,t,x)$ using Lemma \ref{lem:4.2} and \ref{lem:mps:5.1:add} and ideas from the proof of Lemma \ref{lem:5.5a}. Then transfer that to the following proposition using the same techniques as in the proof of Proposition \ref{prop:5.8}. For details we refer to Lemma 5.9 in \cite{mp:11} or to \cite{rippl:12}.

\begin{proposition}\label{prop:5.11}
 Let $0\leq m \leq \bar{m}+1$ and assume that $(P_m)$ holds. Then for all $n\in \NN, \eta_1\in (0,1/2], \eps_0\in(0,1), K\in \NN^{\geq K_1}, \beta\in[0,1/2]$ there is an $N_{\ref{prop:5.11}}=N_{\ref{prop:5.11}}(m,n,\eta_1,\eps_0,K,\beta)(\omega) \in \NN^{\geq 2}$ almost surely such that  for all $N \geq N_{\ref{prop:5.11}}, (t,x) \in Z(N,n,K,\beta)$ and $0\leq t-s \leq N^{-8/\eta_1}$ it holds that for $l=1,\dots,\dimension$:
 \begin{align*}
  |F_{a_n,l} & (s,t,x) -  F_{a_n,l} (t,t,x)|  \leq  2^{-78} \dimension^{-1} a_n^{-\eps_0} \big[ 2^{-N(1-\eta_1)} (a_n^{1/2}\vee 2^{-N})^{(\gamma \gamma_m -1-\alpha /2)\wedge 0 } \\
 & \quad   + 2^{N\eta_1} a_n^{-\alpha/4} (\tfrac{2^{-N}}{\sqrt{a_n}} +1) ( 2^{-N\gamma \tilde{\gamma}_m} + a_n^{\beta \gamma} (\sqrt{a_n} \vee 2^{-N})^\gamma ) \\
  &  \quad +  (t-s)^{(1-\eta_1)/2} ((\sqrt{t-s} \vee \sqrt{a_n})^{\gamma \tilde{\gamma}_m - 1 -\alpha/2} + a_n^{\beta \gamma}(\sqrt{t-s} \vee \sqrt{a_n})^{\gamma-1-\alpha/2}) \big].
 \end{align*}
Moreover, $N_{\ref{prop:5.11}}$ is stochastically bounded uniformly in $(n,\beta)$.
\end{proposition}
There is a similar result for $G_\delta$ in a special case ($(t,x) = (t',x')$), which will be needed later:
\begin{proposition}\label{prop:5.12}
 Let $0 \leq m \leq \bar{m} +1$ and assume that $(P_m)$ holds. For any $n \in \NN, \eta_1 \in (0,\frac{1}{2}), \eps_0 \in (0,1)$, $K\in \NN^{\geq K_1}$, $\lambda \in [0,1]$ and $\beta \in [0,1/2]$, there is an $N_{\ref{prop:5.12}} = N_{\ref{prop:5.12}}(m,n,\eta_1,\eps_0,K,\lambda,\beta) \in \NN$ a.s. such that for all $N\geq N_{\ref{prop:5.12}}$, $(t,x)\in Z(N,n,K,\beta)$, $s\leq t$ and $\sqrt{t-s} \leq 2^{-N}$,
  \begin{align*}
  |G_{a_n^\lambda} (s,t,x) -  G_{a_n^\lambda} (t,t,x)|  \leq   2^{-95} a_n^{-\eps_0-\lambda \alpha/4} (t-s)^{(1-\eta_1)/2} & \big[ (a_n^{\lambda/2}\vee 2^{-N})^{\gamma \gamma_m} \\
  & \quad \quad + a_n^{ \beta \gamma} (\sqrt{a_n} \vee 2^{-N})^\gamma \big].
 \end{align*}
Moreover, $N_{\ref{prop:5.12}}$ is stochastically bounded uniformly in $(n,\lambda,\beta)$.
\end{proposition}
The proof of this result is similar, even easier than the proof leading to Proposition \ref{prop:5.8} and is omitted here, but details can be found in Section 9.8 in \cite{rippl:12}.

Recall that the goal of this section was to do the induction step of $(P_m)$, i.e. to get good H\"older estimates on $u=u_{1,\delta} + u_{2,\delta}$. Now we are able to give a result for $u_{1,\delta}$.\\
\noindent
\textbf{Notation:} Let
 \begin{align*}  \bar{\Delta}_{u_1} (m,n,\lambda,\eps_0,2^{-N}) = a_n^{-\eps_0- \lambda(1+\alpha/2)/2}  \big[ a_n^{\beta +\lambda(1+\alpha/2)/2} +  (a_n^{\lambda/2} \vee 2^{-N})^{\gamma \tilde{\gamma}_m +1} &  \\
   + \1 (m\geq \bar{m}) a_n^{\beta+\lambda(1+\alpha/2)/2}  + a_n^{\beta\gamma} (a_n^{\lambda/2} \vee & 2^{-N})^{\gamma +1} \big] \end{align*}
and
\[ N_{\ref{prop:5.13}}'(\eta) = \min \{N \in \NN:  2^{1-N} \leq N^{-8/\eta} \} \]
for a constant $\eta >0$.

\begin{proposition}\label{prop:5.13}
 Let $0\leq m \leq \bar{m}+1$ and assume that $(P_m)$ holds. For all $n\in \NN, \eta_1\in (0,\frac{1}{2} \wedge (1-\frac{\alpha}{2})), \eps_0, \eps_1 \in(0,1), K\in \NN^{\geq K_1}$, $\lambda \in [\eps_1, 1-\eps_1]$, $\beta\in[0,1/2]$ there is an $N_{\ref{prop:5.13}}=N_{\ref{prop:5.13}}(m,n,\eta_1,\eps_0,K,\lambda,\beta)(\omega) \in \NN^{\geq 2}$ almost surely such that for all $N \geq N_{\ref{prop:5.13}}$ and $ n, \lambda$ that satisfy
 \begin{equation}\label{eq:proof:5.13:1}
  a_n \leq 2^{-2(N_{\ref{prop:5.11}} +1)} \wedge 2^{-2(N_{\ref{prop:5.13}}'( \eta_1 \eps_1) +1)} \text{ and } \lambda \geq \eps_1,
 \end{equation}
and $ (t,x) \in Z(N,n,K,\beta)$, $t' \leq T_K, x' \in \Rdim$ s.t.~$d((t,x),(t',x')) \leq 2^{-N}$ it holds that
 \begin{align*}
  |u_{1,a_n^\lambda} (t,x) -  u_{1,a_n^\lambda} (t',x')|  \leq &  2^{-90} d((t,x),(t',x'))^{1-\eta_1} \bar{\Delta}_{u_1} (m,n,\lambda,\eps_0,2^{-N}).
 \end{align*}
Moreover, $N_{\ref{prop:5.13}}$ is stochastically bounded uniformly in $(n,\lambda,\beta)$.
\end{proposition}
We leave out the proof since it is really the same as the proof of Proposition 5.13 in \cite{mp:11}, but present the key idea. By definition
\begin{align*}
 |u_{1,a_n^\lambda} (t',x') -  u_{1,a_n^\lambda} (t,x)| & \leq |u_{1,a_n^\lambda} (t',x') -  u_{1,a_n^\lambda} (t',x)| + |u_{1,a_n^\lambda} (t',x) -  P_{t-t'}(u_{1,a_n^\lambda} (t',\cdot))(x)| \nonumber \\
 & \quad + |G_{a_n^\lambda}(t',t,x) - G_{a_n^\lambda}(t,t,x)| \\
 & \equiv T_1 + T_2 + T_3.
\end{align*}
Then use Corollary \ref{cor:5.9} for $T_1$ and $T_2$ and Proposition \ref{prop:5.12} for $T_3$ to get the result.

We also would like to obtain a similar result for $u_{2,\delta}$. We omit its proof which is simpler than the previous calculations.
In the statement of the result we use the following abbreviations:\\
\textbf{Notation:}
\begin{align*}
 \bar{\Delta}_{1,u_2} (m,n,\eps_0,2^{-N}) & = a_n^{-\eps_0} 2^{-N\gamma} [(a_n^{1/2} \vee 2^{-N})^{\gamma (\tilde{\gamma}_m -1)} + a_n^{\beta \gamma} ], \\
 \bar{\Delta}_{2,u_2} (m,n,\lambda, \eps_0) & = a_n^{-\eps_0} [
a_n^{\frac{\lambda}{2} (\gamma \tilde{\gamma}_m - \alpha/2)} + a_n^{\beta \gamma} a_n^{\frac{\lambda}{2}(\gamma- \alpha/2)} ].
\end{align*}
\begin{proposition}\label{prop:5.14}
  Let $0\leq m \leq \bar{m}+1$ and assume that $(P_m)$ holds. For all $n\in \NN, \eta_1\in (0,\frac{\alpha}{4}\wedge (1-\frac{\alpha}{2}) ), \eps_0 \in(0,1), K\in \NN^{\geq K_1}$, $\lambda \in [0,1]$, $\beta\in[0,\frac{1}{2}]$ there is an $N_{\ref{prop:5.14}}=N_{\ref{prop:5.14}}(m,n,\eta_1,\eps_0,K,\lambda,\beta)(\omega) \in \NN$ almost surely such that for all $ N \geq N_{\ref{prop:5.14}}$, $ (t,x) \in Z(N,n,K,\beta), t' \leq T_K:$
 \begin{align*}
  & d:= d((t,x),(t',x')) \leq 2^{-N} \text{ implies that }\\
  & |u_{2,a_n^\lambda} (t,x) -  u_{2,a_n^\lambda} (t',x')|  \leq   2^{-94} (d^{(1-\alpha/2)(1-\eta_1)} \bar{\Delta}_{1,u_2} + d^{1-\eta_1} \bar{\Delta}_{2,u_2}).
 \end{align*}
 Moreover, $N_{\ref{prop:5.14}}$ is stochastically bounded uniformly in $(n,\lambda,\beta)$.
\end{proposition}
We omit the proof since the ideas are the same as those leading to Proposition \ref{prop:5.8} and  many steps are analogous to Section 7 in \cite{mp:11}. Now we are ready to complete this section by proving the induction step:\\

\noindent
\textbf{Proof of Proposition \ref{prop:5.1}:}\label{pr:prop:5.1:part2}\\
Let $0 \leq m \leq \bar{m}$ and assume $(P_m)$. We want to show $(P_{m+1})$. Let therefore $\eps_0 \in (0,1)$, $M = \lceil \frac{2}{\eps_0} \rceil, \eps_1 = \frac{1}{M} \leq  \eps_0/2,$ and $\lambda_i = i \eps_1$ for $i = 1, \dots, M$. Then clearly $\lambda_i \in [\eps_1,1]$ for all $i = 1, \dots, M$.

Let $n, \xi, K, \beta$ be as in $(P_m)$ and w.l.o.g. $\xi$ is sufficiently large such that  for $\eta_1 := 1-\xi$ we have $\eta_1 < \tfrac{\alpha}{4} \wedge (1-\tfrac{\alpha}{2})$. Set $\xi' = (1+\xi)/2 \in (\xi, 1)$ and
\begin{align*}
 N_2 (m,n,\xi,\eps_0,K,\beta)(\omega) & = \bigvee_{i = 1}^M N_{\ref{prop:5.13}}(m,n,\eta_1, \eps_0/2, K+1, \lambda_i, \beta), \\
 N_3 (m,n,\xi,\eps_0,K,\beta)(\omega) & = \bigvee_{i = 1}^M N_{\ref{prop:5.14}}(m,n,\eta_1, \eps_0/2, K+1, \lambda_i, \beta), \\
 N_4 (m,n,\xi,\eps_0,K,\beta)(\omega) & = \lceil \frac{2}{1-\xi} \rceil (N_{\ref{prop:5.11}} \vee N_{\ref{prop:5.13}'} + 1)  =: \frac{1}{1- \xi} N_5, \\
 N_1 & = N_2 \vee N_3 \vee N_4 \vee N_0(\xi',K) + 1 \in \NN,
\end{align*}
where $N_0(\xi',K)$ is the constant we obtained from Theorem \ref{thm:collipmod}. By the results on each of the single constants we know that $N_1$ is then stochastically bounded uniformly in $(n,\beta)$. Let now 
$$N \geq N_1, (t,x) \in Z(N,n,K,\beta), t'\leq T_K \text{ and } d := d((t,x),(t',x')) \leq 2^{-N}.$$
There are two cases for the values of $n$ to consider.
We start with small $n$:
\begin{equation} a_n > 2^{-N_5(m,n,\eta_1,\eps_0,K,\beta)}, \end{equation}
which implies $\sqrt{a_n}^{\tilde{\gamma}_{m+1} -1} \geq a_n^{1/2} \geq 2^{-N_5/2}$. As $N \geq N_0(\xi',K)$ we get by $(P_0)$ in the case $\eps_0=0,$
\begin{align*}
 |u(t',x')| & \leq 2^{-N\xi'} \\
 & \leq 2^{-N\xi'}[ (\sqrt{a_n}\vee 2^{-N})^{\tilde{\gamma}_m + 1 } ] 2^{N_5/2} \\
 & \leq 2^{-N(1-\xi)/2} 2^{N_5/2} 2^{-N\xi} [ (\sqrt{a_n}\vee 2^{-N})^{\tilde{\gamma}_m + 1 } + a_n^\beta ] \\
 & \leq 2^{-N\xi} [ (\sqrt{a_n}\vee 2^{-N})^{\tilde{\gamma}_m + 1 } + a_n^\beta ].
\end{align*}
This already completes the first case.
For large $n$,
\begin{equation} a_n \leq 2^{-N_5}  , \label{eq:pr:5.1:case2} \end{equation}
let $N' = N - 1 \geq N_2 \vee N_3$, which gives $(t',x')\in Z(N',n,K+1,\beta)$ by the triangle inequality. As \eqref{eq:pr:5.1:case2} holds, we can apply Proposition \ref{prop:5.13} with $(\eps_0/2,K+1)$ instead of $(\eps_0, K)$. Additionally,  use Proposition \ref{prop:5.14}.
So we can estimate $|u(\hat{t}_0,\hat{x}_0)-u(t',x')|$. Before doing so we have to choose which partition with $\delta =a_n^\lambda$ of $u$ to take in the sense of \eqref{eq:u=u1+u2} and then obtain some estimates. Therefore select $i\in \{1, \dots , M\}$ such that the following holds:
\begin{enumerate}
 \item \label{case:5.1:1} If $2^{-N} > \sqrt{a_n}$ then $a_n^{\lambda_i/2} < 2^{-N'} \leq a_n^{\lambda_{i-1}/2} = a_n^{\lambda_i/2} a_n^{-\eps_1/2}$.
 \item \label{case:5.1:2} If $2^{-N} \leq \sqrt{a_n}$ then $i = M$ and hence $a_n^{\lambda_i/2} = \sqrt{a_n} \geq 2^{-N'}$.
\end{enumerate}
Then in both cases
\begin{equation*}
 a_n^{\lambda_i/2}  \vee 2^{-N'} \leq \sqrt{a_n} \vee 2^{-N'} \label{eq:5.1:2} 
 \end{equation*}
 and writing $\lambda=\lambda_i$ leads to
 \begin{align}
 a_n^{-\frac{\lambda}{2}(1+\alpha/2)} & (\sqrt{a_n} \vee 2^{-N'})^{1+\alpha/2} \leq (a_n^{-\frac{\lambda}{2}} (\sqrt{a_n} \vee 2^{-N'}))^{1+\alpha/2} \nonumber \\
&\leq  \nonumber \left\{\begin{array}{ll}  (a_n^{-\lambda/2} 2^{-N'})^{1+\alpha/2} \leq a_n^{-\frac{\eps_1}{2}(1+\alpha/2)} &\text {in case \eqref{case:5.1:1}}, \\
 (a_n^{-\lambda/2} a_n^{1/2})^{1+\alpha/2} = 1 \leq a_n^{-\frac{\eps_1}{2}(1+\alpha/2)} &  \text{in case \eqref{case:5.1:2}.}
 \end{array}\right.
\end{align}
 Hence in both cases we obtain a bound of 
\begin{equation}
 a_n^{-\frac{\lambda}{2}(1+\alpha/2)} (\sqrt{a_n} \vee 2^{-N'})^{1+\alpha/2} \leq a_n^{-\frac{\eps_1}{2}(1+\alpha/2)}. \label{eq:5.1:3}
 \end{equation}
Furthermore we have by $\xi - \alpha/2 \leq (2-\alpha)\tfrac{\xi}{2},$
\begin{align*}
 2^{-N'(2-\alpha)\frac{\xi}{2}} & \leq 2^{-N'\xi}2^{N'\frac{\alpha}{2}} \text{ and}\\
 2^{-N'((2-\alpha)\frac{\xi}{2} + \gamma)} &= 2^{-N'\xi} 2^{-N'(\gamma-\alpha/2)} \\
 & \leq 2^{-N'\xi} (a_n^{1/2} \vee 2^{-N'})^{\gamma-\alpha/2} \quad \quad ( \gamma > \alpha/2).
\end{align*}
Using the aforementioned propositions, the special $i$, \eqref{eq:5.1:2} and the previous lines we get
\begin{align*}
 |&u (\hat{t}_0,\hat{x}_0)  -u(t',x')| \\ &\leq |u_{1,a_n^{\lambda_i}}(\hat{t}_0,\hat{x}_0) - u_{1,a_n^{\lambda_i}}(t,x) | + |u_{2,a_n^{\lambda_i}}(\hat{t}_0,\hat{x}_0) - u_{2,a_n^{\lambda_i}}(t,x)| \\
 & \leq 2^{-90} a_n^{-\eps_0/2} 
  2^{-N'\xi}[a_n^\beta  + a_n^{ -\frac{\lambda}{2}(1+\alpha/2)} a_n^{\beta \gamma}  ( (a_n^{\frac{\lambda}{2}} \vee 2^{-N'})^{\gamma+1} + a_n^{- \frac{\lambda}{2}(1+\alpha/2)}   (a_n^{\frac{\lambda}{2}} \vee 2^{-N'})^{\gamma \tilde{\gamma}_m +1} \\
  & \phantom{AAAAAAAAAAAAAAAAAAAAAAA} +\1(m=\bar{m})(a_n^{\frac{\lambda}{2}} \vee 2^{-N})]\\
  & \quad  + 2^{-94} a_n^{-\eps_0/2} \big[ 2^{-N'((2-\alpha)\frac{\xi}{2} + \gamma)} \big( (a_n^{1/2} \vee 2^{-N'})^{\gamma( \tilde{\gamma}_m-1)} + a_n^{\beta \gamma} \big) \\
  & \phantom{AAAAAAAA}+ 2^{-N'\xi} \big( (a_n^{1/2} \vee 2^{-N'})^{\gamma \tilde{\gamma}_m- \alpha/2} + a_n^{\beta \gamma} (a_n^{1/2} \vee 2^{-N'})^{\gamma-\alpha/2} \big)\big] \\
  & \leq 2^{-89} a_n^{-\eps_0/2} 2^{-N' \xi} \Big[a_n^\beta + a_n^{ -\lambda(1+\alpha/2)/2} a_n^{\beta \gamma} (a_n^{1/2} \vee 2^{-N})^{1 + \alpha/2}   (a_n^{1/2} \vee 2^{-N})^{\gamma-\alpha/2} \\
  & \qquad  + a_n^{- \lambda(1+\alpha/2)/2}   (a_n^{1/2} \vee 2^{-N})^{1+\alpha/2} (a_n^{1/2} \vee 2^{-N})^{\gamma \tilde{\gamma}_m - \alpha/2} +\1(m=\bar{m})(a_n^{1/2} \vee 2^{-N}) \\
  & \qquad   + (a_n \vee 2^{-N'})^{\gamma-\alpha/2} ( (a_n^{1/2} \vee 2^{-N'})^{\gamma( \tilde{\gamma}_m-1)} + a_n^{\beta \gamma} ) + \\
  & \qquad  + ( (a_n^{1/2} \vee 2^{-N'})^{\tilde{\gamma}_{m+1}} + a_n^{\beta \gamma} (a_n^{1/2} \vee 2^{-N'})^{\gamma-\alpha/2} ) \Big]. 
\end{align*}
Now apply \eqref{eq:5.1:3} to bound this by
\begin{align*}
 &  2^{-89} a_n^{-\eps_0/2} a_n^{-\frac{\eps_1}{2}(1+\alpha/2)} 2^{-N'\xi} \Big[ a_n^\beta a_n^{\frac{\eps_1}{2}(1+\alpha/2)} 
 + a_n^{\beta \gamma} (\sqrt{a_n}\vee 2^{-N})^{\gamma-\alpha/2} (2a_n^{\frac{\eps_1}{2}(1+\alpha/2)}+1) \\
 &  \phantom{AAAAAAAAAAAA}+\1(m=\bar{m})(a_n^{1/2} \vee 2^{-N}) + (\sqrt{a_n}\vee 2^{-N})^{\tilde{\gamma}_{m+1}} ( 1 + 2 a_n^{\frac{\eps_1}{2}(1+\alpha/2)}) \Big]. 
 \end{align*}
 Trivially,
 \[ (\sqrt{a_n}\vee 2^{-N})^{\tilde{\gamma}_{m+1}-1} + \1(m=\bar{m})(\sqrt{a_n}\vee 2^{-N}) \leq 2 (\sqrt{a_n} \vee 2^{-N}). \]
Then, since $\eps_0 \leq \eps_1/2$ and $\eps_1 \geq 0$,
 \begin{align*}
 |u (\hat{t}_0,\hat{x}_0)  -u(t',x')| \leq  2^{-87} 2^{-N'\xi} a_n^{-\eps_0} \big[& a_n^\beta + a_n^{\beta \gamma} (\sqrt{a_n}\vee 2^{-N})^{\gamma  -\alpha/2} \\
 & \quad  + (\sqrt{a_n}\vee 2^{-N})^{\tilde{\gamma}_{m+1}} \big].
\end{align*}
Now we can proceed to the last step of this statement and obtain
\begin{align*}
 |& u  (t',x')|  \leq |u(\hat{t}_0,\hat{x}_0)| + |u(\hat{t}_0,\hat{x}_0)  -u(t',x')| \\
            & \leq \sqrt{a_n}2^{-N} + 2^{-85} 2^{-N\xi} a_n^{-\eps_0}[a_n^\beta + a_n^{\beta \gamma} (\sqrt{a_n}\vee 2^{-N})^{\gamma  -\alpha/2} + (\sqrt{a_n}\vee 2^{-N})^{\tilde{\gamma}_{m+1}} ] \\
            & \leq a_n^{-\eps_0} 2^{-N\xi} [\sqrt{a_n}2^{-N(1-\xi)} + 2^{-85} (a_n^\beta + a_n^{\beta \gamma} (2^{-N}\vee \sqrt{a_n})^{\gamma-\alpha/2}+ (2^{-N}\vee \sqrt{a_n})^{\tilde{\gamma}_{m+1}} ].
\end{align*}
Clearly, $\sqrt{a_n}2^{-N(1-\xi)} \leq \sqrt{a_n}/2 \leq a_n^\beta /2$ and by \eqref{eq:pr:5.1:case2} and an easy calculation (see Lemma 5.15 in \cite{mp:11}) we arrive at
$$a_n^{\beta \gamma}(2^{-N}\vee \sqrt{a_n})^{\gamma-\alpha/2} \leq a_n^\beta \vee 2^{-N} \leq a_n^\beta + (2^{-N}\vee \sqrt{a_n})^{\gamma \tilde{\gamma}_m - \alpha/2}.$$
Therefore, we can write
\begin{align*}
 |u(t',x')| & \leq a_n^{-\eps_0} 2^{-N\xi}( \frac{a_n^\beta}{2} + 2^{-84} a_n^\beta +2^{-84} (2^{-N}\vee \sqrt{a_n})^{\gamma \tilde{\gamma}_m - \alpha/2} )\\
            & \leq  a_n^{-\eps_0} 2^{-N\xi} (a_n^\beta + (2^{-N}\vee \sqrt{a_n})^{\tilde{\gamma}_{m+1}-1} ),
\end{align*}
since
\begin{equation}
 \tilde{\gamma}_{m+1} = (\gamma \gamma_m + 1 - \frac{\alpha}{2}) \wedge 2.
\end{equation}
This completes the proof of Proposition \ref{prop:5.1}.

\section{Proof of Proposition \ref{prop:3.3}}\label{sec:6}

Fix $K_0 \in \NN^{\geq K_1}$, $\eps_0,\eps_1 \in (0,1)$ as in the definition \eqref{eq:epsconditions} and for  $0 < \beta \leq 1- \eps_1$ define
\be \lambda(\beta) := 2(\beta + \eps_1) \in [0,1].\eq
We define four collections of random times the first one being
\begin{align}
\nonumber
 U_{M,n,\beta}^{(1)} = \inf \{& t \geq 0: \exists \eps \in [0,2^{-M}], |x| \leq K_0 +1, \hat{x}_0 , x' \in \RR^{\dimension}, \text{such that }  | x-x'| \leq 2^{-M}, \\
 &\nonumber |\hat{x}_0 -x| \leq \eps, |u(t,\hat{x}_0)| \leq a_n \wedge (\sqrt{a_n} \eps), \, |\nabla u_{1,a_n}(t,\hat{x}_0)| \leq a_n^\beta, \text{ and }\\
 &\nonumber |\nabla u_{1,a_n^\lambda}(t,x) - \nabla u_{1,a_n^\lambda}(t,x')| > 2^{-82} a_n^{-\eps_0 - \eps_1(1+\alpha/2)} |x-x'|^{1-\eps_0} \\
 &\nonumber \qquad \quad [a_n^{-\beta(1+\alpha/2)} (\eps \vee |x-x'|)^{2 \gamma} + 1 + a_n^{\beta(\gamma-1-\alpha/2)}(\eps \vee |x'-x|)^\gamma ]\}\\
 \wedge T_{K_0} &,
\end{align}
whenever $M,n \in \NN$, $\beta >0$. We define $U_{M,n,0}^{(1)}$ in the same way, omitting the condition on $|\nabla u_{1,a_n}(t,\hat{x}_0)|$. These random times are actually stopping times by Theorem IV.T.52 of \cite{pM:66}.

\begin{lemma}\label{lem:6.1}
 For all $n \in \NN$, $\beta$ as in \eqref{eq:3.12} it holds that $U_{M,n,\beta}^{(1)} \nearrow T_{K_0}$ almost surely as $M \to \infty$ and
 $$\lim_{M\to \infty} \sup_{n,0\leq \beta \leq 1/2 -\eps_1} \IP[U_{M,n,\beta}^{(1)} < T_{K_0}] = 0.$$
\end{lemma}

\begin{proof}
The almost sure convergence follows from the second statement by monotonicity of the probabilities in $M$. For the second statement use Corollary \ref{cor:5.9} with $m = \bar{m}+1$ (justified by Proposition \ref{prop:5.1}), $\eta_1 = \eps_0$, $K=K_0+1$. Then there is an $N_0 = N_0 (n,\eps_0, \eps_1, K_0 +1, \beta)$ stochastically bounded in $n$ and $\beta$ such that  for all $N\geq N_0$, $(t,x) \in Z(N,n,K_0+1,\beta)$ and $|x-x'| \leq 2^{-N}$ it holds that
 \begin{align*}
  |\nabla u_{1,a_n^\lambda} (t,x) - \nabla u_{1,a_n^\lambda} (t,x')| & \leq 2^{-85} |x-x'|^{1-\eps_0} a_n^{-\eps_0} [ a_n^{-(\beta +\eps_1)(1+\alpha/2)} 2^{-2N\gamma} + (a_n^{\beta + \eps_1} \vee 2^{-N})^0\\
  & \qquad \quad +  a_n^{-(\beta +\eps_1)(1+\alpha/2)} a_n^{\beta \gamma} (a_n^{\beta + \eps_1} \vee 2^{-N})^\gamma ] \\
  & \leq 2^{-85} |x-x'|^{1-\eps_0} a_n^{-\eps_0 -\eps_1(1+\alpha/2)} \\
  & \qquad [a_n^{-\beta(1+\alpha/2)} 2^{-2N\gamma} + 1 + a_n^{\beta( \gamma - 1-\alpha/2)}(a_n^{\beta + \eps_1} \vee 2^{-N})^\gamma ].
 \end{align*}
Note that by $ \beta(\gamma -1-\alpha/2) + \beta \gamma + (\beta +\eps_1) \gamma = \beta(2\gamma - \alpha/2 -1) + \eps_1 \gamma >0$ we now obtain
\begin{align*}
  |\nabla u_{1,a_n^\lambda} (t,x) - \nabla u_{1,a_n^\lambda} (t,x')| & \leq 2^{-84} a_n^{-\eps_0 - \eps_1 (1+\alpha/2)} |x-x'|^{1-\eps_0} \\
  & \qquad [a_n^{-\beta(1+\alpha/2)}2^{-2N\gamma} + 1 +a_n^{\beta(\gamma -1-\alpha/2)} 2^{-N\gamma}].
\end{align*}
We only do the case $\beta >0$. Assume that $M>N_0$ and that there is a $t< T_{K_0}$, $\eps \in [0,2^{-M}]$, $|x| \leq K_0 + 1$, $\hat{x}_0, x' \in \RR^{\dimension}$ with $|x-x'| \leq 2^{-M}, |\hat{x}_0 -x| \leq \eps$, $|u(t,\hat{x}_0)| \leq a_n \wedge (\sqrt{a_n} \eps),$ and
$$ |\nabla u_{1,a_n} (t,\hat{x}_0)| \leq a_n^\beta.$$
If $x\neq x'$, then there is a $N \geq N_0$ such that  $2^{-N-1} < |x-x'| \vee \eps \leq 2^{-N} \leq 2^{-N_0}$. Then $(t,x) \in Z(N,n,K_0+1, \beta_i)$ and we can use the previous estimate. Hence,
\begin{align*}
 |\nabla u_{1,a_n^\lambda} (t,x) - & \nabla u_{1,a_n^\lambda} (t,x')|  \leq 2^{-82} a_n^{-\eps_0 - \eps_1 (1+\alpha/2)} |x-x'|^{1-\eps_0} \\
  & \qquad [a_n^{-\beta(1+\alpha/2)}(\eps \vee |x-x'|)^{2\gamma} + 1 +a_n^{\beta(\gamma -1-\alpha/2)} (\eps \vee |x-x'|)^{\gamma}].
\end{align*}
Therefore, $U_{M,n,\beta}^{(1)} = T_{K_0}$ and thus $\IP(U_{M,n,\beta}{(1)} < T_{K_0}) = \IP(M<N_0).$ As $N_0$ is stochastically bounded uniformly in $(n,\beta)$ the assertion follows.
\end{proof}

\begin{remark}
We note that the fact that we consider splitting at $\delta = a_n^\lambda$ rather than $\delta = a_n$ is essential in the previous lemma. 
\end{remark}
Let us define more stopping times, this time with $u_2$. For $0< \beta \leq 1/2 - \eps_1$ set
\begin{align}
\nonumber
 U_{M,n,\beta}^{(2)} = \inf \big\{& t \geq 0: \exists \eps \in [0,2^{-M}], |x| \leq K_0 +1, \hat{x}_0 , x' \in \RR^{\dimension}, \text{ such that } | x-x'| \leq 2^{-M}, \\
 \nonumber& |\hat{x}_0 -x| \leq \eps, |u(t,\hat{x}_0)| \leq a_n \wedge (\sqrt{a_n} \eps), \, |\nabla u_{1,a_n} (t,\hat{x}_0)| \leq a_n^\beta, \text{ and }\\
\nonumber & |u_{2,a_n^\lambda} (t,x) - u_{2,a_n^\lambda}(t,x')| > 2^{-87} a_n^{-\eps_0} [ |x-x'|^{(1-\eps_0)(1-\alpha/2)} \\
\nonumber & \qquad \quad [(\sqrt{a_n} \vee \eps \vee |x'-x|)^{2\gamma} + a_n^{\beta \gamma} ( \eps \vee |x'-x|)^\gamma ] \\
 & \qquad +|x-x'|^{1-\eps_0} a_n^{\beta + \eps_1(1-\gamma)} ] \big\} \wedge T_{K_0}.
\end{align}
And in the case $\beta = 0$ we make the same definition but without the condition on $|\nabla u_{1,a_n}(t,\hat{x}_0)|.$ Then, we get
\begin{lemma}\label{lem:6.2}
 For all $n \in \NN$, $\beta$ as in \eqref{eq:3.12} it holds that $U_{M,n,\beta}^{(2)} \nearrow T_{K_0}$ almost surely as $M \to \infty$ and
 $$\lim_{M\to \infty} \sup_{n,0\leq \beta \leq 1/2 -\eps_1} \IP[U_{M,n,\beta}^{(2)} < T_{K_0}] = 0.$$
\end{lemma}
As the proof is similar to the one of Lemma \ref{lem:6.1} this time using Proposition \ref{prop:5.14} instead of Corollary \ref{cor:5.9} we omit it. Define
\begin{align*}
 \tilde{\Delta}_{u_1'}(n,\eps_1,\eps_0,\beta) = a_n^{-\eps_0} \eps^{-\eps_0} \{ \eps + a_n^{-\alpha/4} (\eps a_n^{-1/2} + 1) (\eps^{2\gamma} + a_n^{\beta \gamma} (\eps \vee \sqrt{a_n})^\gamma )\}
\end{align*}
and for $0< \beta \leq 1/2 - \eps_1$ set
\begin{align*}
 U_{M,n,\beta}^{(3)} = &\inf \big\{ t \geq 0: \exists \eps \in [2^{-a_n^{-(\beta + \eps_1)\eps_0/4}},2^{-M}], |x| \leq K_0 +1, \hat{x}_0 \in \RR^{\dimension}, \\
  & \qquad |\hat{x}_0 -x| \leq \eps, |u(t,\hat{x}_0)| \leq a_n \wedge (\sqrt{a_n} \eps), \, |\nabla u_{1,a_n} (t,\hat{x}_0)| \leq a_n^\beta, \text{ and }\\
 & \qquad |\nabla u_{1,a_n^\lambda}(t,x) - \nabla u_{1,a_n}(t,x)| > 2^{-74} (\tilde{\Delta}_{u_1'}(n,\eps_1,\eps_0,\beta) + a_n^{\beta + \eps_1 (1-\gamma)/4}) \big\} \\
 & \wedge T_{K_0}.
\end{align*}
In the case $\beta =0$ we make the analogous definition  without the condition on $ |\nabla u_{1,a_n}(t,\hat{x}_0)|$. Again those are stopping times, and we obtain the analogous statement:
\begin{lemma}\label{lem:6.3}
 For all $n \in \NN$, $\beta$ as in \eqref{eq:3.12} it holds that $U_{M,n,\beta}^{(3)} \nearrow T_{K_0}$ almost surely as $M \to \infty$ and
 $$\lim_{M\to \infty} \sup_{n,0\leq \beta \leq 1/2 -\eps_1} \IP[U_{M,n,\beta}^{(3)} < T_{K_0}] = 0.$$
\end{lemma}
The proof of this lemma requires Proposition \ref{prop:5.11} and structurally equals the one of Lemma \ref{lem:6.1}. As the fourth collection of stopping times define
\begin{align}
\nonumber
 U_M^{(4)} = \inf \{& t \geq 0: \exists \eps \in [0,2^{-M}], |x| \leq K_0 +1, \hat{x}_0 , x' \in \RR^{\dimension}, |x-x'|\leq 2^{-M}, |x-\hat{x}_0| < \eps \\
 & |u(t,\hat{x}_0) | \leq \eps, |u(t,x) - u(t,x')| > (\eps \vee |x'-x|)^{1-\eps_0} \} \wedge T_{K_0}.
\end{align}
\begin{lemma}\label{lem:6.4}
 Almost surely $U_M^{(4)} \nearrow T_{K_0}$ as $M \to \infty$.
\end{lemma}
This proof uses Theorem \ref{thm:collipmod} and is similar to the proof of Lemma \ref{lem:6.1}, so it is omitted here.
Finally define the stopping times for \ref{prop:3.3}:
\begin{align*}
 U_{M,n,\beta} &= \bigvee_{j=1}^3 U_{M,n,\beta}^{(j)} \\
 U_{M,n} &= ( \bigvee_{i=0}^{L(\eps_0,\eps_1)} U_{M,n,\beta_i} ) \wedge U_M^{(4)}.
\end{align*}
By Lemmas \ref{lem:6.1}, \ref{lem:6.2}, \ref{lem:6.3}, \ref{lem:6.4} we have that $U_{M,n}$ fulfills $(H_1)$. Hence there is not much left to do in order to complete the proof of Proposition \ref{prop:3.3}. It just remains to show the compactness of $\tilde{J}_{n,i}(s)$ and $\tilde{J}_{n,i}(s) \supset J_{n,i}(s)$ for all $s < U_{M,n}$. We will be mostly concerned with $\tilde{J}_{n,i}(s) \supset J_{n,i}(s)$, show that in several steps and assume \eqref{eq:3.14} throughout the rest of the section, i.e.
\begin{equation}\label{eq:6} a_n^{\eps_1} \leq 2^{-M-4} \text{ and } \sqrt{a_n} \geq 2^{-a_n^{-\eps_0 \eps_1/4}} .\end{equation}
We first give a list of three lemmas that are analogous to Lemmas 6.5, 6.6 and 6.7 of \cite{mp:11}. As the proofs are quite similar we only show the last lemma since it also contains a slight improvement of Lemma 6.7 of \cite{mp:11}.
\begin{lemma} \label{lem:6.5}
 When $i \in \{0, \dots, L\}, 0 \leq s < U_{M,n}$, $x \in J_{n,i}(s)$, then
 \begin{enumerate}
  \item $|\nabla u_{1,a_n}(s,\hat{x}_n(s,x)) - \nabla u_{1,a_n^{\lambda_i}}(s,\hat{x}_n(s,x))| \leq 2^{-71} a_n^{\beta_i + \eps_1 (1-\gamma)/2}.$
  \item For $i>0$: $\nabla u_{1,a_n^{\lambda_i}}(s,\hat{x}_n(s,x)) \cdot \sigma_x \leq | u_{1,a_n^{\lambda_i}}(s,\hat{x}_n(s,x)) |  \leq a_n^{\beta_i}/2 $.
  \item For $i<L$: $\nabla u_{1,a_n^{\lambda_i}}(s,\hat{x}_n(s,x)) \cdot \sigma_x  \geq a_n^{\beta_{i+1}}/8 $.
 \end{enumerate}
\end{lemma}
The proof is done using $U_{M,n,\beta}^{(3)}$. Next consider the derivatives of $u_{1,a_n^\lambda}$.
\begin{lemma}\label{lem:6.6}
 When $i \in \{0, \dots, L\}, 0 \leq s < U_{M,n}$, $x \in J_{n,i}(s)$ and $|x-x'| \leq 5 \bar{\ell}_n(\beta_i)$, then
 \begin{enumerate}
  \item For $i>0$: $|\nabla u_{1,a_n^{\lambda_i}} (s,x') | \leq a_n^{\beta_i} .$
  \item For $i<L$: $\nabla u_{1,a_n^{\lambda_i}} (s,x') \cdot \sigma_x \geq a_n^{\beta_{i+1}}/16.$
 \end{enumerate}
\end{lemma}
The proof uses $U_{M,n,\beta}^{(1)}$, but is left out. To finish things we only need a similar result for the $u_2$ expressions for which we give the details of the proof:
\begin{lemma}\label{lem:6.7}
  When $i \in \{0, \dots, L\}, 0 \leq s < U_{M,n}$, $x \in J_{n,i}(s)$, $x', x'' \in \RR^\dimension$ and $|x-x'| \leq 4 \sqrt{a_n}$, then
 $$ |u_{2,a_n^{\lambda_i}} (s,x')  -u_{2,a_n^{\lambda_i}} (s,x'') | \leq 2^{-75}a_n^{\beta_{i+1}} (| x'-x''| \vee a_n^{\frac{2}{\alpha}(\gamma-\beta_{i+1} - \eps_1)} \vee a_n ) $$
 as long as $| x'-x''| \leq \bar{\ell}_n(\beta_i).$
\end{lemma}
\begin{remark}\label{rem:lem:6.7}
This lemma is stricter than Lemma 6.7 of \cite{mp:11}. Following their strategy, we would obtain $\frac{1}{\alpha}(\gamma-2\beta_i(1-\gamma)-\eps_1)$ instead of the larger exponent $\frac{2}{\alpha}(\gamma-\beta_{i+1}-\eps_1).$
\end{remark}
\begin{proof}
 Let $(i,n,s,x,x')$ be as above and $\eps = 5 \sqrt{a_n} \leq 2^{-M}$ by \eqref{eq:3.14}. Then
 \begin{align*}
  |x' - \hat{x}_n(s,x)| &\leq |x'-x| + \sqrt{a_n} \leq \eps, |x'| \leq K_0 +1 \quad \text{and}\\
  | u(s,\hat{x}_n(s,x))| &\leq a_n = a_n \wedge (\sqrt{a_n} \eps).
 \end{align*}
If $i>0$ then for $x \in J_{n,i}(s)$ we obtain $|\nabla u_{1,a_n} (s,\hat{x}_n(s,x))| \leq a_n^{\beta_i}/4 \leq a_n^{\beta_i}$.
Let
\[ Q(n,\eps_0, \beta_i, r) := a_n^{-\eps_0} r^{(1-\eps_0)(1-\alpha/2)} ((\sqrt{a_n} \vee r)^{2\gamma} + a_n^{\beta_i \gamma} r^\gamma). \]
Assume that $|x'-x''| \leq \bar{\ell}_n(\beta_i) \, ( \leq 2^{-M})$. Since $s < U_{M,n,\beta_i}^{(2)}$, it holds that
\begin{align}
 |u_{2,a_n^{\lambda_i}} (s,x')  -u_{2,a_n^{\lambda_i}} (s,x'') | & \leq 2^{-87} [ Q(n,\eps_0, \beta_i, \eps \vee |x'-x''|) + a_n^{-\eps_0} |x'-x''|^{1-\eps_0} a_n^{\beta_i + \eps_1 (1-\gamma)}] \nonumber \\
 & \leq 2^{-80} [Q(n, \eps_0, \beta_i, |x'-x''|) + |x'-x''|^{1-\eps_0} a_n^{\beta_i + \eps_1 (1-\gamma) -\eps_0}]. \label{eq:pr:lem:6.7}
\end{align}
We will now show the following\\
\emph{Claim:} 
\begin{equation}
\label{ClaimQ}
 Q(n, \eps_0, \beta_i,r) \leq 2 a_n^{\beta_{i+1}} (r \vee a_n^{\frac{2}{\alpha}(\gamma-\beta_{i+1}-\eps_1)} \vee a_n) \text{ if  }0 \leq r \leq \bar{\ell}_n(\beta_i).
 \end{equation}
 We split the proof of the claim into several cases:

\emph{Case 1:} $\sqrt{a_n} \leq r \leq \bar{\ell}_n(\beta_i) = a_n^{\beta_i + 5 \eps_1}$\\
In this case we will bound $Q$ by $2a_n^{\beta_{i+1}}r$ and this holds if
\begin{align}
 r^{(1-\eps_0)(1-\alpha/2) +2\gamma -1} \leq a_n^{\beta_{i+1} + \eps_0} \label{eq:6.22} \\
 \intertext{ and}
 r^{\gamma -1 + (1-\eps_0)(1-\alpha/2)} \leq a_n^{\beta_{i+1} + \eps_0 - \beta_i \gamma}. \label{eq:6.23}
\end{align}
Note that
$$(1-\eps_0)(1-\alpha/2) - 1 +2\gamma \geq 2 \gamma - \alpha/2 - \eps_0 \geq 1 \quad  (\text{by } \eqref{eq:3.10}).$$
Therefore \eqref{eq:6.22} already follows (additionally using $5 \eps_1 > 2 \eps_0$). For the other inequality by $r \leq a_n^{\beta_i + 5 \eps_1 }$ it suffices to show that
$$ \beta_i (1-\gamma) + 2\eps_0 -\beta_i [\gamma - 1 + (1-\eps_0)(1-\alpha/2)] - 5 \eps_1[\gamma - 1 + (1-\eps_0)(1-\alpha/2)] \leq 0. $$
This holds if
$$ \beta_i( 2\gamma -1 - \alpha/2 + \eps_0(1-\alpha/2)) - 2 \eps_0 +5 \eps_1(\gamma - \alpha/2) \geq 0.$$
And this holds since $\alpha < 2(2\gamma -1)$ and $2\eps_0 < 5 \eps_1$. So we are done with the first case.\\

\emph{Case 2:} $ a_n^{\frac{2}{\alpha}(\gamma-\beta_{i+1}-\eps_1)}  \leq r < \sqrt{a_n}$.\\
In this case
\begin{align*} Q(n,\eps_0, \beta_i, r) & = a_n^{-\eps_0} r^{(1-\eps_0)(1-\alpha/2)}[a_n^\gamma + a_n^{\gamma \beta_i } r^\gamma] . \end{align*}
We need to estimate both summands. First
$$ r^{(1-\eps_0)(1-\alpha/2)} a_n^{\gamma-\eps_0} \leq r a_n^{\beta_{i+1}}$$
is true by the lower bound on $r$ and the fact that $\eps_0 < \eps_1(1+\frac{\alpha}{2})^{-1}.$ The second summand satisfies
$$ a_n^{-\eps_0}r^{(1-\eps_0)(1-\alpha/2)}a_n^{\gamma \beta_i} r^\gamma \leq r a_n^{\beta_{i+1}},$$
since
\begin{align*}
 r^{\gamma -\alpha/2-\eps_0(1-\alpha/2)} &\leq  \sqrt{a_n}^{\gamma-\alpha/2 -\eps_0} \\
 & \leq \sqrt{a_n}^{1-\gamma +5 \eps_0} \leq a_n^{\beta_i(1-\gamma)} a_n^{2\eps_0},
\end{align*}
by $3\eps_0 < \eps_1 < \frac{1}{2}(2\gamma - 1 -\frac{\alpha}{2}).$

\emph{Case 3:} $ r < a_n^{\frac{2}{\alpha}(\gamma-\beta_{i+1} - \eps_1)}$\\
This follows from monotonicity in $r$ and the fact that Case 2 actually occurs since
$$ \frac{2}{\alpha}(\gamma - \beta_{i+1} - \eps_1) \geq \frac{2}{\alpha}(\gamma-\frac{1}{2}) > \frac{\alpha/2}{\alpha} = \frac{1}{2}.$$
Hence the claim in (\ref{ClaimQ}) is shown.

Next, consider the other term in \eqref{eq:pr:lem:6.7} to finish the proof. In the case $r\geq a_n$ we have
\begin{align*} r^{1-\eps_0} a_n^{\beta_i + \eps_1 (1-\gamma) -\eps_0}(a_n^{\beta_{i+1}} r)^{-1} & = r^{-\eps_0} a_n^{-2 \eps_0 + \eps_1 (1-\gamma)} \nonumber \\
& \leq a_n^{-3 \eps_0 + \eps_1 (1-\gamma)} < 1,\end{align*}
since $\eps_0 < \frac{1-\gamma}{3}\eps_1$. So for any $r>0$ (when using the previous estimate with $(r\vee a_n)$ in place of $r$),
\begin{align*} r^{1-\eps_0} a_n^{\beta_i + \eps_1 (1-\gamma) -\eps_0} & \leq a_n^{\beta_{i+1}} (r\vee a_n) \\ & \leq a_n^{\beta_{i+1}} (r \vee a_n^{\frac{2}{\alpha}(\gamma-\beta_{i+1}-\eps_1)} \vee a_n ).\end{align*}
Putting things together we get the statement of the lemma.
\end{proof}

\begin{lemma}\label{lem:6.8}
 If $0 \leq s < U_{M,n}$ and $x\in J_{n,0}(s)$ then
  \be | u(s,x) - u(s,x') | \leq ( \sqrt{a_n} \vee | x-x'|)^{1-\eps_0} \ \text{ if } x' \text{ is such that  } |x'-x|\leq 2^{-M} \eq
 and
  \be |u(s,x')| \leq 3(\sqrt{a_n})^{1-\eps_0} \ \text{ if } |x'-x| \leq \sqrt{a_n}. \eq
\end{lemma}
\noindent
 This statement has just the same proof as Lemma 6.8 in \cite{mp:11}, so we omit it. We are finally going to complete the\\
\noindent
\textbf{Proof of Proposition \ref{prop:3.3}.}\\
The compactness of $\tilde{J}_{n,i}(s)$ follows from the continuity of all the functions involved and the inclusion $J_{n,i}(s) \subset \tilde{J}_{n,i}(s)$ follows from Lemmas \ref{lem:6.6}, \ref{lem:6.7} and \ref{lem:6.8}. \qed

\appendix
\section{Appendix}\label{appendix}

We give the proofs of the results from Section \ref{sec:4} and add some auxiliary results.  Remember that $C>0$ denotes a constant that may change its values from line to line. 
 We start with the proof of Lemma \ref{lem:algebra}.
\begin{proof}[Proof of Lemma \ref{lem:algebra}]
 Consider for $r>0, u\geq 1$  the function
 $$ f(a) = a  \exp \left(- \frac{a^r}{u} \right), \, a \geq 0,$$
 which attains its maximal value $u^{1/r} (\frac{1}{r})^{1/r}\exp(-1/r)$ at $a = (\tfrac{u}{r})^{1/r}.$ Hence, choose
 $C(r_0,r_1) =\max_{r \in [r_0,r_1]} (\frac{1}{r})^{1/r}\exp(-1/r)$ to obtain the result.
\end{proof}
This lemma can be applied in the next proof.
\begin{proof}[Proof of Lemma \ref{lem:4.2}]
By Lemma \ref{lem:algebra} applied with  $a=\frac{|x|}{2\sqrt{t}}, u=1, r=2,$ 
 \begin{align*} |p_{t,l}(x)| &\leq \frac{1}{\sqrt{t}} \frac{|x|}{\sqrt{t}} (2\pi t)^{-\dimension /2} \exp \left( - \frac{|x|^2}{2t} \right) \leq C \frac{1}{\sqrt{t}} (4\pi t)^{-\dimension /2} \exp \left( -\frac{|x|^2}{4t} \right),
 \end{align*}
which  proves the result.
\end{proof}
Next we can extend the results of Lemma 5.2 in \cite{mps:06} to derivatives:
\begin{lemma}\label{lem:heat-ker:diff:1}
 There is a uniform constant $C>0$  such that for any $0<t<t'$, $w, v \in \RR^{\dimension}$
 the following holds for $l=1,\dots,\dimension:$
 \begin{enumerate}
  \item Setting $\hat{v}_0:= 0$ and $\hat{v}_i:= \hat{v}_{i-1} + v_i e_i, 1\leq i \leq \dimension,$ where $e_i$ is the $i$-th unit vectors in $\RR^{\dimension},$ we have
  for the spatial differences
   \begin{equation}
         \begin{split}
          |p_{t,l}(w+v) - p_{t,l}(w)| & \leq  C t^{-1} \sum_{i=1}^{\dimension}  \int_0^{|v_i|} dr_i\,  p_{2t}(w+\hat{v}_{i-1} + r_i e_i).
         \end{split}
        \end{equation}        
  \item We obtain for the time differences
  \begin{equation}
   | p_{t,l} (w) - p_{t',l}(w)| \leq C  |t-t'|^{\frac{1}{2}} t^{-\frac{1}{2}} (t^{-1/2} p_{2t}(w) + t'^{-1/2} p_{4t'}(w)).
   \end{equation}
 \end{enumerate}
\end{lemma}
\begin{proof} We follow \cite[page 1932]{mps:06}. Without loss of generality we can assume that $l=1.$ Then we consider for (a):
\begin{equation*}
 \begin{split}
  | \frac{w_1}{t} & \exp ( - \frac{|w|^2}{2t} ) - \frac{w_1+ v_1}{t}  \exp (- \frac{|w+v|^2}{2t} ) | \\
  & \leq |\frac{w_1}{t} \exp ( - \frac{|w|^2}{2t} ) - \frac{w_1+ v_1}{t}  \exp (- \frac{|w+\hat{v}_1|^2}{2t} ) | + \\
  & \quad + \sum_{i=2}^{\dimension} | \frac{w_1+ v_1}{t}| | \exp (- \frac{|w+\hat{v}_{i-1}|^2}{2t} ) - \exp (- \frac{|w+\hat{v}_{i}|^2}{2t} ) |
  \end{split}
\end{equation*}
Now, observe that \[\partial_{x_1} \left( x_1/t \exp(-|x|^2/(2t))\right) = t^{-1} \exp(-|x|^2/(2t)) - (x_1/t)^2 \exp(-|x|^2/(2t))\] and 
$\partial_{x_1}  \exp(-|x|^2/(2t))=  - (x_1/t) \exp(-|x|^2/(2t)).$  Hence, the above is bounded by
\begin{align*}
  &  |\int_0^{|v_1|} dr_1 \, [t^{-1} \exp(- \frac{|w+r_1 e_1|^2}{2t} ) - (\frac{w_1+r_1}{t})^2 \exp(- \frac{|w+r_1 e_1|^2}{2t} ) ] |\\
  & \quad + \frac{|w_1+ v_1|}{t} \sum_{i=2}^{\dimension} | \int_0^{|v_i|} dr_i\,  \frac{w_i+r_i}{t} \exp(- \frac{|w+\hat{v}_{i-1} + r_i e_i|^2}{2t} ) |.
   \end{align*}
 Now, use \eqref{eq:ineq:1} twice with $a=\frac{w_i+r_i}{\sqrt{t}}, u=4$ and $r=1$ respectively $r=2$ to bound this further by
  \begin{align*}
  &  t^{-1} \int_0^{|v_1|} dr_1\, \left( \exp(- \frac{|w+r_1 e_1|^2}{2t} ) + C \exp(- \frac{|w+r_1 e_1|^2 }{2t} + \frac{|w_1+r_1|^2}{4t} )  \right) \\
  &\quad + C  \sum_{i=2}^{\dimension}  \int_0^{|v_i|} dr_i\,  \exp(- \frac{|w+\hat{v}_{i-1} + r_i e_i|^2}{2t} + \frac{|w_i+r_i|^2}{4t} + \frac{|w_1+ v_1|^2}{4t}) \\
   & \leq C t^{-1} \int_0^{|v_1|} dr_1\, \exp(- \frac{|w+r_1 e_1|^2}{4t} )  +  C  t^{-1} \sum_{i=2}^{\dimension}  \int_0^{|v_i|} dr_i\,  \exp(- \frac{|w+\hat{v}_{i-1} + r_i e_i|^2}{4t} ).
 \end{align*}
And the result follows by multiplication with $(2\pi t)^{-\dimension/2}$.

To prove (b) we consider the time differences, following (52) in \cite{mps:06}. First, rewriting and then using the Mean Value Theorem we get
  \begin{align*} 
  | p_{t,1} (w) - p_{t',1}(w)| & = (2\pi)^{-\dimension /2} | \frac{w_1}{t} t^{-\dimension /2} \exp (-\frac{|w|^2}{2t}) -  \frac{w_1}{t'} t'^{-\dimension /2} \exp (-\frac{|w|^2}{2t'}) | \\
  & \leq (2\pi)^{-\dimension /2} |  (t^{1/2})^{-\dimension-2} -(t'^{1/2})^{-\dimension-2}|\, |w| \exp (-\frac{|w|^2}{2t}) \\
  & \quad +  (2\pi)^{-\dimension /2}  |w| t'^{-\dimension /2-1} |\exp (-\frac{|w|^2}{2t}) -   \exp (-\frac{|w|^2}{2t'}) | \\
  & \leq (2\pi)^{-\dimension/2} (\dimension+2) |w| |t^{1/2} - t'^{1/2}| (t^{1/2})^{-\dimension-3} \exp (-\frac{|w|^2}{2t})  \\
  & \quad + (2\pi)^{-\dimension /2} |w| t'^{-1-\dimension /2} \int_{t^{1/2}}^{t'^{1/2}} \exp (-\frac{|w|^2}{2s^2}) \frac{|w|^2}{s^3}\, ds.
 \end{align*}  
 Using $a\leq \exp (a)$ for $a = |w|^2/(4s^2)$, we have
$$  \int_{t^{1/2}}^{t'^{1/2}} \exp (-\frac{|w|^2}{2s}) \frac{|w|^2}{s^3}\, ds \leq  \int_{t^{1/2}}^{t'^{1/2}} \frac{4}{s} \exp (-\frac{|w|^2}{4s^2})\, ds \leq |t^{1/2} - t'^{1/2}| \frac{4}{t^{1/2}}\exp (-\frac{|w|^2}{4t'}).$$
Using further \eqref{eq:ineq:1} in both lines of the above expression,  we can bound it by
  \begin{align*} 
  | p_{t,1} (w) - p_{t',1}(w)| & \leq  (2\pi)^{-\dimension /2} C t^{-1-\dimension /2} |t^{1/2} - t'^{1/2}| \exp (-\frac{|w|^2}{4t})  \\
  & \quad 
 +  (2\pi)^{-\dimension /2} C t'^{-1/2 - \dimension /2}   |t^{1/2} - t'^{1/2}| t^{-1/2} \exp (-\frac{|w|^2}{8t'}) \\
  & \leq C  |t^{1/2} - t'^{1/2}| t^{-1/2} (t^{-1/2} p_{2t}(w) + t'^{-1/2} p_{4t'}(w)).
  \end{align*}
\end{proof}
Next, combine that lemma with Lemma 5.1 in \cite{mps:06}:
\begin{proof}[Proof of Lemma \ref{lem:heat-ker:diff:2}]
 There are two estimates to make, one for each part of the $\wedge$.\\
 First, let us consider the left part. Expanding the product in the integral gives
\begin{align}
\nn   \intRd \intRd &| \ (p_{t,l}(w-x) - p_{t',l}(w-x'))(p_{t,l}(z-x) - p_{t',l}(z-x')) | \ (|w-z|^{-\alpha} + 1) \, dw  dz \\
\nn    \leq  & \intRd \intRd | p_{t,l}(w-x) p_{t,l}(z-x) | (|w-z|^{-\alpha} + 1) \, dw  dz \\
\nn    & + \intRd \intRd | p_{t',l}(w-x') p_{t',l}(z-x') | (|w-z|^{-\alpha} + 1) \, dw  dz \\
\label{fourptl}  & + \intRd \intRd | p_{t,l}(w-x) p_{t',l}(z-x') | (|w-z|^{-\alpha} + 1) \, dw  dz  \\
\nn    & + \intRd \intRd | p_{t',l}(w-x') p_{t,l}(z-x) | (|w-z|^{-\alpha} + 1) \, dw  dz.
\end{align}
Note that by a change of variables (and $|w| = |-w|$) the last two lines coincide. The same is true for the first two lines except that $t$ and $t'$ differ. Thus,
expression (\ref{fourptl}) is equal to
\begin{equation}
\label{threeptl}
 \begin{split}
  & \intRd \intRd  | p_{t,l}(w) p_{t,l}(z) | (|w-z|^{-\alpha} + 1) \, dw  dz \\
  &+   \intRd \intRd  | p_{t',l}(w) p_{t',l}(z) | (|w-z|^{-\alpha} + 1) \, dw  dz \\
    & + 2 \intRd \intRd | p_{t,l}(w-(x-x')) \; p_{t',l}(z) | (|w-z|^{-\alpha} + 1) \, dw  dz. 
 \end{split}
\end{equation}
For the first line of (\ref{threeptl}) we write, using $|w_l| \leq |w|$ 
 and \eqref{eq:ineq:1},
 \begin{align*}
   & t^{-1} \intRd \intRd (2 \pi t)^{-\dimension} \frac{|w_l|}{\sqrt{t}} \exp(- \frac{|w|^2}{2t})   \frac{|z_l|}{\sqrt{t}} \exp(- \frac{|z|^2}{2t})   (|w-z|^{-\alpha} + 1) \, dw  dz \\
   & \leq  C t^{-1}  \intRd \intRd (2\pi t)^{-\dimension} \exp(- \frac{|w|^2}{4t})    \exp(- \frac{|z|^2}{4t})   (|w-z|^{-\alpha} + 1) \, dw  dz,  \\
   \leq & C(\alpha, \dimension) (t^{-\alpha/2 -1} + t^{-1})
  \end{align*}
   by an application of Lemma 5.1 in \cite{mps:06} and the fact that $t\leq t'$. For the second line (with $t'$) we can do exactly the same and obtain the same even with $t$ instead of $t'$, since $t\leq t'$.

For the third line of (\ref{threeptl}) the same reasoning leads to the bound
  \begin{align*}
   & 2 (tt')^{-1/2} \intRd \intRd (2 \pi)^{-\dimension} (tt')^{-\dimension /2} \frac{|(w-(x-x'))_l|}{\sqrt{t}} \exp(- \frac{|w-(x-x')|^2}{2t})   \\
   & \phantom{ (tt')^{-1/2} \intRd \intRd (2 \pi)^{-\dimension} (tt')^{-\dimension /2}AAAAA}   \frac{|z_l|}{\sqrt{t'}}\exp(- \frac{|z|^2}{2t'})   (|w-z|^{-\alpha} + 1) \, dw  dz \\
   \leq &  C t^{-1} \intRd \intRd (2\pi t)^{-\dimension} \exp(- \frac{|w-(x-x')|^2}{4t})    \exp(- \frac{|z|^2}{4t'})   (|w-z|^{-\alpha} + 1) \, dw  dz,  \\
   & \leq C(\alpha, \dimension) (t^{-\alpha/2 -1} + t^{-1})
  \end{align*}
 by an application of Lemma 5.1 in \cite{mps:06} and $t\leq t' \leq K$.
 \smallskip
 
So this was the first part of the $\wedge$. To consider the second estimate, we start with a decomposition
\begin{equation}\label{eq:heat-ker:decomp}
 \begin{split}
  | \ (p_{t,l}(w-x) & - p_{t',l}(w-x'))(p_{t,l}(z-x) - p_{t',l}(z-x')) \ | \\
  & \leq \ | \ (p_{t,l}(w-x) - p_{t,l}(w-x'))(p_{t,l}(z-x) - p_{t,l}(z-x')) \ | \\
  & \quad + | \ (p_{t,l}(w-x) - p_{t,l}(w-x'))(p_{t,l}(z-x') - p_{t',l}(z-x')) \ | \\
  & \quad + | \ (p_{t,l}(w-x') - p_{t',l}(w-x'))(p_{t,l}(z-x) - p_{t,l}(z-x')) \ | \\
  & \quad + | \ (p_{t,l}(w-x') - p_{t',l}(w-x'))(p_{t,l}(z-x') - p_{t',l}(z-x')) \ |.
 \end{split}
\end{equation}
We start with the simplest case in \eqref{eq:heat-ker:decomp}:
  \begin{align*}
  & \intRd \intRd |  (p_{t,l}(w-x) - p_{t,l}(w-x'))(p_{t,l}(z-x) - p_{t,l}(z-x'))   | (|w-z|^{-\alpha} + 1) \, dw  dz.
   \end{align*} 
Changing variables, setting  $v = x-x'$ and using Lemma \ref{lem:heat-ker:diff:1}, we bound this by
     \begin{align*}
   &C \intRd \intRd  \Bigl[   t^{-1} \sum_{i=1}^{\dimension}  \int_0^{|v_i|} dr_i \,  p_{2t}(w+\hat{v}_{i-1} + r_i e_i) \Bigr] 
    \Bigl[  t^{-1} \sum_{j=1}^{\dimension}  \int_0^{|v_j|} d\tilde{r}_j \  p_{2t}(z+\hat{v}_{j-1} + \tilde{r}_j e_j) \Bigr] \\
  & \qquad \qquad (|w-z|^{-\alpha} + 1) \, dw  dz\\
     & =  C t^{-2} \sum_{i=1}^{\dimension}  \int_0^{|v_i|} dr_i\, \sum_{j=1}^{\dimension}  \int_0^{|v_j|} d\tilde{r}_j \\
   & \qquad \qquad \intRd \intRd p_{2t}(w+\hat{v}_{i-1} + r_i e_i) p_{2t}(z+\hat{v}_{j-1} + \tilde{r}_j e_j) (|w-z|^{-\alpha} + 1) \, dw  dz \\
    & \leq C t^{-2}(t^{-\alpha/2} +1) \max_{i,j} |v_i v_j| \leq C t^{-2}(t^{-\alpha/2} +1) |v|_2^2 = C t^{-2}(t^{-\alpha/2} +1) |x-x'|_2^2,
 \end{align*}
 using Lemma 5.1 (a) of \cite{mps:06} in the last step (compare this with Lemma 5.2 (b) in \cite{mps:06}).

Now, we consider the temporal distances in \eqref{eq:heat-ker:decomp}, i.e.~the last line. There we get by Lemma \ref{lem:heat-ker:diff:1} and Lemma 5.1 (a) of \cite{mps:06} that
  \begin{align*}
  & \intRd \intRd |  (p_{t,l}(w-x') - p_{t',l}(w-x'))(p_{t,l}(z-x') - p_{t',l}(z-x'))   | (|w-z|^{-\alpha} + 1) \, dw  dz \\
  & \leq c|t-t'| t^{-2} (t^{-\alpha/2}+1),
  \end{align*}
and this is the next part of the proposition - similar to Lemma 5.3 in \cite{mps:06}. The mixed parts in \eqref{eq:heat-ker:decomp} can be done similarly.
\end{proof}
Next we give the proof of a technical lemma:
\begin{lemma}\label{lem:heat-ker:deriv:1}
 For $R > 0$ there is a constant $C=C(R)$ such that for any  $y, \tilde{y} \in \RR^{\dimension}$,  $0<t\leq t'$ and $\eta_0 \in(1/R,1/2)$ the following holds for $l=1,\dots,\dimension$:
 \begin{enumerate}
  \item $\1 \{ |\tilde{y}| > t'^{1/2-\eta_0} \vee 2 |y-\tilde{y}| \} \, | p_{t,l}(y) |   \leq C  \exp (-\tfrac{1}{64} t^{-2\eta_0}) p_{4t}(y).$
  \item $\1 \{ |\tilde{y}| > t'^{1/2-\eta_0} \vee 2 |y-\tilde{y}| \} \, | p_{t,l}(y) |   \leq 2^{\dimension} C  \exp (-\tfrac{1}{64} t^{-2\eta_0}) p_{16t}(\tilde{y}).$
 \end{enumerate}
\end{lemma}

\begin{proof}
Let us write $A:= \{ |\tilde{y}| > t'^{1/2-\eta_0} \vee 2 |y-\tilde{y}| \}$, then on that event it holds that
$$ |y| \geq |\tilde{y}| - |y-\tilde{y}| > \tfrac{|\tilde{y}|}{2} > \tfrac{t'^{1/2-\eta_0}}{2} \geq \tfrac{t^{1/2-\eta_0}}{2} \text{ , thus } \frac{|y|^2}{t}\geq \frac{t^{-2\eta_0}}{4}. $$
Using this and (\ref{eq:ineq:1}) twice, we calculate 
  \begin{align*}
 \1_A | p_{t,l}(y) | & = \1_A \frac{|y_l|}{t} (2\pi t)^{-\dimension /2} \exp(- \frac{|y|^2}{2t} ) \\
  & \leq \1_A \frac{|y|}{t} (2\pi t)^{-\dimension /2} \exp(- \frac{|y|^2}{2t} ) \\
  & \leq \1_A C t^{-1/2} (2\pi t)^{-\dimension /2} \exp(- \frac{|y|^2}{4t} )  \\
  & = \1_A C   t^{-1/2} \exp(- \frac{|y|^2}{8t} ) (2\pi t)^{-\dimension /2} \exp(- \frac{|y|^2}{8t} )   \\
  & \leq \1_A C  t^{-1/2} \exp (-\tfrac{1}{32} t^{-2\eta_0}) (2\pi t)^{-\dimension /2} \exp( - \frac{|y|^2}{8t} ) \\
  &\leq \1_A C(R)  \exp (-\tfrac{1}{64} t^{-2\eta_0}) p_{4t}(y). 
   \end{align*} 
Given that on the set $A$ we have
$$ |\tilde{y}| < 2 |y| \text{ , thus } |y|^2 \geq \frac{|\tilde{y}|^2}{4},$$
we can bound this further by
\begin{equation*}
\1_A | p_{t,l}(y) | \leq C(R)  \exp (-\tfrac{1}{64} t^{-2\eta_0}) p_{16t}(\tilde{y}).
\end{equation*}
\end{proof}
In order to prepare Lemma \ref{lem:4.4} we give the following proof.
\begin{proof}[Proof of Lemma \ref{lem:mps:5.1:add}]
 By \eqref{eq:ineq:1}  we bound
 $$ |w|^{r_1} p_t(w) \leq 4^{r_1/2+\dimension /2} t^{r_1/2}  p_{2t}(w)
 \quad \text{ and } \quad  |z|^{r_2} p_{t'}(z) \leq  4^{r_2/2+\dimension /2} t'^{r_2/2}  p_{2t'}(z).$$
 Next apply Lemma 5.1 (b) of \cite{mps:06} if $r_3 >0$ and their Lemma 5.1 (a) if $r_3 = 0$, to get the first estimate.
 For the second estimate note that by \eqref{eq:ineq:1}  and $|x| \leq \sqrt{\dimension} K,$
 \begin{align*}
  p_t(x-w) |w|^{r_1} & \leq p_t(x-w) 2^{r_1} (|w-x|^{r_1} + |x|^{r_1}) \\
   & \leq  2^R  (4t)^{r_1/2} p_{2t}(x-w) + (2\sqrt{\dimension}K)^{r_1} p_t(x-w) \\
   & \leq C(K,R) p_{2t}(x-w) (t^{r_1/2} +1)
 \end{align*}
so that we obtain the result by the first part.
\end{proof}
Finally, we can conclude the appendix with the proof of Lemma \ref{lem:4.4}
\begin{proof}[Proof of Lemma \ref{lem:4.4}]
By H\"older's Inequality we can bound the left hand side in \eqref{eq:lem:4.4} by
  \begin{align*}
  \big[ \intRd &\intRd   | ( p_{t-s,l}(w-x) - p_{t'-s,l}(w-x') ) ( p_{t-s,l}(z-x) - p_{t'-s,l}(z-x') )  |  \\
  &\quad   \ (|w-z|^{-\alpha} + 1) \, dw  dz \big]^{1-\eta_1/2} \\
  & \times
  \big[ \intRd \intRd  | ( p_{t-s,l}(w-x) - p_{t'-s,l}(w-x') ) ( p_{t-s,l}(z-x) - p_{t'-s,l}(z-x') ) | \\
   &  \quad \quad   |w-x|^{2p/ \eta_1} |z-x|^{2p/ \eta_1} \, \1 \{ |w-x| > (t'-s)^{1/2-\eta_0}\vee 2 |x-x'| \} \\
   &  \quad \quad  e^{2r/\eta_1(|w-x| + |z-x|)} (|w-z|^{-\alpha} + 1) \, dw  dz \big] ^{\eta_1/2}.
   \end{align*}
Now estimate the first integral using Lemma \ref{lem:heat-ker:diff:2} and expand the second one to obtain as a bound for \eqref{eq:lem:4.4}
  \begin{equation}\label{eq:4.4:pr:1}
    \begin{split}
    & C(R) (t-s)^{-(1+\alpha/2)(1-\eta_1/2)} \left[1\wedge \left( \frac{|x-x'|^2+|t-t'|}{t-s}   \right) \right]^{1-\eta_1/2} \\
    &  \times \left[ \intRd \intRd |p_{t-s,l}(w-x)p_{t-s,l}(z-x)| L(x,x',w,z,s,t') \, dw  dz \right. \\
    &\quad  + \intRd \intRd |p_{t'-s,l}(w-x')p_{t'-s,l}(z-x') |L(x,x',w,z,s,t')\, dw  dz \\
    & \quad + \intRd \intRd |p_{t-s,l}(w-x)p_{t'-s,l}(z-x') | L(x,x',w,z,s,t')\, dw  dz \\
    & \quad + \left. \intRd \intRd | p_{t'-s,l}(w-x')p_{t-s,l}(z-x) | L(x,x',w,z,s,t')\, dw  dz \right]^{\eta_1/2} ,
  \end{split} \end{equation}
where
\begin{eqnarray*}
L(x,x',w,z,s,t') &:=&  |w-x|^{2p/ \eta_1} |z-x|^{2p/ \eta_1} \, \1 \{ |w-x| > (t'-s)^{1/2-\eta_0}\vee 2 |x-x'| \} \\
& &e^{2r/\eta_1(|w-x| + |z-x|)} (|w-z|^{-\alpha} + 1).
\end{eqnarray*}
Since all of the four summands in the end are similar, we only consider the last one which is the worst with respect to  $(t-s)$-asymptotics. 
Use Lemma \ref{lem:4.2}, replace $\tilde{w} = w-x, \tilde{z}=z-x$ and then use Lemma \ref{lem:heat-ker:deriv:1} (b) to obtain
\begin{align*}
  & \intRd  \intRd  | p_{t'-s,l}(w-x')p_{t-s,l}(z-x)| \ |w-x|^{2p/\eta_1} |z-x|^{2p/\eta_1} \\
   &  \times \1 \{ |w-x| > (t'-s)^{1/2-\eta_0}\vee 2 |x-x'| \}  \, e^{2r/\eta_1(|w-x| + |z-x|)}  (|w-z|^{-\alpha} + 1) \, dw  dz \\
   & \leq C \intRd \intRd  | p_{t'-s,l}(\tilde{w}+x-x') | (t-s)^{-1/2}p_{2(t-s)}(\tilde{z} ) |\tilde{w}|^{2p/\eta_1} |\tilde{z}|^{2p/\eta_1} \\
  &  \quad \quad \times \1 \{ |\tilde{w}| > (t'-s)^{1/2-\eta_0}\vee 2 |x-x'| \}  e^{2r/\eta_1(|\tilde{w}| + |\tilde{z}|)} (|\tilde{w} - \tilde{z}|^{-\alpha} + 1) \, d\tilde{w} \, d\tilde{z}\\
   & \leq C(R) \intRd \intRd \exp (-\tfrac{1}{64} (t'-s)^{-2\eta_0}) p_{16(t'-s)}(\tilde{w})p_{2(t-s)}(\tilde{z}) \\ 
  &  \quad \quad \times |\tilde{w}|^{2p/\eta_1} |\tilde{z}|^{2p/\eta_1} (t-s)^{-1/2}  e^{2r/\eta_1(|\tilde{w}| + |\tilde{z}|)} (|\tilde{w}-\tilde{z}|^{-\alpha} + 1) \, d\tilde{w} \, d\tilde{z}. \\
  & \leq C(R,\eta_1,K)  (t-s)^{-1/2} \exp (-\tfrac{1}{64} (t'-s)^{-2\eta_0}) \nonumber \\
  & \qquad \qquad \qquad  e^{32r^2 \eta_1^{-2}(t'-s)} (t'-s)^{p\eta_1^{-1}} (t-s)^{p\eta_1^{-1}}  ((t'-s)^{-\frac{\alpha}{2}} + 1). \nonumber \\
  & \leq C(R,K) (t-s)^{-1/2} \exp(-\tfrac{1}{128} (t'-s)^{-2\eta_0} ), 
  \end{align*}
where we used Lemma \ref{lem:mps:5.1:add}, first part, in the next to last line and $(t'-s)\leq K$. The other summands are similar, we use Lemma \ref{lem:heat-ker:deriv:1} (with $t=t'$ for lines 1 and 3) and can use the exponential of $t'-s$ ($t-s$ in lines 1 and 3) to control all of the negative exponents.
Putting this back in \eqref{eq:4.4:pr:1} gives the result since $(1+ \alpha/2)(1- \eta_1/2) + (1/2) (\eta_1/2)  \leq 1 + \alpha/2.$
\end{proof}

\noindent
{\bf Acknowledgement.} T.R.~would like to thank Leonid Mytnik for valuable discussions and the invitation to spend time at the Technion.

\bibliographystyle{alpha}
\bibliography{2012-12-17-PUforSHE_article}

\end{document}